\documentclass[12pt,a4paper]{amsart}

\usepackage[english]{babel}
\usepackage{amsmath}
\usepackage{amssymb}
\usepackage{amsfonts}
\usepackage{mathrsfs}
\usepackage{amsthm}
\usepackage{graphicx}
\usepackage{framed}
\usepackage{mdframed}
\usepackage[colorinlistoftodos]{todonotes}
\usepackage{cite}

\usepackage[normalem]{ulem} 

\usepackage[all]{xy}
\usepackage{float}
\usepackage{tikz-cd}

\newtheorem{theorem}{Theorem}[section]
\newtheorem*{theorem*}{Theorem}
\newtheorem{lemma}[theorem]{Lemma}
\newtheorem{proposition}[theorem]{Proposition}

\newtheorem{corollary}[theorem]{Corollary}

\theoremstyle{definition}
\newtheorem{definition}[theorem]{Definition}
\newtheorem{example}[theorem]{Example}

\newtheorem{remark}[theorem]{Remark}

\usepackage{dsfont}
\usepackage{stmaryrd}
\usepackage[scr=boondox]{mathalfa}

\DeclareMathOperator{\ad}{ad}
\DeclareMathOperator{\Ad}{Ad}

\DeclareMathOperator{\Aut}{Aut}

\DeclareMathOperator{\Aff}{Aff}
\DeclareMathOperator{\Orth}{O}
\DeclareMathOperator{\SOrth}{SO}
\DeclareMathOperator{\GLin}{GL}
\DeclareMathOperator{\SLin}{SL}
\DeclareMathOperator{\Sp}{Sp}
\DeclareMathOperator{\SU}{SU}
\DeclareMathOperator{\Spin}{Spin}

\DeclareMathOperator{\PGL}{PGL}
\DeclareMathOperator{\POrth}{PO}

\newcommand{\MC}[1]{\omega_{{}_{#1}}}
\newcommand{\Lt}[1]{\operatorname{L}_{#1}}

\newcommand{\quot}[1]{\mathrm{q}_{{}_{#1}}}

\usepackage[all]{xy}

\usepackage[normalem]{ulem}

\DeclareFontEncoding{T3}{}{}
\DeclareFontSubstitution{T3}{cmr}{m}{n}
\DeclareSymbolFont{tipa}{T3}{cmr}{m}{sl}
\SetSymbolFont{tipa}{bold}{T3}{cmr}{bx}{sl}
\DeclareMathSymbol{\kgf}{\mathord}{tipa}{'255}

\usepackage[T2A,OT1]{fontenc}
\makeatletter
\newcommand{\Zhe}{\mathord{\mathchoice{\mbox{\fontsize\tf@size\z@\usefont{T2A}{\rmdefault}{m}{n}{\CYRZH}}}{\mbox{\fontsize\tf@size\z@\usefont{T2A}{\rmdefault}{m}{n}{\CYRZH}}}{\mbox{\fontsize\sf@size\z@\usefont{T2A}{\rmdefault}{m}{n}{\CYRZH}}}{\mbox{\fontsize\ssf@size\z@\usefont{T2A}{\rmdefault}{m}{n}{\CYRZH}}}}}
\makeatother

\newenvironment{smallbmatrix}
  {\left[\begin{smallmatrix}}
  {\end{smallmatrix}\right]}

\newcommand{\rotoiso}{\begin{smallbmatrix}0 & -1 \\ 1 & 0\end{smallbmatrix}}

\usepackage{hyperref}

\title[A kinematic explanation of the 1:3 ratio]{A kinematic explanation of the 1:3 ratio for rolling spheres and the exceptional simple Lie group of rank two}
\author{Jacob W. Erickson}

\begin{document}
\begin{abstract}Given a pair of (round) 2-dimensional spheres, one of which has radius three times that of the other, the Lie algebra of local infinitesimal symmetries for the distribution corresponding to rolling the spheres along each other (without letting them slip or twist) happens to be isomorphic to the split real form of the exceptional simple Lie algebra of rank 2. These exceptional local symmetries appear only for this specific 1:3 ratio of radii, however, and while there are several proofs of this result, a straightforward kinematic explanation for the seemingly miraculous appearance of an exceptional simple Lie group in this situation has long been desired. In this paper, we provide such an explanation, relating the ratio of radii to the intersections of a pair of one-parameter subgroups that can be seen from the rolling sphere perspective. The approach does not require the split-octonions, as we construct the exceptional simple Lie group directly by Tanaka prolongation to aid in visualizing the underlying geometry.\end{abstract}

\maketitle

\section{Introduction}
When first encountering the exceptional simple Lie groups, they can often seem unapproachable and mysterious. By their very nature, they tend to fall outside of the patterns with which we have more experience, and even the smallest among them is $14$-dimensional. Hoping to find a nice, visual way of thinking about one of them might, therefore, seem overly optimistic. This makes the following result, attributed to Robert Bryant in \cite{BorMontgomery2009}, all the more surprising.

\begin{theorem*}
Consider a pair of 2-spheres $\mathbb{S}^2_{r_0}$ and $\mathbb{S}^2_{r_1}$ of radii $r_0$ and $r_1$, respectively. Let $M$ be the $5$-dimensional space of relative configurations of the two spheres such that they have tangential contact at a point, and let $D_{r_0/r_1}$ be the rank 2 distribution on $M$ whose elements are the velocities corresponding to rolling the spheres along each other without letting them slip or twist.

\hspace{-0.4em}When $\tfrac{r_0}{r_1}\in\{3,\tfrac{1}{3}\}$, the Lie algebra of local infinitesimal symmetries for the distribution $D_{r_0/r_1}$ at a point of $M$ is isomorphic to the split real form of the exceptional simple Lie algebra of rank 2. For all other ratios of radii with $r_0\neq r_1$, the Lie algebra of local infinitesimal symmetries is isomorphic to $\mathfrak{so}(3)\oplus\mathfrak{so}(3)$.
\end{theorem*}

This gives us a relatively manageable way of seeing what $\Zhe_2$, the smallest non-compact exceptional simple Lie group---also called $\mathrm{G}_{2(2)}$ in \cite{LeistnerNurowski2012} and $\mathrm{G}_2'$ in \cite{Evans2025}, among other names in other places---``looks like'', since we get a local picture that just involves rolling a pair of $2$-spheres along each other. Indeed, we provide somewhat detailed methods for visualizing $\Zhe_2$ based on this idea in Chapter 3 of \cite{EricksonThesis}.

But why is it that the ratio of the radii \textit{must} be 1:3 or 3:1 for the exceptional symmetries to show up? Given how miraculous this result appears, it is perhaps not too surprising that there are several (partial) proofs of it in the literature, using a plethora of different perspectives.\linebreak In \cite{BorMontgomery2009}, for example, Bor and Montgomery heroically dredge the 1:3 ratio from a careful, meticulous investigation of structure constants and bracket considerations in the exceptional simple Lie algebra of rank $2$. Agrachev indicated a slightly more geometric approach in \cite{Agrachev2007}, later echoed in \cite{BaezHuerta2014}, by using the exceptional simple Lie group $\Zhe_2$ in its guise as the automorphism group of the split-octonions to relate certain curves to projectivizations of $2$-dimensional null subalgebras, with the 1:3 ratio being the only one that makes this relationship work. Earlier work of Zelenko in \cite{Zelenko2006} singled out the 1:3 ratio without any reference to the exceptional simple Lie groups or the split-octonions at all, instead focusing on the vanishing of a particular curvature invariant related to variational methods and Cartan's infamous ``five variables'' paper \cite{Cartan1910}. Later, interested in the more general (and still active) question of when pairs of Riemannian surfaces can admit rolling distributions with such exceptional symmetries, An and Nurowski recovered the 1:3 ratio in \cite{AnNurowski2014} using ideas from conformal pseudo-Riemannian geometry, also drawing inspiration from Cartan's work in \cite{Cartan1910}. Even more recently, a proof was given by The in \cite{The2022}, via a computationally efficient application of the modern Cartan-geometric machinery for $(2,3,5)$-distributions.

While these previous approaches all provide valid and worthwhile justifications for why the 1:3 ratio appears here, they also take us fairly far afield from the kind of picture that we would instinctively want. Would it not be more satisfying to be able to point to some geometric or kinematic facet of the rolling spheres themselves to explain where this ratio comes from, rather than divining it from the structure of a modification of the octonions with split signature norm? In \cite{BorMontgomery2009}, Bor and Montgomery pose the finding of such an interpretation of the ratio as an open problem, and in this paper, we will share an answer to this problem, following Chapter 2 of the author's PhD dissertation \cite{EricksonThesis}.

Let us outline our visual interpretation of the 1:3 ratio here, without any further preamble. Consider a pair of great circles, one on each of the two spheres $\mathbb{S}^2_{r_0}$ and $\mathbb{S}^2_{r_1}$, configured so that they are tangent to each other. There are two natural one-parameter subgroups of symmetries for the distribution that move the spheres in such a way that these great circles remain tangent to each other: one that moves the spheres by rolling them together along the great circles without letting them slip or twist, and the other sliding them along the great circles in such a way that, if we imagine the smaller sphere as an eyeball and the larger sphere as being fixed in place, then the eyeball is always looking toward the same point on the horizon. We have illustrated what these look like in Figure \ref{benotafraid}. Directly from the root diagram---see Lemma \ref{2intersections}---we can see that the corresponding one-parameter subgroups of $\Zhe_2$, if they exist, must intersect exactly twice. This amounts to the orbits of these one-parameter subgroups within the relative configuration space of the rolling spheres intersecting exactly four times, which (as noted in \cite{BaezHuerta2014}) happens if and only if the ratio of the radii is 1:3 or 3:1.

\begin{figure}[h]
\centering\includegraphics[width=0.45\textwidth]{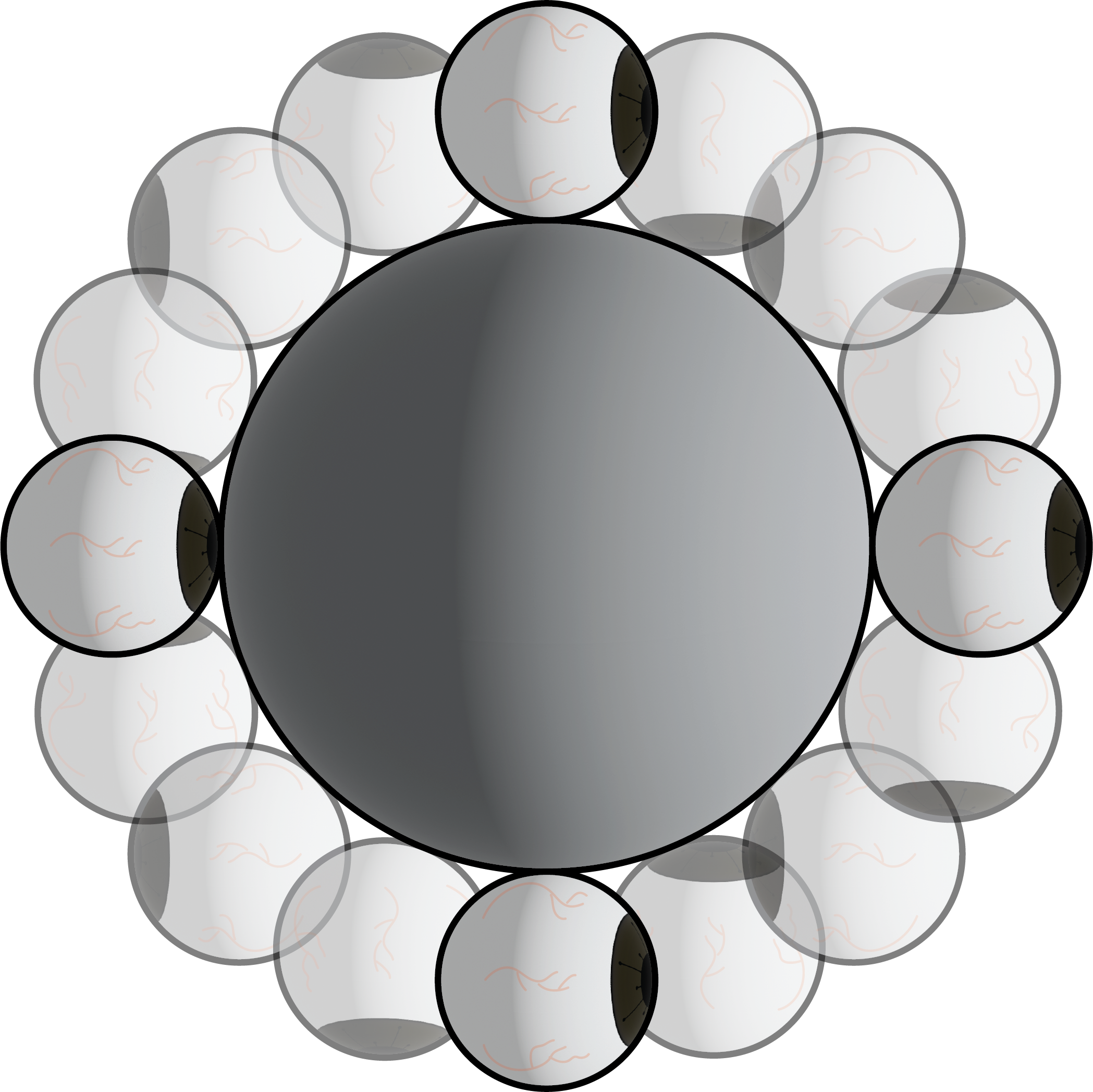}
\hfill\includegraphics[width=0.45\textwidth]{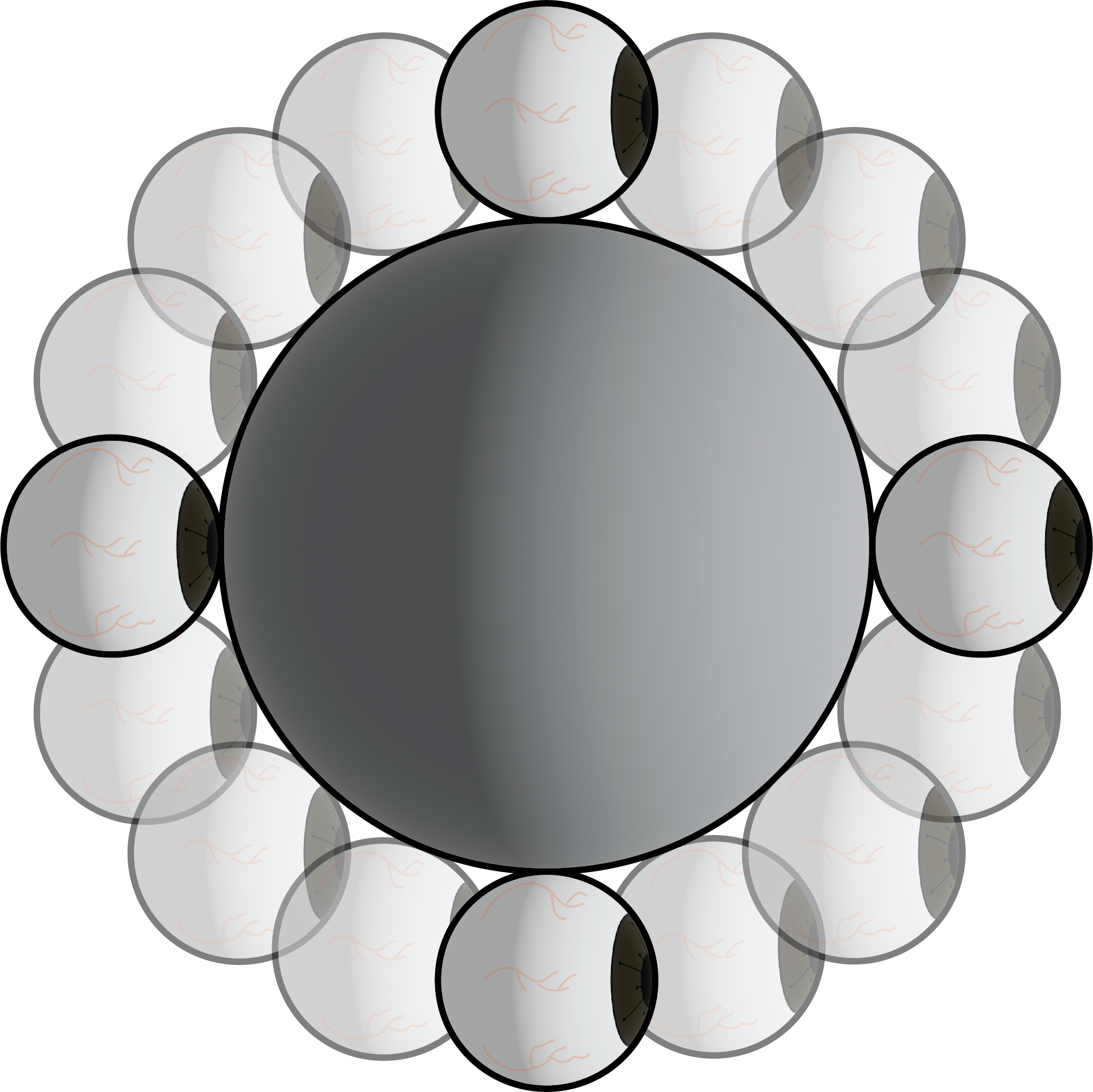}
\caption{Illustrations of the orbits for two particular one-parameter subgroups of symmetries for the rolling distribution, with the left one corresponding to rolling the spheres together along a pair of great circles and the right one corresponding to sliding them along those great circles such that the eyeball always looks in the same direction. The four points of intersection of the two orbits are indicated by greater opacity}
\label{benotafraid}
\end{figure}

The rest of the article is dedicated to filling in the details surrounding this interpretation of the 1:3 ratio. After reviewing some relevant tools from the modern Erlangen program in Section \ref{background}, we will construct $\Zhe_2$ geometrically in Section \ref{buildingZhe2} and then use the geometric structure that we glean from it to prove the main theorem above in Section \ref{proofsection}, including proving that the ratios of radii that are not 1:1, 1:3, or 3:1 admit no additional local infinitesimal symmetries beyond $\mathfrak{so}(3)\oplus\mathfrak{so}(3)$.

\section{Some background on model geometries}\label{background}
From the perspective of Felix Klein's Erlangen program, geometric structure amounts to whatever is preserved by the symmetries that we prescribe for the geometry. This point of view has a natural unifying spirit to it, as it brings all homogeneous geometries under the same conceptual framework, and conveniently, that conceptual framework reformulates a lot of geometry as Lie theory. The Erlangen viewpoint, for this reason, is especially well-suited for our current endeavor, where we seek to understand the relationship between a particular Lie group and this visibly geometric concept of rolling spheres.

Throughout, we follow the terminology and notation of \cite{EricksonThesis}, though much of this section can be found between \cite{Sharpe1997} and \cite{CapSlovakPG1}, which are the current standard references for the topic.

\subsection{Generalities for model geometries}
\begin{definition}A \emph{model geometry} (or just \emph{model}) is a pair $(G,H)$, where $G$ is a Lie group and $H\leq G$ is a closed subgroup such that $G/H$ is connected. For a model geometry $(G,H)$, the Lie group $G$ is called the \emph{symmetry group} or \emph{model group} of the geometry, while $H$ is called the \emph{isotropy} or \emph{stabilizer subgroup}.\end{definition}

For a model $(G,H)$, we always denote by $\quot{H}\!:G\to G/H,~ g\mapsto gH$ the natural quotient map from $G$ to $G/H$. This makes $G$ a principal $H$-bundle over $G/H$.

\begin{example}Euclidean geometry is encoded as a model geometry as $(\mathrm{Euc}(m),\Orth(m))$, where $\mathrm{Euc}(m)\simeq\mathbb{R}^m\rtimes\Orth(m)$. The symmetry group of $m$-dimensional Euclidean geometry is $\mathrm{Euc}(m)$, which acts by isometries on the connected homogeneous space $\mathrm{Euc}(m)/\Orth(m)\cong\mathbb{R}^m$, which we can think of as $m$-dimensional Euclidean space.\end{example}

\begin{example}In a similar vein, affine geometry corresponds to the model $(\mathrm{Aff}(m),\GLin_m\mathbb{R})$, where $\mathrm{Aff}(m)\simeq\mathbb{R}^m\rtimes\GLin_m\mathbb{R}$ is the Lie group of affine transformations on $\mathrm{Aff}(m)/\GLin_m\mathbb{R}\cong\mathbb{R}^m$, which we think of as $m$-dimensional affine space.\end{example}

\begin{example}We can encode $m$-dimensional spherical geometry via the model $(\Orth(m+1),\Orth(m))$, where $\Orth(m)$ is embedded in $\Orth(m+1)$ via $i:A\mapsto\left[\begin{smallmatrix}1 & 0 \\ 0 & A\end{smallmatrix}\right]$, so that $i(\Orth(m))$ is the stabilizer $\mathrm{Stab}_{\Orth(m+1)}(e_1)$ of the first basis vector $e_1\in\mathbb{R}^{m+1}$ under the usual action of $\Orth(m+1)$. For this model, the connected homogeneous space is $\Orth(m+1)/\Orth(m)\cong\mathbb{S}^m$, the $m$-dimensional sphere.\end{example}



Euclidean, affine, and spherical geometry are all examples of what are called \emph{reductive geometries}, meaning that, as an $\Ad_H$-representation, the Lie algebra $\mathfrak{g}$ of the model group $G$ decomposes as $\mathfrak{g}=\mathfrak{m}\oplus\mathfrak{h}$ for some $\Ad_H$-subrepresentation $\mathfrak{m}\subseteq\mathfrak{g}$. In the Euclidean case, we can write this $\Ad_{\Orth(m)}$-invariant subspace $\mathfrak{m}$ as the ideal corresponding to translations in the Lie algebra $\mathfrak{euc}(m)$, while for spherical geometry, the $\Ad_{\Orth(m)}$-invariant subspace $\mathfrak{m}$ is given by $\{\left[\begin{smallmatrix}0 & -v^\top \\ v & 0\end{smallmatrix}\right]:v\in\mathbb{R}^m\}$, which we can similarly think of as a space of infinitesimal translations, although it is not a subalgebra of the model Lie algebra $\mathfrak{o}(m+1)$.

For $(G,H)$ reductive, a curve in $G/H$ is called a \emph{geodesic} if and only if it is of the form $t\mapsto \quot{H}(g\exp(tX))$ for some $X\in\mathfrak{m}$ and $g\in G$. These are precisely the geodesics that we are used to seeing for these types of geometries: for $(v,0)\in\mathfrak{m}\subseteq\mathfrak{euc}(m)$, $\exp(t(v,0))=(tv,\mathds{1})$, so geodesics take the form \[t\mapsto \quot{\Orth(m)}((u,A)(tv,\mathds{1}))=\quot{\Orth(m)}(u+t(Av),A)\cong u+t(Av),\] which are just the straight lines through $u$ with constant velocity $Av$. Similarly, for spherical geometry, geodesics are left-translations of great circles through $\quot{\Orth(m)}(\mathds{1})\cong e_1\in\mathbb{S}^m$, parametrized with constant speed.



This description of geodesics can seem rather artificial when first encountering it, so it is worth taking a moment to dispel this faulty impression. Consider, when we cross the street, that we typically take a geodesic path, or at least something approximating one: we walk ``straight across'' rather than taking a more winding route. After some reflection, however, we should recognize that we are almost certainly not devoting the kind of energy required for our brains to genuinely contemplate the infinite-dimensional space of possible paths in order to find the one that minimizes distance. Instead, we are specifying our motion by, at each point in time along our trajectory, telling ourselves to do something along the lines of ``move forward (and do not turn)''. We are, in short, designating a constant direction for our (translational)\linebreak velocity while stipulating that our angular velocity is $0$. This is exactly what we are doing when specifying a geodesic within this Lie-theoretic framework: we pick a constant ``translational'' velocity in $\mathfrak{m}$ and move in that direction at each point in time, with the analogue of the angular velocity component in $\mathfrak{h}$ then required to be $0$.

We can extend this way of thinking using the Maurer--Cartan form, the principal invariant for model geometries.

\begin{definition}For a Lie group $G$, the \emph{Maurer--Cartan form} $\MC{G}$ is the left-invariant $\mathfrak{g}$-valued 1-form on $G$ defined by $\MC{G}(\xi_g):=\Lt{g^{-1}*}\xi_g\in \mathfrak{g}$ for $\xi_g\in T_gG$, where $\Lt{a}:g\mapsto ag$ denotes left-translation by $a\in G$.\end{definition}

The Maurer--Cartan form $\MC{G}$ lets us describe tangent vectors at any element of the Lie group $G$ in terms of corresponding velocities through the identity, allowing us to specify tangent directions in terms of a ``constant'' space of velocities as we did with geodesics above. Since the geometric structure of a model $(G,H)$ is defined to be preserved under left-translation by elements of $G$, this gives us a natural way of specifying invariants of the geometry.

Since $\MC{G}$ restricts to a linear isomorphism from $T_gG$ to $\mathfrak{g}=T_eG$ for each $g\in G$, we can consider its inverse as well: for $X\in\mathfrak{g}$, $\MC{G}^{-1}(X)$ is the corresponding left-invariant vector field. Since the bracket on the Lie algebra coincides with the bracket on the corresponding left-invariant vector fields, we will always have $[\MC{G}^{-1}(X),\MC{G}^{-1}(Y)]=\MC{G}^{-1}([X,Y])$ for $X,Y\in\mathfrak{g}$.

\subsection{Geometry of rolling spheres}\label{rollattach}\hfill\\
Now, let us concoct a model geometry $(G,H)$ for a pair of rolling spheres of dimension 2. We want our model's underlying homogeneous space $G/H$ to be (a connected component of) the $5$-dimensional space of relative configurations of two spheres $\mathbb{S}^2_{r_0}$ and $\mathbb{S}^2_{r_1}$ such that they are in tangential contact at a point, where $\mathbb{S}^2_r$ denotes the round sphere of radius $r$. More specifically, we want $G/H$ to be one of the two connected components $M$ of the space
\[\left\{(x_0,x_1,\phi):\begin{array}{l}x_0\in\mathbb{S}^2_{r_0}, x_1\in\mathbb{S}^2_{r_1},\text{ and } \phi \text{ a linear} \\ \text{isometry from } T_{x_0}\mathbb{S}^2_{r_0} \text{ to } T_{x_1}\mathbb{S}^2_{r_1}\end{array}\right\},\]
with $(x_0,x_1,\phi)$ being the point of the relative configuration space given by gluing the tangent space of $x_0\in\mathbb{S}^2_{r_0}$ to the tangent space of $x_1\in\mathbb{S}^2_{r_1}$ by $\phi$ (see Figure \ref{touchingspheres}). If we assign orientations to the two spheres, then the two connected components correspond to whether the linear map $\phi$ is orientation-preserving or reversing, which pictorially corresponds to whether one sphere is rolling along the inside or outside of the other sphere; it does not matter which connected component we choose, since applying an orientation-reversing isometry to one of the two spheres will swap between them.

\begin{figure}[h]
\centering\includegraphics[width=0.55\textwidth]{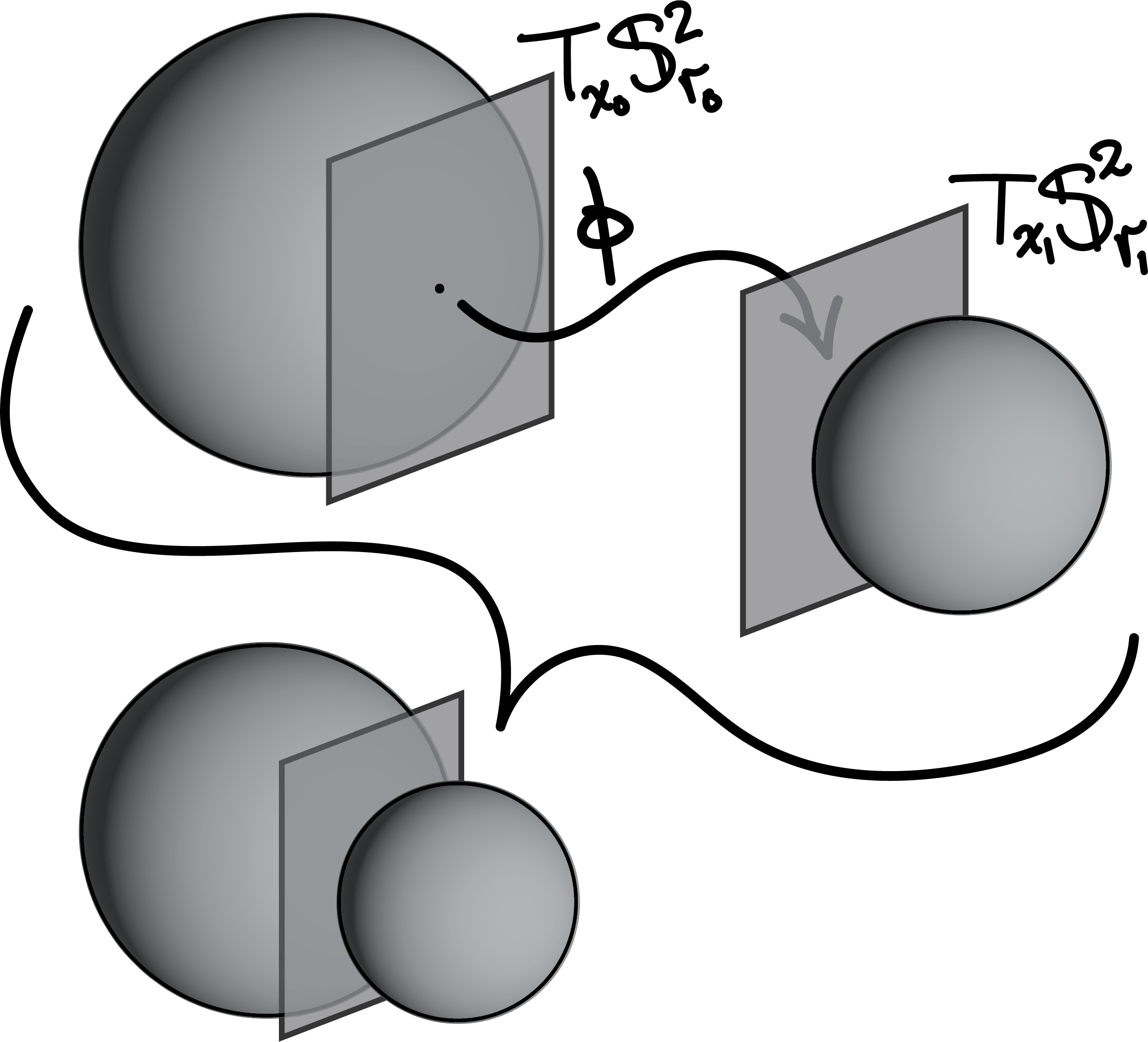}
\caption{An illustration of a point $(x_0,x_1,\phi)$ in the relative configuration space for rolling spheres}
\label{touchingspheres}
\end{figure}

Consider the Lie group \[\mathrm{S}(\Orth(3)\times\Orth(3)):=\{(A_0,A_1)\in\Orth(3)\times\Orth(3):\det(A_0)\det(A_1)=1\}.\] This has a transitive action on $M$ given by \[(A_0,A_1)\cdot(x_0,x_1,\phi):=(A_0x_0,A_1x_1,A_1\circ\phi\circ A_0^{-1}),\] where we note that the determinant condition on $(A_0,A_1)$ guarantees that $A_1\circ\phi\circ A_0^{-1}$ preserves or reverses orientation precisely when $\phi$ does, so the action preserves the connected component $M$. Pictorially, $(A_0,A_1)\in\mathrm{S}(\Orth(3)\times\Orth(3))$ acts on $M$ by applying the isometry $A_0$ to $\mathbb{S}^2_{r_0}$ and the isometry $A_1$ to $\mathbb{S}^2_{r_1}$. To find the stabilizer of this action, we can consider a point of the form $(r_0e_1,\pm r_1e_1,\pm\tfrac{r_1}{r_0}\mathds{1}_*)$, depending on whether $\phi=\pm\tfrac{r_1}{r_0}\mathds{1}_*$ must be orientation-preserving or reversing. We have $A_0(r_0e_1)=r_0(A_0e_1)=r_0e_1$ if and only if $A_0\in\mathrm{Stab}_{\Orth(3)}(e_1)$, and likewise, $A_1(\pm r_1e_1)=\pm r_1e_1$ if and only if $A_1\in\mathrm{Stab}_{\Orth(3)}(e_1)$, so \[\mathrm{Stab}_{\mathrm{S}(\Orth(3)\times\Orth(3))}(r_0e_1,\pm r_1e_1,\pm\tfrac{r_1}{r_0}\mathds{1}_*)\leq\mathrm{Stab}_{\Orth(3)}(e_1)\times\mathrm{Stab}_{\Orth(3)}(e_1).\] Since $\pm\tfrac{r_1}{r_0}\mathds{1}$ is central in $\GLin_3\mathbb{R}$, we have \[A_1\circ(\pm\tfrac{r_1}{r_0}\mathds{1})\circ A_0^{-1}=A_1\circ A_0^{-1}\circ(\pm\tfrac{r_1}{r_0}\mathds{1})=\pm\tfrac{r_1}{r_0}\mathds{1}\] if and only if $A_1\circ A_0^{-1}=\mathds{1}$, so the stabilizer of $(r_0e_1,\pm r_1e_1,\pm\tfrac{r_1}{r_0}\mathds{1}_*)$ is \[\{(A,A)\in\mathrm{S}(\Orth(3)\times\Orth(3)):A\in\mathrm{Stab}_{\Orth(3)}(e_1)\}\simeq\mathrm{Stab}_{\Orth(3)}(e_1)\simeq \Orth(2).\]

\begin{definition}\label{rollingsph} The model for $2$-dimensional \emph{rolling spheres} is given by $(\mathrm{S}(\Orth(3)\times\Orth(3)),\Orth(2))$, where $\Orth(2)\simeq\mathrm{Stab}_{\Orth(3)}(e_1)$ is embedded into $\mathrm{S}(\Orth(3)\times\Orth(3))$ via the diagonal embedding $A\mapsto\left(\left[\begin{smallmatrix}1 & 0 \\ 0 & A\end{smallmatrix}\right],\left[\begin{smallmatrix}1 & 0 \\ 0 & A\end{smallmatrix}\right]\right)$.\end{definition}

Note that the Lie algebra of $\mathrm{S}(\Orth(3)\times\Orth(3))$ is $\mathfrak{so}(3)\oplus\mathfrak{so}(3)$; this is precisely where all of the (local) infinitesimal symmetries for the rolling distributions come from when the radii $r_0$ and $r_1$ are not in the 1:3 ratio. Speaking of rolling distributions, we can now use the Maurer--Cartan form to give an explicit description of it on $M\cong\mathrm{S}(\Orth(3)\times\Orth(3))/\Orth(2)$.

\begin{definition}\label{rollingdist} For radii $r_0,r_1>0$, we define the corresponding \emph{rolling distribution} $D_{r_0/r_1}$ for the model geometry $(\mathrm{S}(\Orth(3)\times\Orth(3)),\Orth(2))$ to be the left-invariant distribution on $\mathrm{S}(\Orth(3)\times\Orth(3))/\Orth(2)$ given by
\begin{align*}D_{r_0/r_1} & :=\mathrm{q}_{{}_{\Orth(2)}*}\MC{\mathrm{S}(\Orth(3)\times\Orth(3))}^{-1}(\{(\tfrac{1}{r_0}X,\tfrac{1}{r_1}X):X\in\mathfrak{m}\}+\mathfrak{o}(2)) \\ & ~=\mathrm{q}_{{}_{\Orth(2)}*}\MC{\mathrm{S}(\Orth(3)\times\Orth(3))}^{-1}(\{(\tfrac{1}{r_0}X,\tfrac{1}{r_1}X):X\in\mathfrak{m}\}),\end{align*}
where $\mathfrak{m}$ is the $\Ad_{\Orth(2)}$-invariant complement to our copy of $\mathfrak{o}(2)$ in $\mathfrak{o}(3)$, used to define geodesics in 2-dimensional spherical geometry above.\end{definition}

To understand where this definition comes from, note that model geometries do not generally give us a canonical choice of scale. For example, regardless of our choice of radius $r$, the isometry group of $\mathbb{S}^2_r$ is naturally isomorphic to $\Orth(3)$ and $\mathbb{S}^2_r=\Orth(3)\cdot(re_1)\cong\Orth(3)/\Orth(2)$, so we cannot distinguish between different choices of radius by symmetry alone: we need to specify that externally because the symmetries do not specify a choice of unit for us. In particular, the same geodesic $t\mapsto \quot{\Orth(2)}(A\exp(tX))$ for the model geometry $(\Orth(3),\Orth(2))$ has different constant speeds when viewed as a geodesic for $\mathbb{S}^2_{r}$ with different choices of $r$. Since we want $D_{r_0/r_1}$ to correspond to velocities where the spheres roll along each other (without slipping or twisting), we need the two spheres to move with the same speed, so we normalize their velocities by the choices of radii.

Because $(\tfrac{1}{\lambda r_0}X,\tfrac{1}{\lambda r_1}X)=(\tfrac{1}{r_0}(\tfrac{1}{\lambda}X),\tfrac{1}{r_1}(\tfrac{1}{\lambda}X))$, the distribution $D_{r_0/r_1}$ only depends on the ratio $r_0/r_1$. Just like with spheres, however, this model geometry does not give us a canonical choice for this ratio, since the symmetries never actually ``see'' the scale for either sphere; we again need to specify it externally. Our goal, then, is to find a choice of ratio $\tfrac{r_0}{r_1}$ so that the distribution $D_{r_0/r_1}$ has more local symmetries than just those coming from the isometries of the spheres, which means that a natural next step is to figure out what type of distribution $D_{r_0/r_1}$ is and what kind of symmetries it can have. We will follow this line of thought in the next section.

\subsection{Parabolic model geometries}\hfill\\
Amongst the various types of model geometries of modern interest to differential geometers, a particularly rich class of examples comes from those whose symmetry group is semisimple and whose isotropy subgroup is parabolic.

\begin{definition}In a semisimple Lie algebra $\mathfrak{g}$, a subalgebra $\mathfrak{p}\leq\mathfrak{g}$ is said to be \emph{parabolic} if and only if $\mathfrak{p}^\perp$ is a nilpotent subalgebra, where $\mathfrak{p}^\perp$ denotes the subspace perpendicular to $\mathfrak{p}$ with respect to the Killing form $\kgf(X,Y):=\mathrm{tr}(\ad_X\circ\ad_Y)$. In a semisimple Lie group $G$, a closed subgroup $P$ is \emph{parabolic}\footnote{This is the definition used in \cite{EricksonThesis} and is equivalent to Definition 3.1.3 of \cite{CapSlovakPG1}. Note that, for this definition of parabolic subgroup, $P$ does not need to be the full normalizer of $\mathfrak{p}$, though it will always turn out to be so in the examples considered in this paper.} if and only if its corresponding subalgebra $\mathfrak{p}\leq\mathfrak{g}$ is parabolic.\end{definition}

\begin{definition}A model geometry $(G,P)$ is called \emph{parabolic} when $G$ is a semisimple Lie group and $P$ is a parabolic subgroup.\end{definition}

A parabolic model geometry $(G,P)$ for which $\mathfrak{p}^\perp$ is a $k$-step nilpotent Lie algebra always comes with a natural $\Ad_P$-invariant filtration of $\mathfrak{g}$ of the form
\[\mathfrak{g}=\mathfrak{g}^{-k}\supset\mathfrak{g}^{-k+1}\supset\cdots\supset\mathfrak{g}^k\supset\{0\},\]
with filtration components $\mathfrak{g}^i$ defined by $\mathfrak{g}^0:=\mathfrak{p}$, $\mathfrak{g}^1:=\mathfrak{p}^\perp$, and for each $i>0$, $\mathfrak{g}^{i+1}:=[\mathfrak{p}^\perp,\mathfrak{g}^i]$ and $\mathfrak{g}^{-i}:=(\mathfrak{g}^{i+1})^\perp$. This makes the semisimple Lie algebra $\mathfrak{g}$ into a \emph{filtered Lie algebra}, since $[\mathfrak{g}^i,\mathfrak{g}^j]\subseteq\mathfrak{g}^{i+j}$ for $i,j\in\mathbb{Z}$, noting that the definition above guarantees that $\mathfrak{g}^i=\mathfrak{g}$ for $i\leq -k$ and $\mathfrak{g}^i=\{0\}$ for $i>k$.

Within a parabolic subgroup $P$, there are two especially important subgroups to keep in mind. First, the nilradical $\mathfrak{p}_+:=\mathfrak{p}^\perp=\mathfrak{g}^1$ of $\mathfrak{p}$ determines a simply connected nilpotent subgroup $P_+:=\exp(\mathfrak{p}_+)$ in $P$, which we may call the \emph{positive horospherical subgroup}. This subgroup is a canonical normal subgroup within $P$. The second subgroup to keep in mind is a \emph{Levi subgroup} $G_0\leq P$, which is a choice of reductive Lie subgroup isomorphic to $P/P_+$, so that the short exact sequence \[\{e\}\to P_+\hookrightarrow P\twoheadrightarrow P/P_+\to\{e\}\] splits. There will generally be several (conjugate) choices of $G_0$ within a given $P$, so it is not canonical, though we often pick one and effectively treat it as if it is.

Choosing a Levi subgroup $G_0\leq P$ gives us a decomposition of $\mathfrak{g}$ into $G_0$-subrepresentations. Together with the filtration, this decomposition gives us a natural grading
\[\mathfrak{g}=\sum_{i=-k}^k\mathfrak{g}_i=\mathfrak{g}_{-k}+\cdots+\mathfrak{g}_0+\cdots+\mathfrak{g}_k\]
on $\mathfrak{g}$, with each grading component $\mathfrak{g}_i\approx\mathfrak{g}^i/\mathfrak{g}^{i+1}$ as $G_0$-representations, so that $\mathfrak{g}^i=\sum_{j\geq i}\mathfrak{g}_j$. Notably, this makes $\mathfrak{g}$ into a \emph{graded Lie algebra}\footnote{Following the conventions of \cite{CapSlovakPG1}, all our gradings here will be assumed to be with respect to $\mathbb{Z}$, so a ``graded Lie algebra'' will always mean a $\mathbb{Z}$-graded Lie algebra.}, since $[\mathfrak{g}_i,\mathfrak{g}_j]\subseteq\mathfrak{g}_{i+j}$ for $i,j\in\mathbb{Z}$, noting that $\mathfrak{g}^i/\mathfrak{g}^{i+1}=\{0\}$ for $|i|>k$. The grading component $\mathfrak{g}_0$ of homogeneity 0 is precisely the Lie algebra of $G_0$, and $\mathfrak{p}_+=\sum_{i>0}\mathfrak{g}_i$ is the sum of the positive grading components, justifying their notations.

Using the grading, we can then also define the \emph{negative horospherical subgroup} $G_-:=\exp(\mathfrak{g}_-)$, where $\mathfrak{g}_-$ is the nilpotent subalgebra given by the sum of the negative grading components $\mathfrak{g}_-:=\sum_{i<0}\mathfrak{g}_i$.

These algebraic aspects of parabolic geometries convey a staggeringly vast amount of geometric information in an annoyingly subtle way, so we find that it is helpful to go through a nice, relatively simple example to see what is going on.

\begin{example}Real projective geometry in dimension $m$ is encoded by the model $(\PGL_{m+1}\mathbb{R},P)$, where $\PGL_{m+1}\mathbb{R}\!=\!\GLin_{m+1}\mathbb{R}/\mathbb{R}^\times\mathds{1}$ is the quotient of $\GLin_{m+1}\mathbb{R}$ by its center and
\[P=\left\{\begin{pmatrix}a & \alpha \\ 0 & A\end{pmatrix}\in\PGL_{m+1}\mathbb{R}:a\in\mathbb{R}^\times, \alpha^\top\in\mathbb{R}^m, A\in\GLin_m\mathbb{R}\right\}\]
is the stabilizer of the 1-dimensional subspace spanned by the first basis vector in $\mathbb{R}^{m+1}$, viewed as a point of $\mathbb{RP}^m$. Here, our symmetry group $\PGL_{m+1}\mathbb{R}$ acts by projective transformations on $\PGL_{m+1}\mathbb{R}/P\cong\mathbb{RP}^m$.
\end{example}

The isotropy subgroup $P$ in this example is, of course, parabolic, with corresponding filtration of $\mathfrak{g}=\mathfrak{pgl}_{m+1}\mathbb{R}=\mathfrak{gl}_{m+1}\mathbb{R}/\langle\mathds{1}\rangle$ given by
\begin{align*}\mathfrak{g}^{-1} & =\mathfrak{pgl}_{m+1}\mathbb{R}, \\
\mathfrak{g}^0 & =\mathfrak{p}=\left\{\begin{pmatrix}r & \alpha \\ 0 & R\end{pmatrix}\in\mathfrak{pgl}_{m+1}\mathbb{R}:r\in\mathbb{R}, \alpha^\top\in\mathbb{R}^m, R\in\mathfrak{gl}_m\mathbb{R}\right\}, \\
\mathfrak{g}^1 & =\mathfrak{p}_+=\{\left(\begin{smallmatrix}0 & \alpha \\ 0 & 0\end{smallmatrix}\right)\in\mathfrak{pgl}_{m+1}\mathbb{R}: \alpha^\top\in\mathbb{R}^m\}.\end{align*}
Here, we recall that, by definition, matrices for elements of $\mathfrak{pgl}_{m+1}\mathbb{R}$ represent the same element when they differ from one another by a scalar multiple of the identity matrix $\mathds{1}$. We then have
\[P_+=\exp(\mathfrak{p}_+)=\{\left(\begin{smallmatrix}1 & \alpha \\ 0 & \mathds{1}\end{smallmatrix}\right)\in\PGL_{m+1}\mathbb{R}:\alpha^\top\in\mathbb{R}^m\},\]
and we can choose
\[G_0=\{\left(\begin{smallmatrix}1 & 0 \\ 0 & A\end{smallmatrix}\right)\in\PGL_{m+1}\mathbb{R}:A\in\GLin_m\mathbb{R}\}\simeq\GLin_m\mathbb{R}.\]
The corresponding grading of $\mathfrak{g}=\mathfrak{pgl}_{m+1}\mathbb{R}$ is given by
\begin{align*}\mathfrak{g}_{-1} & =\mathfrak{g}_-=\{\left(\begin{smallmatrix}0 & 0 \\ v & 0\end{smallmatrix}\right)\in\mathfrak{pgl}_{m+1}\mathbb{R}: v\in\mathbb{R}^m\}, \\
\mathfrak{g}_0 & =\{\left(\begin{smallmatrix}r & 0 \\ 0 & R\end{smallmatrix}\right)\in\mathfrak{pgl}_{m+1}\mathbb{R}:r\in\mathbb{R},R\in\mathfrak{gl}_m\mathbb{R}\},\text{ and} \\
\mathfrak{g}_1 & =\mathfrak{p}_+=\{\left(\begin{smallmatrix}0 & \alpha \\ 0 & 0\end{smallmatrix}\right)\in\mathfrak{pgl}_{m+1}\mathbb{R}:\alpha^\top\in\mathbb{R}^m\},\end{align*}
which further lets us specify the negative horospherical subgroup
\[G_-=\{\left(\begin{smallmatrix}1 & 0 \\ v & \mathds{1}\end{smallmatrix}\right)\in\PGL_{m+1}\mathbb{R}: v\in\mathbb{R}^m\}.\]

To see how this helps us understand the geometry of projective space, consider the image
\[\quot{P}(G_-)\cong \left\{\begin{pmatrix}1 & 0 \\ v & \mathds{1}\end{pmatrix}\begin{pmatrix}1 \\ 0\end{pmatrix}=\begin{pmatrix}1 \\ v\end{pmatrix}:v\in\mathbb{R}^m\right\}\]
of $G_-$ under the quotient map $\quot{P}:\PGL_{m+1}\mathbb{R}\to\PGL_{m+1}\mathbb{R}/P\cong\mathbb{RP}^m$. This is an affine chart through $\quot{P}(e)=(\begin{smallmatrix}1 \\ 0\end{smallmatrix})\in\mathbb{RP}^m$, and indeed, every left-translate of it by an element $g\in\PGL_{m+1}\mathbb{R}$ gives us another affine chart $g\,\quot{P}\!(G_-)=\quot{P}\!(gG_-)$ through $\quot{P}(g)$ in $\mathbb{RP}^m$. Inside $\PGL_{m+1}\mathbb{R}$, the subgroup preserving $\quot{P}(G_-)$ is
\[G_-G_0=\left\{\begin{pmatrix}1 & 0 \\ v & A\end{pmatrix}\in\PGL_{m+1}\mathbb{R}:v\in\mathbb{R}^m,A\in\GLin_m\mathbb{R}\right\},\]
and this subgroup acts transitively on $\quot{P}(G_-)$, with stabilizer $G_0$ at the point $\quot{P}(e)=(\begin{smallmatrix}1 \\ 0\end{smallmatrix})\in \quot{P}(G_-)\subset\mathbb{RP}^m$, so that $\quot{P}(G_-)$ carries the structure of the model geometry $(G_-G_0,G_0)$. Because
\[G_-G_0\simeq\mathbb{R}^m\rtimes\GLin_m\mathbb{R}\simeq\Aff(m)\]
by $(\begin{smallmatrix}1 & 0 \\ v & A\end{smallmatrix})\mapsto(v,A)$, this tells us that the geometric structure carried by the affine chart is, perhaps unsurprisingly, exactly the geometric structure of affine space $(G_-G_0,G_0)\cong(\Aff(m),\GLin_m\mathbb{R})$.

Each element $g\in\PGL_{m+1}\mathbb{R}$ therefore determines an open subset
\[g\,\quot{P}(G_-)=\quot{P}(gG_-)=\quot{P}(gG_-G_0)\]
inside $\mathbb{RP}^m$ that we aptly call an affine chart because it naturally carries the geometric structure of affine space, with $g'$ determining the same affine chart on $\mathbb{RP}^m$ as $g$ if and only if $g'\in gG_-G_0$. The various affine charts containing $\quot{P}(e)$ will, then, correspond to the ones determined by elements of the isotropy $P<\PGL_{m+1}\mathbb{R}$ lying over $\quot{P}(e)=\quot{P}(eP)$, with $p,p'\in P$ determining the same affine chart if and only if $p'\in pG_0$. Because $P/G_0\cong P_+$, it follows that $P_+$ acts simply transitively on the set of affine charts through $\quot{P}(e)$ by ``tilting'' between them; we have attempted to depict what this looks like for the 1-dimenisonal case in Figure \ref{RP1tilting}. Viewing $\PGL_{m+1}\mathbb{R}$ as a principal $P$-bundle over $\mathbb{RP}^m$, we can think of right-translation by elements of $G_0\simeq\GLin_m\mathbb{R}$ as essentially just linear changes of frame over the affine chart determined by the given element of $\PGL_{m+1}\mathbb{R}$, with right-translation by elements of $P_+$ tilting between different affine charts through the underlying point of $\mathbb{RP}^m$.

\begin{figure}[h]
\centering\includegraphics[width=\textwidth]{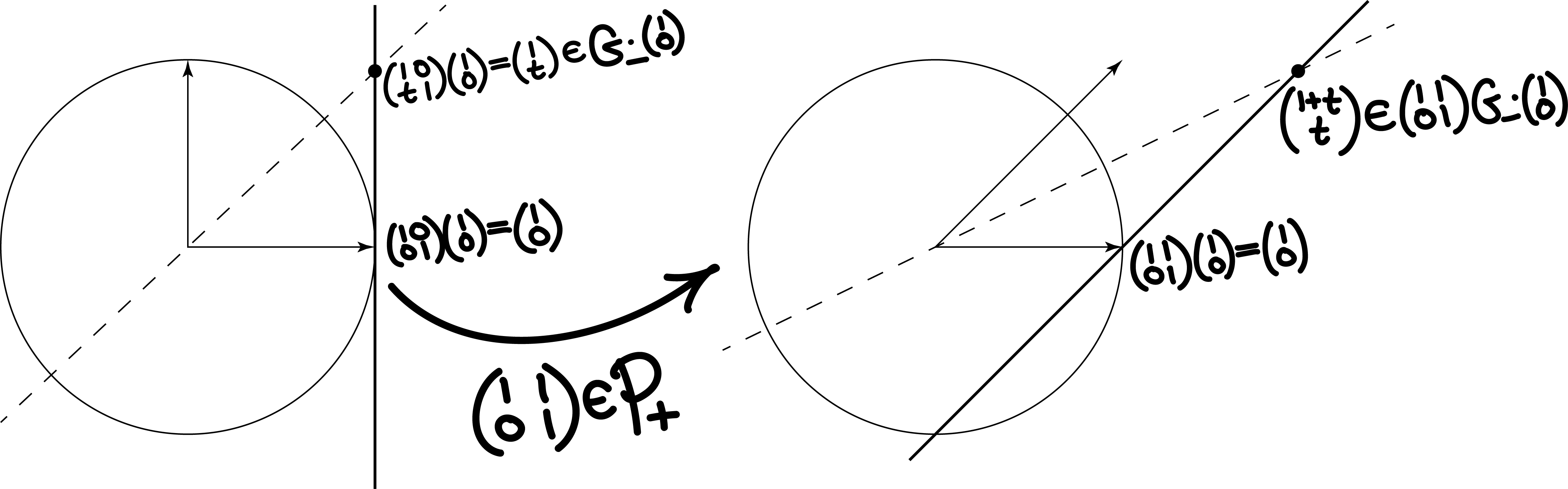}
\caption{A depiction of the way that $P_+$ tilts between the affine charts through the point $\quot{P}(e)=(\begin{smallmatrix}1 \\ 0\end{smallmatrix})\in\mathbb{RP}^1$, with the leftmost part of the picture showing the affine chart $G_-(\begin{smallmatrix}1 \\ 0\end{smallmatrix})$ determined by the identity element $(\begin{smallmatrix}1 & 0 \\ 0 & 1\end{smallmatrix})$\linebreak and the rightmost part of the picture showing the affine chart $(\begin{smallmatrix}1 & 1 \\ 0 & 1\end{smallmatrix})G_-(\begin{smallmatrix}1 \\ 0\end{smallmatrix})$ determined by $(\begin{smallmatrix}1 & 1 \\ 0 & 1\end{smallmatrix})\in P_+$}
\label{RP1tilting}
\end{figure}

By replacing affine charts with the slightly more abstract notion of open Schubert cells, the above picture actually applies to all parabolic geometries $(G,P)$ in more or less the same way. Each $g\in G$ determines an open Schubert cell $\quot{P}\!(gG_-)=\quot{P}\!(gG_-G_0)$ through $\quot{P}\!(g)\!\in\! G/P$, and this is diffeomorphic to $G_-$ because $G_-\cap P=\{e\}$. The open Schubert cells are imbued with the geometric structure of $(G_-G_0,G_0)$, which we can think of as the ``reductive analogue'' of the parabolic geometry $(G,P)$, and each of them will be dense in $G/P$ when $P=\mathrm{N}_G(\mathfrak{p})$, so that $G/P$ can be thought of as a compactification of $G_-\cong G_-G_0/G_0$. Viewing $G$ as a principal $P$-bundle over $G/P$, for $g\in G$, we can think of right-translating by an element of $G_0$ as changing the ``linear frame'' over the open Schubert cell $\quot{P}(gG_-G_0)$, while right-translating by an element of $P_+$ tilts between the different open Schubert cells through the underlying point $\quot{P}(g)\in G/P$.

\section{Building \texorpdfstring{$\Zhe_2$}{Zhe2} from \texorpdfstring{$(2,3,5)$}{(2,3,5)}-distributions}\label{buildingZhe2}
The adjoint split real form $\Zhe_2$ of the exceptional simple Lie group of type $\mathrm{G}_2$ is sometimes characterized as the automorphism group of the split-octonions; see, for example, \cite{Agrachev2007}, \cite{BaezHuerta2014}, and \cite{Evans2025} for elaborations on this perspective. However, while the octonionic viewpoint can be beneficial in some contexts, it ends up hiding much of the geometric intuition for the Lie-theoretic structure of $\Zhe_2$ behind the algebraic structure of the split-octonions, which can be especially unhelpful when one is not already familiar with them. As such, we find it better to take a more hands-on approach: we will build the Lie group $\Zhe_2$ for ourselves from the geometry of $(2,3,5)$-distributions, mirroring Chapter 1 of \cite{EricksonThesis}. This has the benefit of naturally partitioning the Lie algebra of $\Zhe_2$ into geometrically meaningful pieces that are easier to visualize while also incidentally proving that this Lie algebra is maximal among possible symmetry algebras for this type of distribution, which guarantees that the Lie algebra of local infinitesimal symmetries must be exactly the Lie algebra of $\Zhe_2$ if it contains a copy of it.

\subsection{Distributions with small growth vector \texorpdfstring{$(2,3,5)$}{(2,3,5)}}\hfill \\
For a distribution $D=:D^{(1)}$, we will write $D^{(2)}:=[D,D]+D$ for the distribution spanned by Lie brackets of vector fields in $D$ and
\[D^{(3)}:=[D,D^{(2)}]+D^{(2)}=[D,[D,D]]+[D,D]+D\]
for the distribution spanned by Lie brackets of vector fields in $D$ with vector fields in $D^{(2)}$. For convenience, we will also occasionally use $D^{(i)}$ with $i\leq 0$ to denote the zero distribution.

\begin{definition}A rank 2 distribution $D$ on a 5-manifold $M$ is called a \emph{$(2,3,5)$-distribution} if and only if $D^{(2)}$ has rank 3 and $D^{(3)}$ has rank 5, so that $D^{(3)}$ is equal to the whole tangent space at each point of $M$.\end{definition}


Whenever $\tfrac{r_0}{r_1}\neq 1$, the rolling distribution $D_{r_0/r_1}$ from Definition \ref{rollingdist} is an example of a $(2,3,5)$-distribution. To see this, consider the vector fields
\[\xi_1:=\MC{\mathrm{S}(\Orth(3)\times\Orth(3))}^{-1}\left(\tfrac{1}{r_0}\begin{smallbmatrix}0 & -e_1^\top \\ e_1 & 0\end{smallbmatrix}, \tfrac{1}{r_1}\begin{smallbmatrix}0 & -e_1^\top \\ e_1 & 0\end{smallbmatrix}\right)\]
and
\[\xi_2:=\MC{\mathrm{S}(\Orth(3)\times\Orth(3))}^{-1}\left(\tfrac{1}{r_0}\begin{smallbmatrix}0 & -e_2^\top \\ e_2 & 0\end{smallbmatrix}, \tfrac{1}{r_1}\begin{smallbmatrix}0 & -e_2^\top \\ e_2 & 0\end{smallbmatrix}\right)\]
on $\mathrm{S}(\Orth(3)\times\Orth(3))$, so that $\mathrm{q}_{{}_{\Orth(2)}*}\xi_1$ and $\mathrm{q}_{{}_{\Orth(2)}*}\xi_2$ span $D_{r_0/r_1}$. Then, since Maurer--Cartan forms $\MC{G}$ satisfy $[\MC{G}^{-1}(X),\MC{G}^{-1}(Y)]=\MC{G}^{-1}([X,Y])$, we have that
\[[\xi_1,\xi_2]=\MC{\mathrm{S}(\Orth(3)\times\Orth(3))}^{-1}\left(\tfrac{1}{r_0^2}\begin{smallbmatrix}0 & 0 & 0 \\ 0 & 0 & -1 \\ 0 & 1 & 0\end{smallbmatrix},\tfrac{1}{r_1^2}\begin{smallbmatrix}0 & 0 & 0 \\ 0 & 0 & -1 \\ 0 & 1 & 0\end{smallbmatrix}\right)=:2\xi_3,\]
which is not in $\mathrm{q}_{{}_{\Orth(2)}*}^{-1}(D_{r_0/r_1})$ unless $r_0=r_1$, so $D_{r_0/r_1}^{(2)}$ has rank 3. Similarly, we have
\[[\xi_3,\xi_1]=\MC{\mathrm{S}(\Orth(3)\times\Orth(3))}^{-1}\left(\tfrac{1}{2r_0^3}\begin{smallbmatrix}0 & -e_2^\top \\ e_2 & 0\end{smallbmatrix}, \tfrac{1}{2r_1^3}\begin{smallbmatrix}0 & -e_2^\top \\ e_2 & 0\end{smallbmatrix}\right)\]
and
\[[\xi_3,\xi_2]=\MC{\mathrm{S}(\Orth(3)\times\Orth(3))}^{-1}\left(\tfrac{-1}{2r_0^3}\begin{smallbmatrix}0 & -e_1^\top \\ e_1 & 0\end{smallbmatrix}, \tfrac{-1}{2r_1^3}\begin{smallbmatrix}0 & -e_1^\top \\ e_1 & 0\end{smallbmatrix}\right),\]
which are again not in $\mathrm{q}_{{}_{\Orth(2)}*}^{-1}(D_{r_0/r_1}^{(2)})$ unless $r_0=r_1$, from which it follows that $D_{r_0/r_1}^{(3)}$ has rank 5.

Since we want to find and understand local symmetries of certain $(2,3,5)$-distributions, a reasonable starting point would be to figure out how much local symmetry a $(2,3,5)$-distribution can have. This is accomplished by Tanaka prolongation, introduced by Tanaka in \cite{Tanaka1970}, which allows us to build a symmetry algebra of maximal dimension for certain types of distribution. We will briefly describe how this works for $(2,3,5)$-distributions, which will allow us to build the Lie algebra of $\Zhe_2$ directly.

\subsection{A glimpse at Tanaka prolongation}\label{tanakalight} \hfill\\
Borrowing the terminology of geometric control theorists, we want to start by ``nilpotentizing'' the bracket structure of a $(2,3,5)$-distribution, to try to capture its geometric structure with the least complicated Lie algebra that we can.

\begin{definition}The \emph{symbol algebra} of a $(2,3,5)$-distribution $D$ at a point $x\in M$ is the graded 
nilpotent Lie algebra
\[\mathfrak{g}_-:=\mathfrak{g}_{-3}+\mathfrak{g}_{-2}+\mathfrak{g}_{-1},\]
with grading components given by $\mathfrak{g}_{-i}:=D^{(i)}_x/D^{(i-1)}_x$ for each $i$, whose bracket is the one induced by the Lie bracket of vector fields.\end{definition}

To clarify what we mean by the bracket induced by the Lie bracket of vector fields, note that if $\tilde{X}$ is a vector field with values in $D^{(i)}$ representing $X\in\mathfrak{g}_{-i}$, then every other vector field corresponding to $X$ is of the form $\tilde{X}+V$ for some vector field $V$ with values in $D^{(i-1)}$, and similarly, if $\tilde{Y}$ is a vector field with values in $D^{(j)}$ representing $Y\in\mathfrak{g}_{-j}$, then every other vector field corresponding to $Y$ is of the form $\tilde{Y}+W$ for some vector field $W$ with values in $D^{(j-1)}$. Since
\[[\tilde{X}+V,\tilde{Y}+W]_x=[\tilde{X},\tilde{Y}]_x+[V,\tilde{Y}]_x+[\tilde{X},W]_x+[V,W]_x\]
is contained in $[\tilde{X},\tilde{Y}]_x+D^{(i+j-1)}_x$, the choice of extension of $X$ and $Y$ to overlying vector fields therefore does not affect their Lie bracket modulo $D^{(i+j-1)}_x$, so we get a well-defined bracket on the graded vector space $\mathfrak{g}_-$.

The symbol algebra for a $(2,3,5)$-distribution will always be the same up to isomorphism of graded Lie algebras, regardless of the choice of point $x\in M$ or even the choice of distribution itself: for a basis $\{e_1,e_2\}$ of $D_x$, we will have $\mathfrak{g}_{-1}=\langle e_1\rangle\oplus\langle e_2\rangle\approx\mathbb{R}^2$, $\mathfrak{g}_{-2}=\langle[e_1,e_2]\rangle$, and $\mathfrak{g}_{-3}=\langle[[e_1,e_2],e_1]\rangle\oplus\langle[[e_1,e_2],e_2]\rangle$, with $\mathfrak{g}_{-i}=\{0\}$ for all $i>0$, so if we used a different point or different distribution, then for a basis $\{e_1',e_2'\}$ for the distribution at that point, we would have an isomorphism of graded Lie algebras by just sending $e_i\mapsto e_i'$, $[e_1,e_2]\mapsto[e_1',e_2']$, and $[[e_1,e_2],e_i]\mapsto[[e_1',e_2'],e_i']$. As such, it makes sense to refer to $\mathfrak{g}_-$ as ``the'' symbol algebra for $(2,3,5)$-distributions, since they are all isomorphic.

\begin{definition}We define
\[G_0:=\Aut_\mathrm{gr}(\mathfrak{g}_-)=\{\phi\in\Aut(\mathfrak{g}_-):\phi(\mathfrak{g}_{-i})=\mathfrak{g}_{-i}\text{ for each }i\},\] the Lie group of automorphisms of $\mathfrak{g}_-$ as a graded Lie algebra. We may preemptively refer to this as the \emph{Levi subgroup for $(2,3,5)$-distributions}.\end{definition}

This allows us to give a more concretely algebraic description of $\mathfrak{g}_-$.

\begin{proposition}As representations of $G_0\simeq\GLin(\mathfrak{g}_{-1})\simeq\GLin_2\mathbb{R}$, we may identify $\mathfrak{g}_{-1}\approx\mathbb{R}^2$, $\mathfrak{g}_{-2}\approx\Lambda^2\mathbb{R}^2$, and $\mathfrak{g}_{-3}\approx(\Lambda^2\mathbb{R}^2)\otimes\mathbb{R}^2$, so that
\[\mathfrak{g}_-\approx((\Lambda^2\mathbb{R}^2)\otimes\mathbb{R}^2)\oplus\Lambda^2\mathbb{R}^2\oplus\mathbb{R}^2\]
with bracket given by $[v,w]:=2v\wedge w$ and $[v\wedge w,u]:=3v\wedge w\otimes u$ for $v,w,u\in\mathbb{R}^2\approx\mathfrak{g}_{-1}$, and so that $\mathfrak{g}_{-3}\approx(\Lambda^2\mathbb{R}^2)\otimes\mathbb{R}^2$ is central.\end{proposition}
\begin{proof}Because the symbol algebra $\mathfrak{g}_-$ is generated by $\mathfrak{g}_{-1}$, each $\phi\in G_0$ is uniquely determined by its restriction $\phi|_{\mathfrak{g}_{-1}}\in\GLin(\mathfrak{g}_{-1})$ to $\mathfrak{g}_{-1}$:
\[\phi([v,w])=[\phi(v),\phi(w)]=[\phi|_{\mathfrak{g}_{-1}}(v),\phi|_{\mathfrak{g}_{-1}}(w)],\]
and likewise,
\[\phi([[v,w],u])=[\phi([v,w]),\phi(u)]=[[\phi|_{\mathfrak{g}_{-1}}(v),\phi|_{\mathfrak{g}_{-1}}(w)],\phi|_{\mathfrak{g}_{-1}}(u)].\]
Moreover, every $A\in\GLin(\mathfrak{g}_{-1})$ uniquely determines a corresponding element in $G_0$ given by $e_i\mapsto A(e_i)$, $[e_1,e_2]\mapsto [A(e_1),A(e_2)]$, and $[[e_1,e_2],e_i]\mapsto [[A(e_1),A(e_2)],A(e_i)]$ for a basis $\{e_1,e_2\}$ of $\mathfrak{g}_{-1}\approx\mathbb{R}^2$, so we can conveniently identify $G_0$ with $\GLin(\mathfrak{g}_{-1})\simeq\GLin_2\mathbb{R}$. Under this identification, for an element $A=\begin{smallbmatrix}a & b \\ c & d\end{smallbmatrix}\in\GLin_2\mathbb{R}\simeq G_0$, we have
\begin{align*}A([e_1,e_2]) & =[Ae_1,Ae_2]=[ae_1+ce_2,be_1+de_2] \\ & =(ad-bc)[e_1,e_2]=\det(A)[e_1,e_2]\end{align*}
and
\[A([[e_1,e_2],u])=[A([e_1,e_2]),A(u)]=\det(A)[[e_1,e_2],A(u)]\]
for each $u\in\mathbb{R}^2\approx\mathfrak{g}_{-1}$, so as $\GLin_2\mathbb{R}$-representations, $\mathfrak{g}_{-2}\approx\Lambda^2\mathbb{R}^2$ and $\mathfrak{g}_{-3}\approx(\Lambda^2\mathbb{R}^2)\otimes\mathbb{R}^2$. Therefore, we may write the symbol algebra as the graded $\GLin_2\mathbb{R}$-representation
\[\mathfrak{g}_-\approx((\Lambda^2\mathbb{R}^2)\otimes\mathbb{R}^2)\oplus\Lambda^2\mathbb{R}^2\oplus\mathbb{R}^2,\]
and by rescaling each grading component as necessary, we may write the bracket for $\mathfrak{g}_-$ so that $[e_1,e_2]:=2e_1\wedge e_2$ and $[e_1\wedge e_2,u]:=3e_1\wedge e_2\otimes u$ for $u\in\mathbb{R}^2\approx\mathfrak{g}_{-1}$.\mbox{\qedhere}\end{proof}

The general behavior of this graded nilpotent Lie algebra $\mathfrak{g}_-$ can be captured by the diagram in Figure \ref{gminus_roots}. Each node of the diagram represents a $1$-dimensional subspace in a grading component of $\mathfrak{g}_-$, and the bracket of elements in these subspaces is given by adding together the vectors ending in those nodes: if the result is another node on the diagram, then the bracket is in that node's $1$-dimensional subspace, and if the result runs off of the diagram, then the bracket is $0$.

\begin{figure}[h]
\centering\includegraphics[width=0.4\textwidth]{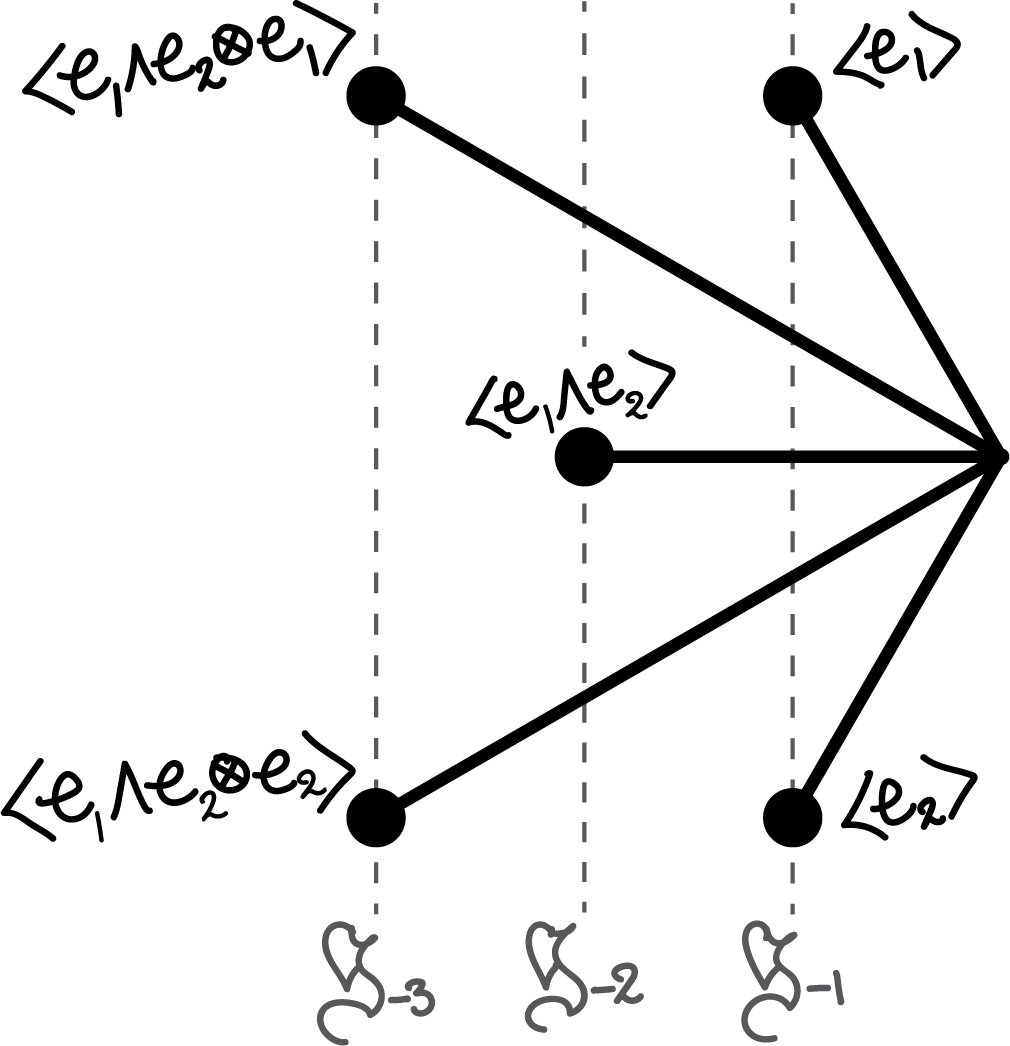}
\caption{A diagram to help keep track of the bracket on $\mathfrak{g}_-$, with the grading components $\mathfrak{g}_{-3}$, $\mathfrak{g}_{-2}$, and $\mathfrak{g}_{-1}$ labeled by dashed vertical gray lines}
\label{gminus_roots}
\end{figure}

Defining $G_-$ to be the simply connected nilpotent Lie group with Lie algebra $\mathfrak{g}_-$, we get a natural $(2,3,5)$-distribution $D:=\MC{G_-}^{-1}(\mathfrak{g}_{-1})$ on $G_-$, since $[\MC{G_-}^{-1}(X),\MC{G_-}^{-1}(Y)]=\MC{G_-}^{-1}([X,Y])$ for all $X,Y\in\mathfrak{g}_-$. By construction, this distribution is invariant under left-translation by the elements of $G_-$. Moreover, each $\phi\in G_0\simeq\GLin_2\mathbb{R}$ determines a unique corresponding Lie automorphism of $G_-$ given by $\exp(X)\mapsto\exp(\phi(X))$ for $X\in\mathfrak{g}_-$, since $G_-$ is simply connected; denoting this automorphism by the same symbol $\phi$, we then have that $\phi^*\MC{G_-}=\phi(\MC{G_-})$, so because elements of $G_0$ preserve $\mathfrak{g}_{-1}$, each $\phi\in G_0$ gives us a new symmetry of the distribution $D=\MC{G_-}^{-1}(\mathfrak{g}_{-1})$. Combining these transformations into the semidirect product $G_-\rtimes G_0$ gives us the full symmetry group for this distribution.

Thus, we have a transitive action of a Lie group $G_-G_0:=G_-\rtimes G_0$ on a connected manifold $G_-\cong G_-G_0/G_0$, which means we have a model geometry $(G_-G_0,G_0)$.

\begin{definition}We may refer to $(G_-G_0,G_0)$ as the \emph{model geometry for graded $(2,3,5)$-structures}.\end{definition}

Note that this model geometry is reductive, as $\mathfrak{g}_-$ is an $\Ad_{G_0}$-invariant complement to $\mathfrak{g}_0$ in $\mathfrak{g}_-+\mathfrak{g}_0$. By analogy with Euclidean and affine geometry, we can imagine $G_-$ as a 5-dimensional nilpotent subgroup of ``translations'' acting freely and transitively on $G_-\cong G_-G_0/G_0$, with the isotropy subgroup $G_0$ acting by analogues of linear transformations on this nilpotent translation space. Thinking of $G_-G_0$ as a principal $G_0$-bundle over $G_-\cong G_-G_0/G_0$, we may identify it with the bundle of frames for the invariant rank 2 distribution $D=\MC{G_-}^{-1}(\mathfrak{g}_{-1})$, with fiber over $x\in G_-$ given by the space of linear isomorphisms from $\mathbb{R}^2$ to $D_x$ and right-action by $G_0\simeq\GLin_2\mathbb{R}$ given by precomposition, so that if $\phi:\mathbb{R}^2\to D_x$ is a frame over $x$ for the distribution $D$ and $A\in\GLin_2\mathbb{R}$, then $\phi\cdot A:=\phi\circ A$. In other words, right-translation by an element of $G_0$ amounts to a linear change of frame within the distribution $D$.

Since we have expanded our Lie group of symmetries, let us also expand our diagram from Figure \ref{gminus_roots} to help us keep track of the bracket in this larger Lie algebra $\mathfrak{g}_-+\mathfrak{g}_0$. The element $\begin{smallbmatrix}0 & 1 \\ 0 & 0\end{smallbmatrix}\in\mathfrak{g}_0\approx\mathfrak{gl}_2\mathbb{R}$ acts on $\mathfrak{g}_{-1}\approx\mathbb{R}^2$ by sending $e_2=\begin{smallbmatrix}0 \\ 1\end{smallbmatrix}$ to $e_1=\begin{smallbmatrix}1 \\ 0\end{smallbmatrix}$ and $e_1$ to $0$. It also sends the element $e_1\wedge e_2\in\mathfrak{g}_{-2}$ to $0$, and $e_1\wedge e_2\otimes e_2\in\mathfrak{g}_{-3}$ gets sent to $e_1\wedge e_2\otimes e_1$, which then gets sent to $0$ too. In our diagram, we can therefore represent the subspace spanned by $\begin{smallbmatrix}0 & 1 \\ 0 & 0\end{smallbmatrix}\in\mathfrak{g}_0$ by the upward vector that we could add to the node for $e_2$ to get to the node for $e_1$. Similarly, we can add a downward vector for $\begin{smallbmatrix}0 & 0 \\ 1 & 0\end{smallbmatrix}$. The subalgebra of diagonal matrices in $\mathfrak{g}_0\approx\mathfrak{gl}_2\mathbb{R}$ only rescales the subspaces represented by the other nodes, including the two we just added, so bracketing with them does not change the subspace in the Lie algebra; we can indicate this by imagining this 2-dimensional subalgebra $\mathfrak{a}=\langle\begin{smallbmatrix}1 & 0 \\ 0 & 0\end{smallbmatrix},\begin{smallbmatrix}0 & 0 \\ 0 & 1\end{smallbmatrix}\rangle$ as being represented by the zero vector in the diagram. The new diagram resulting from these additions can be seen in Figure \ref{gminusandg0}.

\begin{figure}[h]
\centering\includegraphics[width=0.5\textwidth]{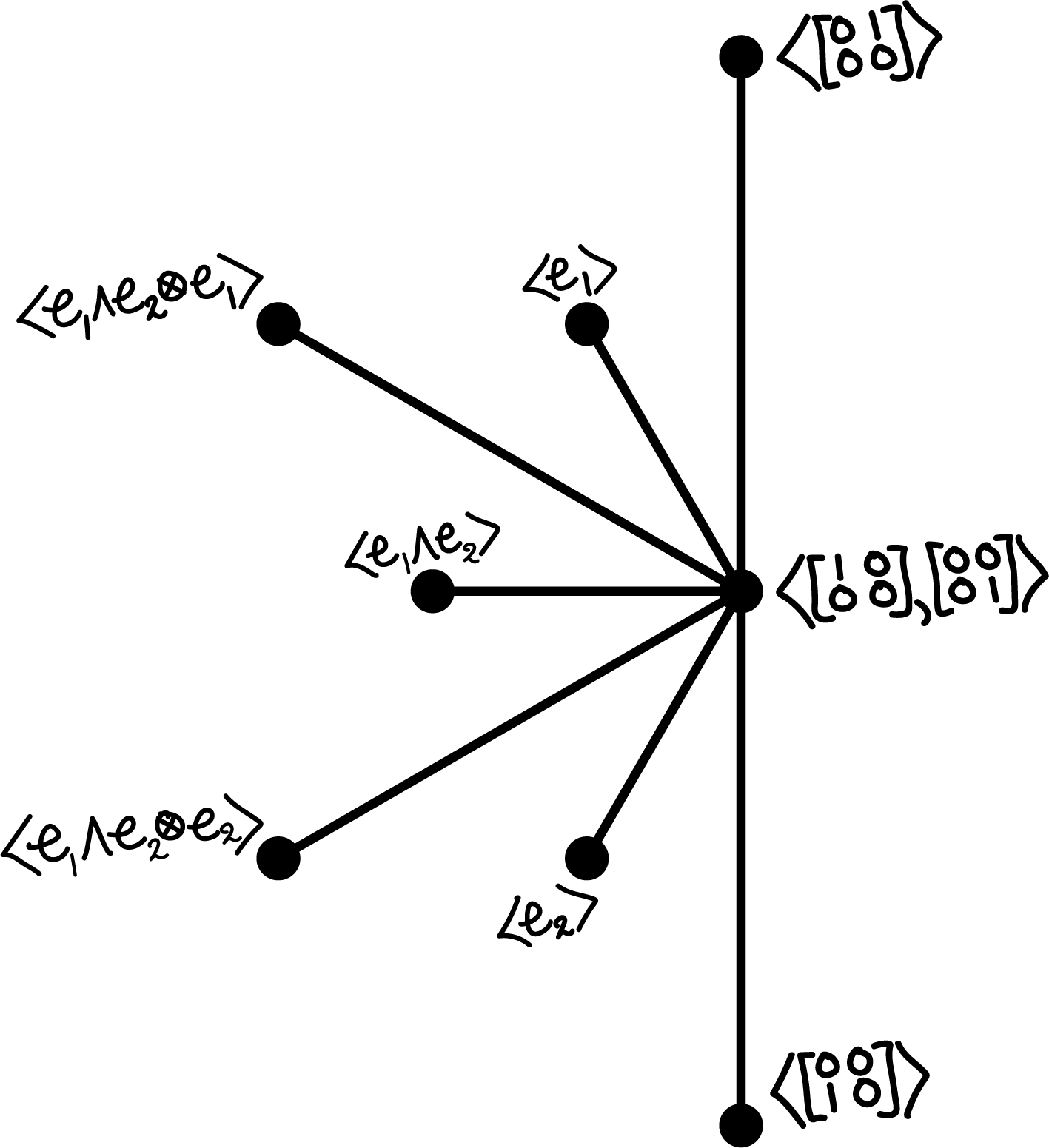}
\caption{A diagram to help keep track of the bracket on $\mathfrak{g}_-+\mathfrak{g}_0$, given by adding the nodes from $\mathfrak{g}_0$ to the previous diagram as indicated}
\label{gminusandg0}
\end{figure}

The symmetries in $G_-G_0$ not only preserve the $(2,3,5)$-distribution
\begin{align*}D=\MC{G_-}^{-1}(\mathfrak{g}_{-1}) & =\mathrm{q}_{{}_{G_0}*}\MC{G_-G_0}^{-1}(\mathfrak{g}_{-1}) \\ & =\mathrm{q}_{{}_{G_0}*}\MC{G_-G_0}^{-1}(\mathfrak{g}_{-1}+\mathfrak{g}_0)\end{align*}
on $G_-$, but also the invariant distributions $\MC{G_-}^{-1}(\mathfrak{g}_{-2})=\mathrm{q}_{{}_{G_0}*}\MC{G_-G_0}^{-1}(\mathfrak{g}_{-2})$ and $\MC{G_-}^{-1}(\mathfrak{g}_{-3})=\mathrm{q}_{{}_{G_0}*}\MC{G_-G_0}^{-1}(\mathfrak{g}_{-3})$ as well. Altogether, this tells us that $G_-G_0$ preserves a grading on the tangent bundle of $G_-\cong G_-G_0/G_0$. Preserving this graded structure is a considerably stronger condition on our symmetries than just preserving the distribution. Indeed, since
\begin{align*}D^{(2)}=[D,D]+D & =\MC{G_-}^{-1}(\mathfrak{g}_{-2}+\mathfrak{g}_{-1}) \\ & =\mathrm{q}_{{}_{G_0}*}\MC{G_-G_0}^{-1}(\mathfrak{g}_{-2}+\mathfrak{g}_{-1}+\mathfrak{g}_0)\end{align*}
and
\begin{align*}D^{(3)}=TG_- & =\MC{G_-}^{-1}(\mathfrak{g}_{-3}+\mathfrak{g}_{-2}+\mathfrak{g}_{-1}) \\ & =\mathrm{q}_{{}_{G_0}*}\MC{G_-G_0}^{-1}(\mathfrak{g}_{-3}+\mathfrak{g}_{-2}+\mathfrak{g}_{-1}+\mathfrak{g}_0),\end{align*}
preserving the distribution $D$ is equivalent to simply preserving the underlying filtration $D^{(3)}\supset D^{(2)}\supset D$ induced by successively adding up those grading components. The idea of Tanaka prolongation, then, is to find the additional local infinitesimal symmetries for the filtration by algebraically constructing new higher grading components $\mathfrak{g}_i$ (with $i>0$) to add to our existing graded symmetry algebra. Each of these positive grading components $\mathfrak{g}_i$ will break the graded structure on the tangent bundle because they will allow the degree of the grading to be pushed upward, but this will still preserve the underlying filtration.

Let us demonstrate how the construction works by building $\mathfrak{g}_1$. By the definition of graded Lie algebras, we want $[\mathfrak{g}_i,\mathfrak{g}_j]\subseteq\mathfrak{g}_{i+j}$ for all $i,j\in\mathbb{Z}$, where we again take $\mathfrak{g}_i=\{0\}$ for $i<-3$. In particular, we should have $[\mathfrak{g}_1,\mathfrak{g}_i]\subseteq\mathfrak{g}_{i+1}$ for all $i<0$, so each $\alpha\in\mathfrak{g}_1$ induces a linear map $\tilde{\alpha}:=\ad_\alpha|_{\mathfrak{g}_-}:\mathfrak{g}_-\to\mathfrak{g}_-+\mathfrak{g}_0$. This linear map $\tilde{\alpha}$ uniquely determines $\alpha$, since if $\tilde{\alpha}=0$, then the local symmetries generated by $\alpha$ preserve the grading on $G_-$, so ${\alpha\in(\mathfrak{g}_-+\mathfrak{g}_0)\cap\mathfrak{g}_1=\{0\}}$. Moreover, the bracket on $\mathfrak{g}$ must satisfy the Jacobi identity, so $\tilde{\alpha}=\ad_\alpha|_{\mathfrak{g}_-}$ must satisfy $\tilde{\alpha}([X,Y])=[\tilde{\alpha}(X),Y]+[X,\tilde{\alpha}(Y)]$ for all $X,Y\in\mathfrak{g}_-$, and since $\mathfrak{g}_-$ is generated by $\mathfrak{g}_{-1}$, it follows that $\tilde{\alpha}$---and hence $\alpha$---is uniquely determined by its restriction $\tilde{\alpha}|_{\mathfrak{g}_{-1}}=\ad_\alpha|_{\mathfrak{g}_{-1}}$ to $\mathfrak{g}_{-1}$:
\begin{align*}\tilde{\alpha}(e_1\wedge e_2) & =\frac{1}{2}\tilde{\alpha}([e_1,e_2])=\frac{1}{2}([\tilde{\alpha}(e_1),e_2]+[e_1,\tilde{\alpha}(e_2)]) \\ & =\frac{1}{2}(\tilde{\alpha}(e_1)e_2-\tilde{\alpha}(e_2)e_1),\end{align*}
and similarly,
\begin{align*}\tilde{\alpha}(e_1\wedge e_2\otimes u) & =\frac{1}{3}\tilde{\alpha}([e_1\wedge e_2,u]) \\ & =\frac{1}{3}([\tilde{\alpha}(e_1\wedge e_2),u]+[e_1\wedge e_2,\tilde{\alpha}(u)]) \\ & =\frac{1}{3}(2(\tilde{\alpha}(e_1)e_2-\tilde{\alpha}(e_2)e_1)\wedge u-\mathrm{tr}(\tilde{\alpha}(u))e_1\wedge e_2) \\ & =-c_\alpha(u)e_1\wedge e_2\end{align*}
for some linear functional $c_\alpha\in(\mathbb{R}^2)^\vee$ in the dual of $\mathbb{R}^2\approx\mathfrak{g}_{-1}$.

Since $\mathfrak{g}_{-3}$ is central in $\mathfrak{g}_-$, though, we must also have
\begin{align*}0 & =\tilde{\alpha}([e_1\wedge e_2\otimes u,v])=[\tilde{\alpha}(e_1\wedge e_2\otimes u),v]+[e_1\wedge e_2\otimes u,\tilde{\alpha}(v)] \\ & =-3 c_\alpha(u)e_1\wedge e_2\otimes v-\tilde{\alpha}(v)\cdot(e_1\wedge e_2\otimes u),\end{align*}
so $\tilde{\alpha}(v)$ needs to be the (unique) element of $\mathfrak{g}_0\approx\mathfrak{gl}_2\mathbb{R}$ such that, for each $u\in\mathbb{R}^2$,
\[\tilde{\alpha}(v)\cdot(e_1\wedge e_2\otimes u)=e_1\wedge e_2\otimes(\tilde{\alpha}(v)u+\mathrm{tr}(\tilde{\alpha}(v))u)\]
is equal to $-3 c_\alpha(u)e_1\wedge e_2\otimes v$; since $\mathfrak{g}_{-1}\approx\mathbb{R}^2$ is 2-dimensional, this means that
\[\tilde{\alpha}(v)=-3c_\alpha\otimes v-\tfrac{1}{2+1}(-3c_\alpha(v))\mathds{1}=c_\alpha(v)\mathds{1}-3c_\alpha\otimes v.\]
In particular, $\alpha$ is uniquely determined by the linear functional $c_\alpha$ given by $c_\alpha(v)=-\mathrm{tr}(\tilde{\alpha}(v))$, and because each linear functional $c_\alpha\in(\mathbb{R}^2)^\vee$ defines a corresponding $\tilde{\alpha}$ by setting $\tilde{\alpha}(v)=c_\alpha(v)\mathds{1}-3c_\alpha\otimes v$, we may identify $\mathfrak{g}_1$ with $(\mathbb{R}^2)^\vee\approx(\mathfrak{g}_{-1})^\vee$. This identification is more than just an isomorphism of vector spaces: identifying $\alpha\in\mathfrak{g}_1$ with its linear functional ${c_\alpha\in(\mathbb{R}^2)^\vee}$, we have, for $A\in\GLin_2\mathbb{R}\simeq G_0$ and $v\in\mathbb{R}^2\approx\mathfrak{g}_{-1}$,
\begin{align*}\widetilde{\Ad_A(\alpha)}(v) & =[\Ad_A(\alpha),v]=\Ad_A[\alpha,\Ad_{A^{-1}}(v)]=\Ad_A(\tilde{\alpha}(A^{-1}(v))) \\ & =A\left(\alpha(A^{-1}v)\mathds{1}-3\alpha\otimes A^{-1}v\right)A^{-1} \\ & =(\alpha\circ A^{-1})(v)\mathds{1}-3(\alpha\circ A^{-1})\otimes v \\ & =\widetilde{\alpha\circ A^{-1}}(v),\end{align*}
so $\mathfrak{g}_1\approx(\mathbb{R}^2)^\vee$ as $\GLin_2\mathbb{R}$-representations as well.

This process for building higher-order local infinitesimal symmetries continues in much the same way for each positive grading component, successively defining
\[\mathfrak{g}_i\approx\ad_{\mathfrak{g}_i}|_{\mathfrak{g}_-}:=\{\zeta\in\mathrm{Der}(\mathfrak{g}_-;{\textstyle\sum}_{j<i}\mathfrak{g}_j):\zeta(\mathfrak{g}_j)\subseteq\mathfrak{g}_{i+j}\text{ for }j<0\}\]
for each $i>0$, where
\[\mathrm{Der}(\mathfrak{g}_-;V):=\left\{\zeta\in\mathfrak{g}_-^\vee\otimes V:\begin{array}{c}\zeta([X,Y])=X\cdot\zeta(Y)-Y\cdot\zeta(X) \\ \text{for all }X,Y\in\mathfrak{g}_-\end{array}\right\}\]
is the space of \emph{derivations} from the Lie algebra $\mathfrak{g}_-$ to a representation $V$ of $\mathfrak{g}_-$. Ultimately, we find that, as representations of $G_0\simeq\GLin_2\mathbb{R}$, $\mathfrak{g}_2\approx\Lambda^2(\mathbb{R}^2)^\vee$ and $\mathfrak{g}_3\approx(\Lambda^2(\mathbb{R}^2)^\vee)\otimes(\mathbb{R}^2)^\vee$, with all grading components $\mathfrak{g}_i$ for $i>3$ vanishing; see Appendix A of \cite{EricksonThesis} for details. 

\begin{definition}\label{tanakaprolongation} The \emph{Tanaka prolongation of $\mathfrak{g}_-+\mathfrak{g}_0$} is the result of the construction described above. In our case, it is the finite-dimensional graded Lie algebra $\mathfrak{g}:=\sum_{i=-3}^3\mathfrak{g}_i$, with grading components
\begin{gather*}\mathfrak{g}_{-3}\approx(\Lambda^2\mathbb{R}^2)\otimes\mathbb{R}^2,~ \mathfrak{g}_{-2}\approx\Lambda^2\mathbb{R}^2,~ \mathfrak{g}_{-1}\approx\mathbb{R}^2, \\ \mathfrak{g}_0\approx\mathfrak{gl}_2\mathbb{R}, \\ \mathfrak{g}_1\approx(\mathbb{R}^2)^\vee,~ \mathfrak{g}_2\approx\Lambda^2(\mathbb{R}^2)^\vee, \text{ and}~ \mathfrak{g}_3\approx(\Lambda^2(\mathbb{R}^2)^\vee)\otimes(\mathbb{R}^2)^\vee,\end{gather*}
with bracket given by, for $v,w,u\in\mathbb{R}^2\approx\mathfrak{g}_{-1}$ and $\alpha,\alpha',\beta\in(\mathbb{R}^2)^\vee\approx\mathfrak{g}_1$:
\begin{gather*}[v,w]=2v\wedge w, [v\wedge w,u]=3v\wedge w\otimes u, \\
[\alpha,\alpha']=-2\alpha\wedge\alpha',~ [\alpha\wedge\alpha',\beta]=-3\alpha\wedge\alpha'\otimes\beta, \\
[\alpha,v]=\alpha(v)\mathds{1}-3\alpha\otimes v,~ [\alpha,v\wedge w]=2\iota_\alpha(v\wedge w), \\
[\alpha\wedge\alpha',v\wedge w]=-\alpha\wedge\alpha'(v\wedge w)\mathds{1},~ [\alpha\wedge\alpha',v]=2\iota_v(\alpha\wedge\alpha'), \\
[\alpha,v\wedge w\otimes u]=-\alpha(u)v\wedge w,~[\alpha\wedge\alpha',v\wedge w\otimes u]=\alpha\wedge\alpha'(v\wedge w)\, u, \\
[\alpha\wedge\alpha'\otimes\beta,v]=-\beta(v)\alpha\wedge\alpha',~ [\alpha\wedge\alpha'\otimes\beta,v\wedge w]=\alpha\wedge\alpha'(v\wedge w)\,\beta, \\
[\alpha\wedge\alpha'\otimes\beta,v\wedge w\otimes u]=-\alpha\wedge\alpha'(v\wedge w)\,\beta\otimes u,\end{gather*}
together with the adjoint action of $\mathfrak{g}_0\approx\mathfrak{gl}_2\mathbb{R}$ on each $\mathfrak{g}_i$ corresponding to the action of $\mathfrak{gl}_2\mathbb{R}$ on $\mathfrak{g}_i$ as a $\GLin_2\mathbb{R}$-representation. Here, we write $\alpha\wedge\alpha'(v\wedge w):=\alpha(v)\alpha'(w)-\alpha(w)\alpha'(v)$, $\iota_v(\alpha\wedge\alpha'):=\alpha(v)\alpha'-\alpha'(v)\alpha$, and $\iota_\alpha(v\wedge w):=\alpha(v)w-\alpha(w)v$.
\end{definition}

Since this presentation of the Lie algebra $\mathfrak{g}$ is a bit overwhelming to take in all at once, it is instructive to extend the diagrams we built in Figures \ref{gminus_roots} and \ref{gminusandg0} to include the new positive grading components, so that we can keep track of what the bracket does. The result of this, shown in Figure \ref{bracket_rootdiagram}, is visibly just the root system of type $\mathrm{G}_2$.

\begin{figure}[h]
\centering\includegraphics[width=0.65\textwidth]{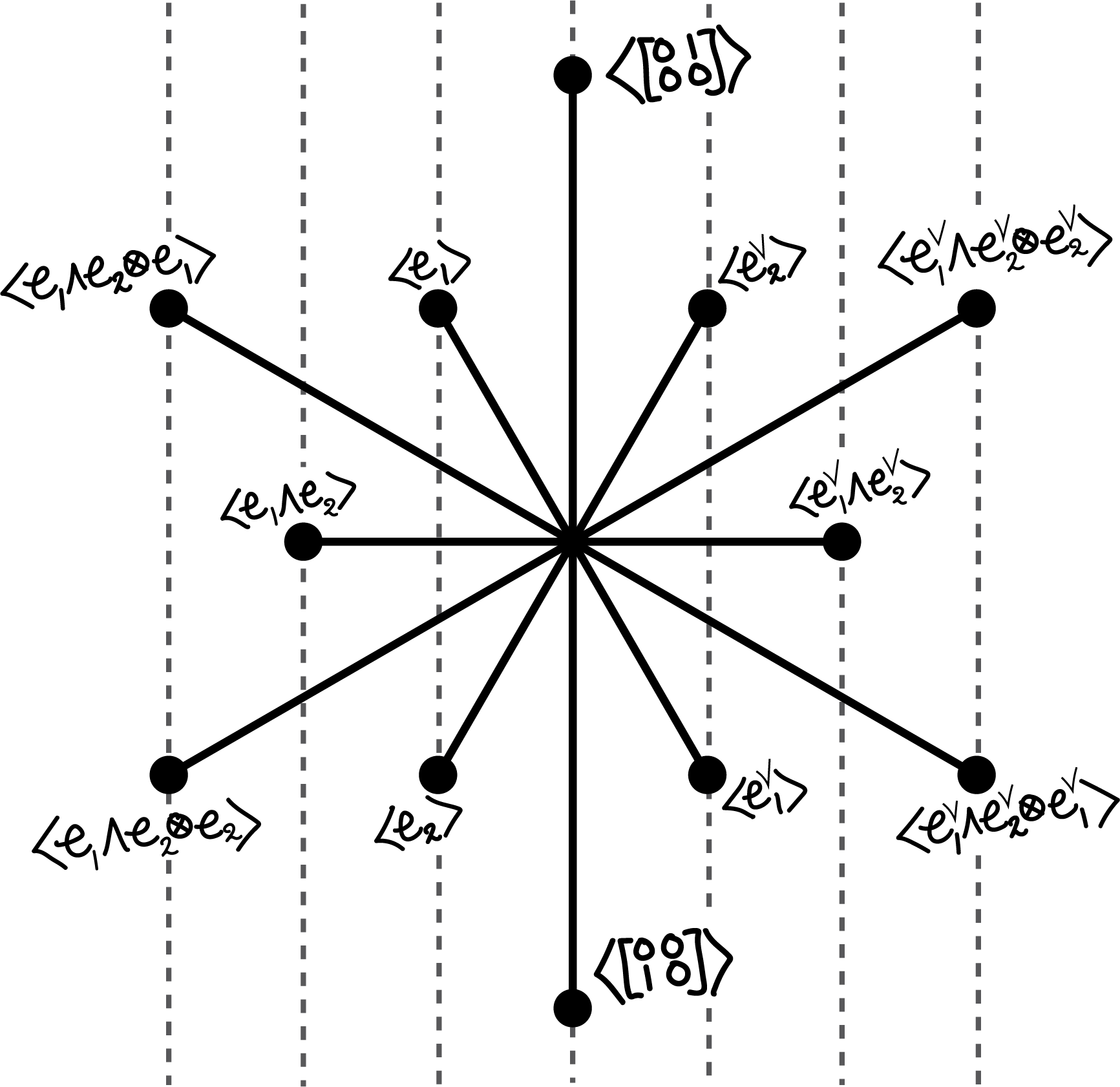}
\caption{A diagram to help keep track of the bracket on $\mathfrak{g}$, with each grading component indicated by a dashed gray vertical line through the corresponding nodes}
\label{bracket_rootdiagram}
\end{figure}

Whenever we have a Lie algebra $\mathfrak{s}$ of local infinitesimal symmetries for a $(2,3,5)$-distribution $D$ on a 5-manifold $M$, it must embed---as a subspace, though not necessarily as a subalgebra---into this prolonged Lie algebra $\mathfrak{g}$ that we have just constructed. To see this, we build a natural filtration on $\mathfrak{s}$ as follows: fixing a point $x\in M$, we define
\[\mathfrak{s}^0:=\{\eta\in\mathfrak{s}:\eta_x=0\}=\{\eta\in\mathfrak{s}:\eta_x\in D^{(0)}_x\}\]
and, for each $i>0$,
$\mathfrak{s}^{-i}:=\{\eta\in\mathfrak{s}:\eta_x\in D^{(i)}_x\}$ and
\[\mathfrak{s}^i:=\{\eta\in\mathfrak{s}^0:[\eta,D^{(j)}]_x\subseteq D^{(j-i)}_x\text{ for each }j\geq 0\}.\]
This makes $\mathfrak{s}$ into a filtered Lie algebra, which allows us to define the \emph{graded image} $\mathrm{gr}(\mathfrak{s})$ of $\mathfrak{s}$ to be the graded Lie algebra with grading components $\mathrm{gr}_i(\mathfrak{s}):=\mathfrak{s}^i/\mathfrak{s}^{i+1}$ and bracket given by, for $\eta+\mathfrak{s}^{i+1}\in\mathrm{gr}_i(\mathfrak{s})$ and $\eta'+\mathfrak{s}^{j+1}\in\mathrm{gr}_j(\mathfrak{s})$,
\[[\eta+\mathfrak{s}^{i+1},\eta'+\mathfrak{s}^{j+1}]:=[\eta,\eta']+\mathfrak{s}^{i+j+1}\in\mathrm{gr}_{i+j}(\mathfrak{s}).\]

\begin{proposition}[see also Theorem 8.4 of \cite{Tanaka1970}]\label{gradedimageprop} If $\mathfrak{s}$ is a Lie algebra of local infinitesimal symmetries as above for a $(2,3,5)$-distribution $D$ on a connected 5-manifold $M$, then $\mathrm{gr}(\mathfrak{s})$ embeds as a graded subalgebra of the graded Lie algebra $\mathfrak{g}=\sum_{i=-3}^3\mathfrak{g}_i$ from Definition \ref{tanakaprolongation}.\end{proposition}
\begin{proof}By definition, $\mathfrak{s}^{-i}_x=\{\eta_x:\eta\in\mathfrak{s}^{-i}\}\subseteq D^{(i)}_x$, so for each $i>0$,
\[\mathrm{gr}_{-i}(\mathfrak{s})\approx\mathfrak{s}^{-i}_x/\mathfrak{s}^{-i+1}_x=\mathfrak{s}^{-i}_x/\mathfrak{s}^{-i}_x\cap D^{(i-1)}_x\]
embeds into $\mathfrak{g}_{-i}\approx D^{(i)}_x/D^{(i-1)}_x$. Since the bracket is the one induced by the Lie bracket of vector fields, this makes $\mathrm{gr}_{-3}(\mathfrak{s})+\mathrm{gr}_{-2}(\mathfrak{s})+\mathrm{gr}_{-1}(\mathfrak{s})$ into a graded subalgebra of the symbol algebra $\mathfrak{g}_-$.

Because $\mathfrak{s}$ satisfies $[\mathfrak{s},D]\subseteq D$ by virtue of being a Lie algebra of local infinitesimal symmetries, the Lie bracket of vector fields induces a representation of $\mathfrak{s}^0$ on $D_x$ with kernel $\mathfrak{s}^1$, so that the resulting quotient $\mathrm{gr}_0(\mathfrak{s})=\mathfrak{s}^0/\mathfrak{s}^1$ naturally embeds into $\mathfrak{gl}(D_x)\approx\mathfrak{g}_0$. This is induced by the Lie bracket as well, so we similarly get that $\sum_{i\leq 0}\mathrm{gr}_i(\mathfrak{s})$ is a graded subalgebra of $\mathfrak{g}_-+\mathfrak{g}_0$.

Likewise, each $\eta\in\mathfrak{s}^1$ determines a map $\tilde{\eta}:\mathfrak{g}_-\to\mathfrak{g}_-+\mathfrak{g}_0$ from the graded tangent space $\mathrm{gr}(T_xM)\approx\mathfrak{g}_-$ to $\mathfrak{g}_-+\mathfrak{g}_0$ via the Lie bracket, with the map from $D_x\approx\mathfrak{g}_{-1}$ to $\mathfrak{gl}(D_x)\approx\mathfrak{g}_0$ given by $X_x\mapsto[\eta,X_x]$, where $[\eta,X_x](Y_x):=[[\eta,X],Y]_x$ for vector fields $X$ and $Y$ in $D$; this only depends on the values of $X$ and $Y$ at $x$, since for $f_1,f_2\in C^\infty(M)$,
\begin{align*}[[\eta,f_1X],f_2Y]_x & =[\eta(f_1)X+f_1[\eta,X],f_2Y]_x \\ & =[\eta(f_1)X,f_2Y]_x+[f_1[\eta,X],f_2Y]_x \\ & =f_1(x)f_2(x)[[\eta,X],Y]_x.\end{align*}
This map $\tilde{\eta}$ determined by $\eta\in\mathfrak{s}^1$ shifts the grading up by $1$, sending each $\mathfrak{g}_{-i}$ to $\mathfrak{g}_{-i+1}$, and is naturally a derivation, since it is given by acting by $\eta$ via the Lie bracket, so $\tilde{\eta}\in\mathfrak{g}_1$. The kernel of the map $\eta\mapsto\tilde{\eta}$ is $\mathfrak{s}^2$, so we get an embedding of $\mathrm{gr}_1(\mathfrak{s})=\mathfrak{s}^1/\mathfrak{s}^2$ into $\mathfrak{g}_1$. Doing the analogous thing for each $\eta\in\mathfrak{s}^2$, defining $\tilde{\eta}:\mathfrak{g}_-\to\mathfrak{g}_-+\mathfrak{g}_0+\mathfrak{g}_1$ with $\tilde{\eta}|_{\mathfrak{g}_{-2}}:\mathfrak{g}_{-2}\to\mathfrak{g}_0$ given by $X_x\mapsto[\eta,X_x]$ and $\tilde{\eta}|_{\mathfrak{g}_{-1}}:\mathfrak{g}_{-1}\to\mathfrak{g}_1$ given by $X_x\mapsto(Y_x\mapsto(Z_x\mapsto[[[\eta,X],Y],Z]_x))$, gives an embedding of $\mathrm{gr}_2(\mathfrak{s})$ into $\mathfrak{g}_2$, and similarly, using the Lie bracket to construct a derivation $\tilde{\eta}:\mathfrak{g}_-\to\mathfrak{g}_-+\mathfrak{g}_0+\mathfrak{g}_1+\mathfrak{g}_2$ for each $\eta\in\mathfrak{s}^3$ gives an embedding of $\mathrm{gr}_3(\mathfrak{s})$ into $\mathfrak{g}_3$ as well. In short, we get an embedding of $\mathrm{gr}(\mathfrak{s})$ as a graded subalgebra of $\mathfrak{g}$, as desired.\mbox{\qedhere}\end{proof}

In particular, this makes $\mathfrak{g}$ the maximal possible graded Lie algebra of local infinitesimal symmetries for a $(2,3,5)$-distribution, since all other such graded Lie algebras (with grading compatible with the filtration) must embed into it. Moreover, it has maximal dimension among (not necessarily graded) Lie algebras of local infinitesimal symmetries, since $\mathrm{gr}(\mathfrak{s})$ is isomorphic to $\mathfrak{s}$ as a vector space and embeds into $\mathfrak{g}$.

It is worth noting, however, that $\mathrm{gr}(\mathfrak{s})$ is generally not isomorphic to $\mathfrak{s}$ as a Lie algebra. Take, for example, the Lie algebra $\mathfrak{so}(3)\oplus\mathfrak{so}(3)$ for $D_{r_0/r_1}$ in the rolling spheres examples: since $\mathrm{S}(\Orth(3)\times\Orth(3))$ is acting transitively there, the vector fields from $\mathfrak{so}(3)\oplus\mathfrak{so}(3)$ will span the tangent space at each point, so that $\mathfrak{g}_-$ must be a graded subalgebra of $\mathrm{gr}(\mathfrak{so}(3)\oplus\mathfrak{so}(3))$, and since the stabilizer subalgebra here is isomorphic to $\mathfrak{o}(2)$,
\[\mathrm{gr}(\mathfrak{so}(3)\oplus\mathfrak{so}(3))\approx\mathfrak{g}_-+\mathfrak{o}(2)<\mathfrak{g}_-+\mathfrak{g}_0,\]
with $\mathfrak{o}(2)$ embedded in $\mathfrak{g}_0\approx\mathfrak{gl}_2\mathbb{R}$. But this means that $\mathfrak{g}_-$ is a (solvable) ideal in $\mathrm{gr}(\mathfrak{so}(3)\oplus\mathfrak{so}(3))$, so $\mathrm{gr}(\mathfrak{so}(3)\oplus\mathfrak{so}(3))$ cannot be isomorphic to the semisimple Lie algebra $\mathfrak{so}(3)\oplus\mathfrak{so}(3)$.

\vspace{1em}
\subsection{The structure of the prolonged Lie algebra \texorpdfstring{$\mathfrak{g}$}{\mathfrak{g}}}\label{structurestuff}\hfill\\
In this subsection, we will prove the following result.
\begin{theorem}The Lie algebra $\mathfrak{g}=\sum_{i=-3}^3\mathfrak{g}_i$ constructed as the Tanaka prolongation of $\mathfrak{g}_-+\mathfrak{g}_0$ above is the split real form of the exceptional simple Lie algebra of type $\mathrm{G}_2$.\end{theorem}

From the description of the bracket given in Definition \ref{tanakaprolongation} above, we can compute that the Killing form $\kgf$ of the Lie algebra $\mathfrak{g}$ satisfies, for $v,w,u\in\mathfrak{g}_{-1}\approx\mathbb{R}^2$, $\alpha,\alpha',\beta\in\mathfrak{g}_1\approx(\mathbb{R}^2)^\vee$, and $R,R'\in\mathfrak{g}_0\approx\mathfrak{gl}_2\mathbb{R}$,
\begin{gather*}\kgf(\alpha\wedge\alpha'\otimes\beta,v\wedge w\otimes u)=8\alpha\wedge\alpha'(v\wedge w)\beta(u), \\ \kgf(\alpha\wedge\alpha',v\wedge w)=24\alpha\wedge\alpha'(v\wedge w), \\ \kgf(\alpha,v)=24\alpha(v),\text{ and} \\ \kgf(R,R')=8(\mathrm{tr}(RR')+\mathrm{tr}(R)\mathrm{tr}(R'));\end{gather*}
note that, since $\mathfrak{g}=\sum_{i=-3}^3\mathfrak{g}_i$ is a graded Lie algebra, if $X\in\mathfrak{g}_i$ and $Y\in\mathfrak{g}_j$, then $\ad_X\circ\ad_Y(\mathfrak{g}_\ell)\subseteq\mathfrak{g}_{i+j+\ell}$, so $\kgf(X,Y)=0$ unless $i+j=0$. Because the computation shows that each grading component $\mathfrak{g}_i$ is dual to $\mathfrak{g}_{-i}$ with respect to the Killing form, $\kgf$ is nondegenerate on $\mathfrak{g}$, hence $\mathfrak{g}$ is semisimple.

\newpage
Now, consider the involutive Lie algebra automorphism $\theta\in\Aut(\mathfrak{g})$ determined by, for $\{e_1,e_2\}$ our basis for $\mathfrak{g}_{-1}\approx\mathbb{R}^2$ with dual basis $\{e_1^\vee,e_2^\vee\}$ for $\mathfrak{g}_1\approx(\mathbb{R}^2)^\vee$ and $R\in\mathfrak{g}_0\approx\mathfrak{gl}_2\mathbb{R}$,
\begin{gather*}\theta(e_j)=-e_j^\vee, \\ \theta(e_1\wedge e_2)=-e_1^\vee\wedge e_2^\vee, \\ \theta(e_1\wedge e_2\otimes e_j)=-e_1^\vee\wedge e_2^\vee\otimes e_j^\vee, \\ \text{and }\theta(R)=-R^\top,\end{gather*}
where $R^\top$ denotes the matrix transpose of $R$. We can see that $\theta$ satisfies $\theta(\mathfrak{g}_i)=\mathfrak{g}_{-i}$ for each $i$, and moreover, that
\begin{gather*}\kgf(v_1e_1+v_2e_2,\theta(v_1e_1+v_2e_2))=24(-v_1^2-v_2^2)=-24(v_1^2+v_2^2), \\ \kgf(e_1\wedge e_2,\theta(e_1\wedge e_2))=-24, \\ \kgf(e_1\wedge e_2\otimes(u_1e_1+u_2e_2),\theta(e_1\wedge e_2\otimes(u_1e_1+u_2e_2)))=-8(u_1^2+u_2^2),\end{gather*}
and
\begin{align*}\kgf(R,\theta(R)) & =8(\mathrm{tr}(R(-R^\top))+\mathrm{tr}(R)\mathrm{tr}(-R^\top)) \\ & =-8(\mathrm{tr}(R^\top R)+\mathrm{tr}(R)^2).\end{align*}
In each of these cases, we see that $\kgf(X,\theta(X))<0$ unless $X=0$, so $\theta$ is, in fact, a Cartan involution for $\mathfrak{g}$.

Knowing $\theta$, we then get a corresponding Cartan decomposition of $\mathfrak{g}$ as the sum of the Lie algebra $\mathfrak{k}:=\{X\in\mathfrak{g}:\theta(X)=X\}$ and its $\kgf$-perpendicular subspace
\[\mathfrak{k}^\perp:=\{X\in\mathfrak{g}:\kgf(X,\mathfrak{k})=\{0\}\}=\{X\in\mathfrak{g}:\theta(X)=-X\}.\]
The subalgebra $\mathfrak{a}=\langle\begin{smallbmatrix}1 & 0 \\ 0 & 0\end{smallbmatrix},\begin{smallbmatrix}0 & 0 \\ 0 & 1\end{smallbmatrix}\rangle$ of diagonal matrices in $\mathfrak{g}_0\approx\mathfrak{gl}_2\mathbb{R}$ is contained in $\mathfrak{k}^\perp$ and equal to its own centralizer in $\mathfrak{g}$, so (following the terminology from \cite{CapSlovakPG1}) $\mathfrak{a}$ is a $\theta$-stable maximally noncompact Cartan subalgebra of $\mathfrak{g}$. As such, we are dealing with a split real semisimple Lie algebra $\mathfrak{g}$ of real rank $\dim(\mathfrak{a})=2$.

To identify which split semisimple Lie algebra $\mathfrak{g}$ is, the natural next step is to construct its (restricted) root system. For this, we find it convenient to identify the subalgebra $\mathfrak{g}_{-3}+\mathfrak{g}_0+\mathfrak{g}_3$ with $\mathfrak{sl}_3\mathbb{R}$ via
\[e_1\wedge e_2\otimes u+R+e_1^\vee\wedge e_2^\vee\otimes\beta\mapsto\begin{smallbmatrix}-\mathrm{tr}(R) & \beta \\ u & R\end{smallbmatrix};\]
under this isomorphism, $\mathfrak{a}$ identifies with the Cartan subalgebra of $\mathfrak{sl}_3\mathbb{R}$ given by trace-free diagonal matrices, so every (restricted) root of $\mathfrak{g}$ corresponds to a weight of $\mathfrak{sl}_3\mathbb{R}$. Let us define $\varepsilon_1(\begin{smallbmatrix}s & 0 \\ 0 & t\end{smallbmatrix}):=-s-t$, $\varepsilon_2(\begin{smallbmatrix}s & 0 \\ 0 & t\end{smallbmatrix}):=s$, and $\varepsilon_3(\begin{smallbmatrix}s & 0 \\ 0 & t\end{smallbmatrix}):=t$, so that under the identification of $\mathfrak{g}_{-3}+\mathfrak{g}_0+\mathfrak{g}_3$ with $\mathfrak{sl}_3\mathbb{R}$, we get
\[\varepsilon_i\left(\begin{smallbmatrix}\lambda_1 & 0 & 0 \\ 0 & \lambda_2 & 0 \\ 0 & 0 & \lambda_3\end{smallbmatrix}\right)=\lambda_i.\]
Then, writing $\mathfrak{g}_\gamma$ for the root space with root $\gamma\in\mathfrak{a}^\vee$, the root spaces of $\mathfrak{g}$ are as follows:
\begin{gather*}\mathfrak{g}_{\varepsilon_2-\varepsilon_1}=\langle e_1\wedge e_2\otimes e_1\rangle, \mathfrak{g}_{\varepsilon_3-\varepsilon_1}=\langle e_1\wedge e_2\otimes e_2\rangle, \\ \mathfrak{g}_{-\varepsilon_1}=\langle e_1\wedge e_2\rangle, \\ \mathfrak{g}_{\varepsilon_2}=\langle e_1\rangle, \mathfrak{g}_{\varepsilon_3}=\langle e_2\rangle, \\ \mathfrak{g}_{\varepsilon_3-\varepsilon_2}=\langle\begin{smallbmatrix}0 & 0 \\ 1 & 0\end{smallbmatrix}\rangle, \mathfrak{g}_{\varepsilon_2-\varepsilon_3}=\langle\begin{smallbmatrix}0 & 1 \\ 0 & 0\end{smallbmatrix}\rangle, \\ \mathfrak{g}_{-\varepsilon_2}=\langle e_1^\vee\rangle, \mathfrak{g}_{-\varepsilon_3}=\langle e_2^\vee\rangle, \\ \mathfrak{g}_{\varepsilon_1}=\langle e_1^\vee\wedge e_2^\vee\rangle, \\ \mathfrak{g}_{\varepsilon_1-\varepsilon_2}=\langle e_1^\vee\wedge e_2^\vee\otimes e_1^\vee\rangle,\text{ and }\mathfrak{g}_{\varepsilon_1-\varepsilon_3}=\langle e_1^\vee\wedge e_2^\vee\otimes e_2^\vee\rangle.\end{gather*}
Each of these restricted roots precisely corresponds to one of the nodes of the bracket diagram from Figure \ref{bracket_rootdiagram}, and indeed, that is because the bracket diagram we constructed is essentially just the root diagram in disguise. We show the relabeled root diagram in Figure \ref{relabeledrootdiagram}. This is, of course, just the root diagram of type $\mathrm{G}_2$, so $\mathfrak{g}$ is the split real form of the exceptional simple Lie algebra of type $\mathrm{G}_2$.

\begin{figure}[h]
\centering\includegraphics[width=0.55\textwidth]{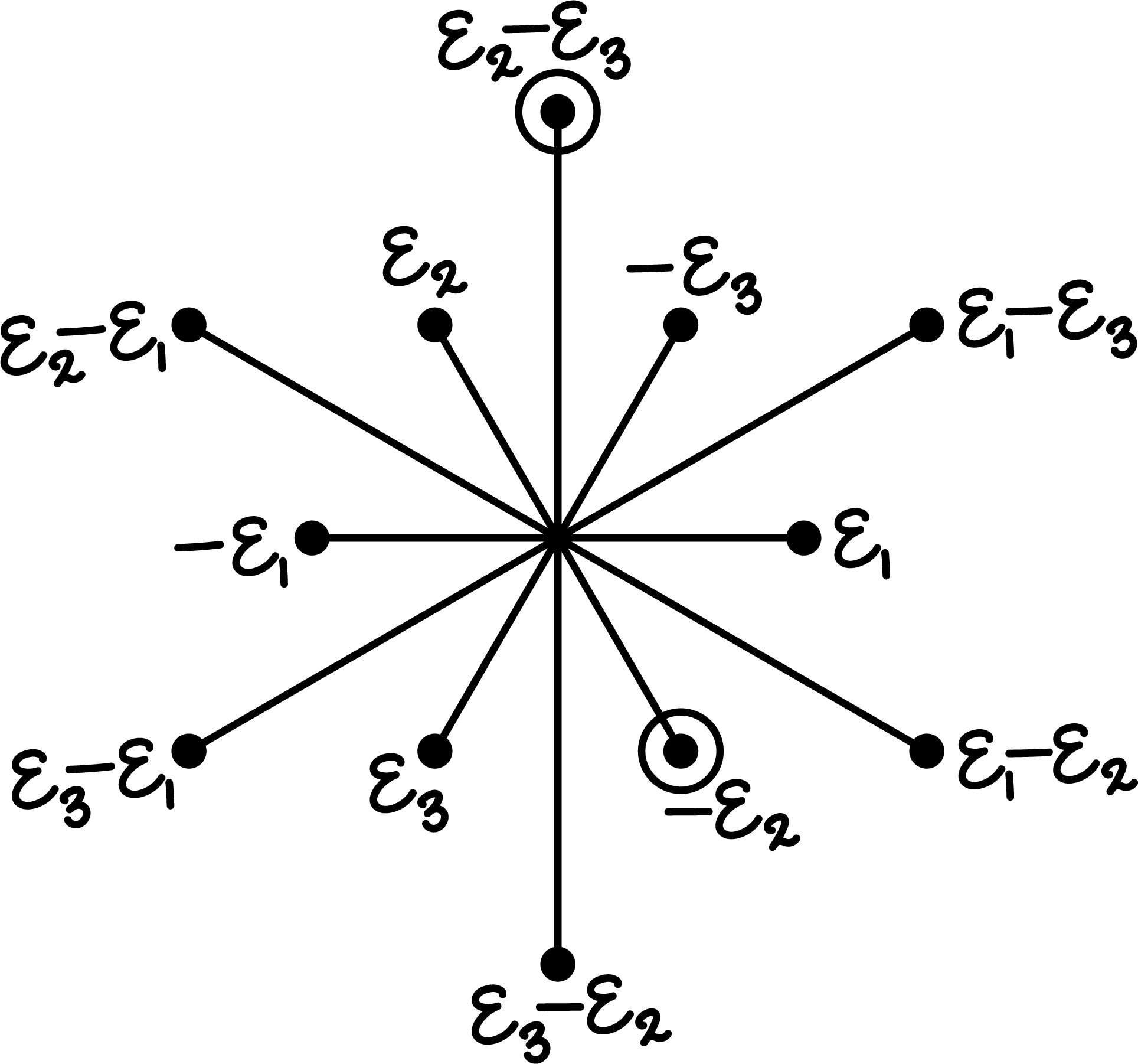}
\caption{The root diagram for $\mathfrak{g}$, with the roots $-\varepsilon_2$ and $\varepsilon_2-\varepsilon_3$ circled as simple roots}
\label{relabeledrootdiagram}
\end{figure}

In the sequel, we will make convenient use of the following result on the Lie group $\Aut(\mathfrak{g})$ of Lie algebra automorphisms of $\mathfrak{g}$.

\begin{proposition}\label{innerauto} Every automorphism of $\mathfrak{g}$ is an inner automorphism, meaning that $\exp(\ad_\mathfrak{g})$ generates $\Aut(\mathfrak{g})$.\end{proposition}
\begin{proof}Suppose that $\phi\in\Aut(\mathfrak{g})$ and consider the image $\phi(\mathfrak{a})$ of $\mathfrak{a}$ under $\phi$. Because $\phi(\mathfrak{a})$ is another maximally noncompact Cartan subalgebra and all such Cartan subalgebras are conjugate, there must be some inner automorphism, which we write as $\Ad_g$, such that $\Ad_g(\phi(\mathfrak{a}))=\mathfrak{a}$. Therefore, for $Z\in\mathfrak{a}$ and a root vector $\eta_\nu\in\mathfrak{g}_\nu$, we have
\begin{align*}[Z,(\Ad_g\!\circ\,\phi)(\eta_\nu)] & =(\Ad_g\!\circ\,\phi)\left([\phi^{-1}(\Ad_g^{-1}(Z)),\eta_\nu]\right) \\ & =\nu\left(\Ad_g^{-1}\circ\,\phi^{-1}(Z)\right)(\Ad_g\circ\,\phi)(\eta_\nu).\end{align*}
This means that the action of $\Ad_g\!\circ\phi$ on $\mathfrak{a}^\vee$ must permute the roots, and in particular, it takes each (unordered) choice of simple roots to another (unordered) choice of simple roots. The inner automorphisms already act transitively on the sets of simple roots via the action of the Weyl group, so by changing our choice of $\Ad_g$, we may assume that $\Ad_g\circ\phi$ preserves a set of simple roots. Such an automorphism can only be nontrivial if it nontrivially permutes the simple roots. However, there are only two simple roots---one short and one long---and no isometry can swap the short and long roots, so $\Ad_g\circ\phi=\mathrm{id}_\mathfrak{g}$. Thus, every element of $\Aut(\mathfrak{g})$ is an inner automorphism.\mbox{\qedhere}\end{proof}

\subsection{The model geometry for \texorpdfstring{$(2,3,5)$}{(2,3,5)}-distributions}\hfill \\
By construction, $\mathfrak{g}$ is the Lie algebra of local infinitesimal symmetries for the natural $(2,3,5)$-distribution in the model geometry $(G_-G_0,G_0)$. However, integrating these local infinitesimal symmetries on the model space $G_-\cong G_-G_0/G_0$ only gives us local symmetries; except for those coming from $\mathfrak{g}_-+\mathfrak{g}_0$, they do not integrate to 1-parameter subgroups of global symmetries for the distribution. We would like to expand $G_-G_0$ to a larger Lie group that can include all of these local symmetries from $\mathfrak{g}$ as global symmetries for some other space.

\begin{definition}For $\mathfrak{g}$ the split real form of the Lie algebra of type $\mathrm{G}_2$ described above, we define $G=\Zhe_2$ to be the adjoint simple Lie group with Lie algebra $\mathfrak{g}$, meaning the unique centerless connected Lie group with this Lie algebra.\end{definition}

By Proposition \ref{innerauto} in the previous section, the adjoint representation $\Ad:\Zhe_2\to\Aut(\mathfrak{g})$ is a Lie group isomorphism in this case.

Of course, $\Zhe_2$ is our desired candidate for the Lie group integrating all of the local symmetries from $\mathfrak{g}$ to global ones, and as expected, $G_-G_0$ naturally embeds into $G=\Zhe_2$. Indeed, note that the adjoint action of $\mathfrak{g}_-$ on $\mathfrak{g}$ extends its adjoint action on $\mathfrak{g}_-+\mathfrak{g}_0$, so because the adjoint action of $G_-=\exp(\mathfrak{g}_-)<G_-G_0$ on $\mathfrak{g}_-+\mathfrak{g}_0$ is faithful, we may identify the simply connected nilpotent Lie group $G_-$ with the image of $\mathfrak{g}_-<\mathfrak{g}$ under the exponential map in $\Zhe_2\simeq\Aut(\mathfrak{g})$. Likewise, since we constructed $\mathfrak{g}$ as a graded representation of $G_0\simeq\GLin_2\mathbb{R}$, we already have a natural identification of $G_0$ with a subgroup of $\Aut(\mathfrak{g})$ via this action.

\begin{proposition}Every element of $\Aut_\mathrm{gr}(\mathfrak{g}_-)$ canonically extends to a unique element of $\Aut_\mathrm{gr}(\mathfrak{g}):=\{\phi\in\Aut(\mathfrak{g}):\phi(\mathfrak{g}_i)=\mathfrak{g}_i\text{ for each }i\}$. As such, we may identify $G_0:=\Aut_\mathrm{gr}(\mathfrak{g}_-)\simeq\GLin_2\mathbb{R}$ with the preimage of $\Aut_\mathrm{gr}(\mathfrak{g})$ under the isomorphism $\Ad:\Zhe_2\to\Aut(\mathfrak{g})$.\end{proposition}
\begin{proof}By construction, $\mathfrak{g}=\mathfrak{g}_{-3}+\mathfrak{g}_{-2}+\mathfrak{g}_{-1}+\mathfrak{g}_0+\mathfrak{g}_1+\mathfrak{g}_2+\mathfrak{g}_3$ is a graded representation of $G_0\simeq\GLin_2\mathbb{R}$, with $G_0$ acting on each grading component. Each of the bracket operations included in Definition \ref{tanakaprolongation} is naturally $G_0$-equivariant: for example, for $A\in G_0\simeq\GLin_2\mathbb{R}$, we have
\begin{align*}[A\cdot\alpha,A\cdot v] & =[\alpha\circ A^{-1},Av]=(\alpha\circ A^{-1})(Av)\mathds{1}-3(\alpha\circ A^{-1})\otimes Av \\ & =A(\alpha(v)\mathds{1}-3\alpha\otimes v)A^{-1}=A\cdot[\alpha,v]\end{align*}
for all $\alpha\in\mathfrak{g}_1\approx(\mathbb{R}^2)^\vee$ and $v\in\mathfrak{g}_{-1}\approx\mathbb{R}^2$. Therefore, this faithful action of $G_0=\Aut_\text{gr}(\mathfrak{g}_-)$, which extends the action of $G_0$ on $\mathfrak{g}_-$ to make $\mathfrak{g}_-$ a $G_0$-subrepresentation of $\mathfrak{g}$, determines an embedding of $G_0$ into $\Aut_\text{gr}(\mathfrak{g})$.

To show that this extension of elements from $\Aut_\text{gr}(\mathfrak{g}_-)$ to $\Aut_\text{gr}(\mathfrak{g})$ is unique, suppose that $\phi\in\Aut_\text{gr}(\mathfrak{g})$ restricts to the identity on $\mathfrak{g}_-$. For $R\in\mathfrak{g}_0$, we must have
\[\phi(R)v=[\phi(R),v]=[\phi(R),\phi(v)]=\phi([R,v])=\phi(Rv)=Rv\]
for all $v\in\mathfrak{g}_{-1}$, so $\phi$ must also restrict to the identity on $\mathfrak{g}_0$. Iteratively, since the elements $\eta$ of each nonnegative grading component $\mathfrak{g}_i$ are---by construction---uniquely determined by the restriction
\[\ad_\eta|_{\mathfrak{g}_{-1}}:\mathfrak{g}_{-1}\to\mathfrak{g}_{i-1}\]
of their adjoint action to $\mathfrak{g}_{-1}$, we can see that if $\phi|_{\mathfrak{g}_{i-1}}$ is the identity, then
\[[\phi(\eta),v]=[\phi(\eta),\phi(v)]=\phi([\eta,v])=[\eta,v]\]
for each $v\in\mathfrak{g}_{-1}$, so $\phi|_{\mathfrak{g}_i}$ is the identity too. Thus, $\phi$ restricts to the identity on each grading component, so $\phi=\mathrm{id}_\mathfrak{g}$.\mbox{\qedhere}\end{proof}

The elements of $\mathfrak{g}$ that generate local symmetries fixing the identity element of $G_-\cong G_-G_0/G_0$ are the local infinitesimal symmetries that vanish there. These elements are precisely the sums of elements from the nonnegative grading components of $\mathfrak{g}$, which form a subalgebra
\[\mathfrak{p}:=\mathfrak{g}_0+\mathfrak{g}_1+\mathfrak{g}_2+\mathfrak{g}_3.\]
Since we want $\Zhe_2$ to extend the action of $G_-G_0$ in a way that makes the 1-parameter subgroups from $\mathfrak{g}$ into global symmetries, this subalgebra $\mathfrak{p}$ should be the Lie algebra of the isotropy subgroup $P$ for a point of our new homogeneous space. Knowing that $G_0$ already determines global symmetries in $G_-G_0$ that fix the identity element in $G_-$, the subgroup $G_0<\Zhe_2$ should be contained in $P$. Moreover, the nilpotent subalgebra
\[\mathfrak{p}_+:=\mathfrak{g}_1+\mathfrak{g}_2+\mathfrak{g}_3=\theta(\mathfrak{g}_-)\]
determines a simply connected nilpotent Lie subgroup
\[P_+:=\exp(\mathfrak{p}_+)=\exp(\theta(\mathfrak{g}_-))\]
conjugate to $G_-=\exp(\mathfrak{g}_-)$ in $\Zhe_2$, and we need $P_+$ to be a subgroup of $P$ because $\mathfrak{p}_+<\mathfrak{p}$. Putting these together generates our new isotropy subgroup.

\begin{definition}We will refer to $(\Zhe_2,P)$ as the \emph{model geometry for $(2,3,5)$-distributions}, where $P:=G_0P_+<\Zhe_2$ is the closed subgroup generated by $G_0$ and $P_+:=\exp(\mathfrak{p}_+)$.\end{definition}

As our choices of notation above have hopefully foreshadowed, this subgroup $P$ is a parabolic subgroup of $\Zhe_2$, since $\mathfrak{p}^\perp=\mathfrak{p}_+$ is a nilpotent subalgebra of $\mathfrak{g}$, and $G_0$ is a Levi subgroup of $P$ in the usual sense. Pictorially, this gives us the kind of description of $(\Zhe_2,P)$ that we get for all parabolic geometries in terms of open Schubert cells: each $g\in\Zhe_2$ determines an open Schubert cell $\quot{P}\!(gG_-G_0)$ on $\Zhe_2/P$ imbued with the geometric structure of $(G_-G_0,G_0)$, with right-translation by elements of $G_0$ corresponding to linearly changing the frame within the distribution $D$ and right-translation by elements of $P_+$ amounting to changing the choice of open Schubert cell through the underlying point $\quot{P}(g)$. This makes $\Zhe_2/P$ into a compactification of $G_-$ in such a way that the local symmetries from $\mathfrak{g}$ extend to global symmetries, exactly as we wanted.

On each of these open Schubert cells, which we should again think of as analogues of affine charts for this model geometry, we get a natural $(2,3,5)$-distribution $D$ from the geometry of $(G_-G_0,G_0)$. Using the $\Ad_P$-invariant filtration of $\mathfrak{g}$ given by successively adding up the grading components, so that $\mathfrak{g}^i:=\sum_{j\geq i}\mathfrak{g}_j$, we can simultaneously extend all of these distributions to a global $\Zhe_2$-invariant $(2,3,5)$-distribution on $\Zhe_2/P$ by redefining
\[D:=\mathrm{q}_{{}_P*}\MC{\Zhe_2}^{-1}(\mathfrak{g}^{-1})=\mathrm{q}_{{}_P*}\MC{\Zhe_2}^{-1}(\mathfrak{g}_{-1}+\mathfrak{g}_0+\mathfrak{g}_1+\mathfrak{g}_2+\mathfrak{g}_3).\]
This new $D$ restricts to the old one on each open Schubert cell because $\mathfrak{g}^{-1}=\mathfrak{g}_{-1}+\mathfrak{p}$ and $\quot{P}$ restricts to the quotient map $\quot{G_0}$ for each open Schubert cell. Thus, the model geometry $(\Zhe_2,P)$ carries a canonical $(2,3,5)$-distribution determined by the Lie-theoretic structure, and this distribution extends the natural $(2,3,5)$-distribution on $G_-$.

\section{Proof of the 1:3 ratio theorem}\label{proofsection}
In this section, we will prove the main theorem from the introduction\linebreak by relating the model $(\Zhe_2,P)$ that we have just constructed to the geometry of rolling spheres $(\mathrm{S}(\Orth(3)\times\Orth(3)),\Orth(2))$ from Definition \ref{rollingsph}. Similar to \cite{BorMontgomery2009}, the core idea of the proof is to reframe the geometries in terms of a maximal compact subgroup $K<\Zhe_2$, described in the first subsection below. As we will see in Section \ref{comparison}, the geometry induced by the action of $K$ on $\Zhe_2/P$ happens to be locally isomorphic to the geometry of rolling spheres, and this will allow us to single out certain choices of rolling distribution $D_{r_0/r_1}$ with additional local symmetry.

Setting up and demonstrating this local equivalence will require a bit of Lie theory
. To compare the geometries, we will specify a favorable form for a covering homomorphism $\rho$ from $\Spin(3)\times\Spin(3)$ to $K$ in Section \ref{rhoform}. Then, in Section \ref{intersecttwice}, we will see from the root diagram of $\mathfrak{g}$ that two particular one-parameter subgroups of $\Zhe_2$---corresponding to the ones described in the introduction---must intersect exactly twice, which further pins down the form of $\rho$. Finally, we will use $\rho$ to describe the local isomorphism between the geometries and pick out the 1:3 ratio in Section \ref{comparison}, then finish the proof by showing that other ratios of radii do not give additional local symmetry in Section \ref{otherratios}.

\subsection{The maximal compact subgroup of \texorpdfstring{$\Zhe_2$}{Zhe2}}\hfill \\
Recall the Cartan involution $\theta$ that we defined in Section \ref{structurestuff} above. Using the isomorphism $\Ad:\Zhe_2\to\Aut(\mathfrak{g})$, we may identify $\theta\in\Aut(\mathfrak{g})$ with its preimage $\hat{\theta}:=\Ad^{-1}(\theta)\in\Zhe_2$, so that conjugation by $\hat{\theta}\in\Zhe_2$ gives a Lie group automorphism of $\Zhe_2$ that induces $\theta=\Ad_{\hat{\theta}}\in\Aut(\mathfrak{g})$. From this, we get a maximal compact subgroup $K<\Zhe_2$ given by the centralizer
\[K:=\mathrm{Z}_{\Zhe_2}(\hat{\theta})=\{k\in\Zhe_2:\hat{\theta} k\hat{\theta}^{-1}=k\},\]
with corresponding Lie subalgebra $\mathfrak{k}:=\{X\in\mathfrak{g}:\theta(X)=X\}$.

\begin{lemma}$K$ is a 6-dimensional connected compact Lie group with 2-dimensional maximal torus.\end{lemma}
\begin{proof}Connectedness of $K$ is just a consequence of the fact that $\Zhe_2$ is connected, since Lie groups with finitely many connected components must deformation retract onto their maximal compact subgroups.

All other information in the lemma is essentially visible directly from the root diagram for $\Zhe_2$, utilizing that $\mathfrak{g}$ is split real with maximally noncompact $\theta$-stable Cartan subalgebra $\mathfrak{a}=\langle\begin{smallbmatrix}1 & 0 \\ 0 & 0\end{smallbmatrix},\begin{smallbmatrix}0 & 0 \\ 0 & 1\end{smallbmatrix}\rangle<\mathfrak{g}_0$. Indeed, since $\mathfrak{g}$ is split real, all of the (restricted) root spaces are 1-dimensional and the Cartan subalgebra $\mathfrak{a}$ contains no compact subalgebra, so we get a basis for $\mathfrak{k}$ indexed by the (nonzero) root pairs $\pm\nu$: for each positive root $\nu$, we may pick a nonzero $\eta_\nu\in\mathfrak{g}_\nu$, which gives us $\theta(\eta_\nu)\in\mathfrak{g}_{-\nu}$ and $\eta_\nu+\theta(\eta_\nu)\in\mathfrak{k}$. The dimension of $\mathfrak{k}$ is, therefore, just the number of such root pairs---or equivalently, the number of positive roots---which is $6$. To see that $K$ has maximal tori of dimension $2$, note that each short root pair $\pm\nu_\text{short}$ has a unique pair of long roots $\pm\nu_\text{long}$ perpendicular to it. From the root diagram, we can see (in Figure \ref{perproots}) that $a\nu_\text{short}+b\nu_\text{long}$ is never another root for $a,b\in\mathbb{Z}$ unless either $a$ or $b$ is $0$, so in particular,
\[\left[\eta_{\nu_\text{short}}\!\!+\theta(\eta_{\nu_\text{short}}),\eta_{\nu_\text{long}}\!\!+\theta(\eta_{\nu_\text{long}})\right]=0.\]
Consequently, we get a 2-dimensional torus in $K$ generated by these two elements. Since the complexification $\mathbb{C}\otimes\mathfrak{g}$ is the complex exceptional simple Lie algebra of (complex) rank 2, we know that $\mathfrak{g}$ cannot contain a toral subalgebra of dimension greater than 2, so such a 2-dimensional torus must be maximal.\mbox{\qedhere}\end{proof}

\begin{figure}
\centering\includegraphics[width=0.7\textwidth]{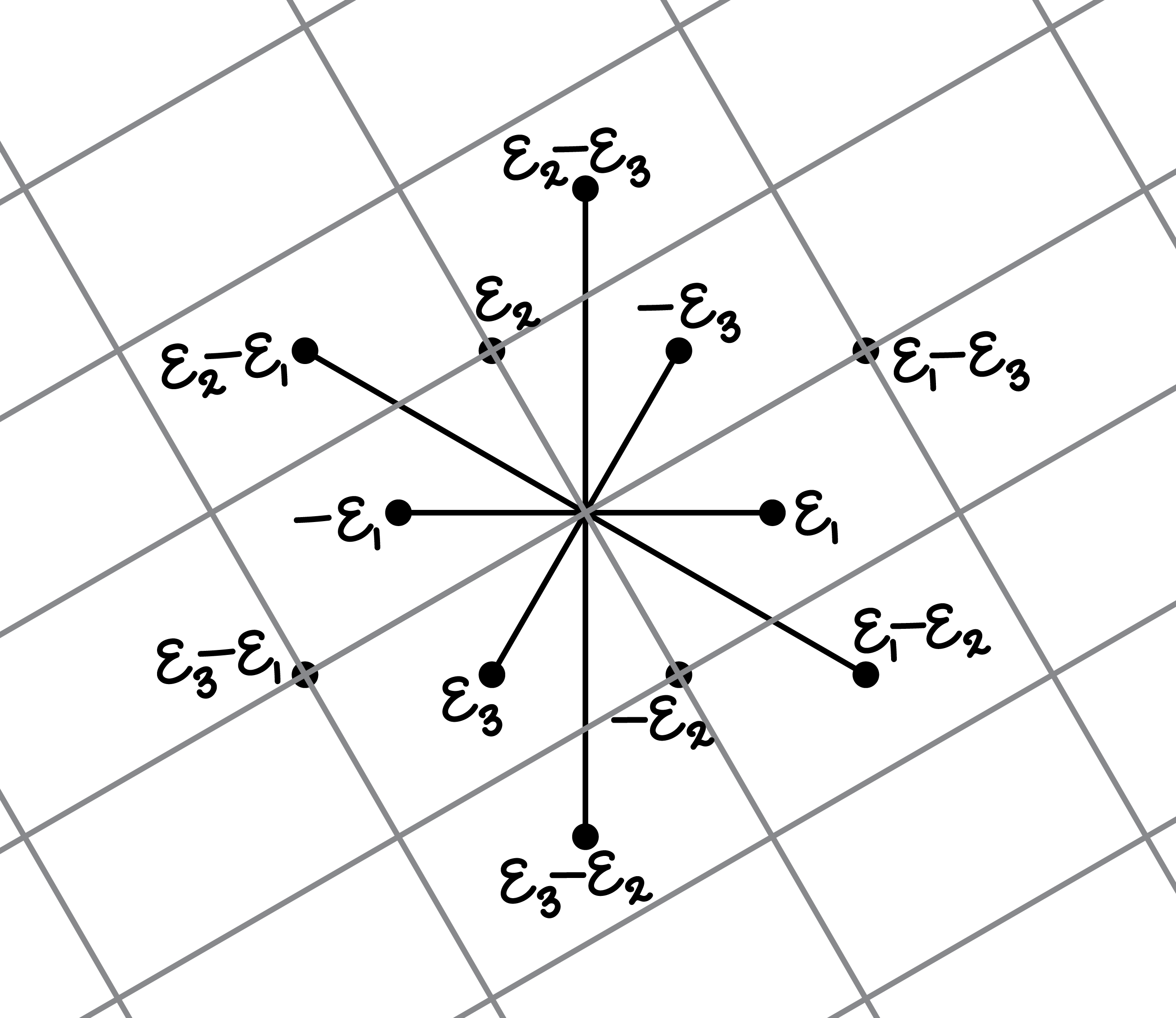}
\caption{A picture showing that the integer lattice generated by $\pm\varepsilon_2$ and $\pm(\varepsilon_1-\varepsilon_3)$ only intersects the set of (nonzero) roots for $\mathfrak{g}$ at $\pm\varepsilon_2$ and $\pm(\varepsilon_1-\varepsilon_3)$}
\label{perproots}
\end{figure}

Up to isomorphism of Lie groups, $\Spin(3)=\widetilde{\SOrth(3)}\simeq\SU(2)\simeq\Sp(1)$ is the only simply connected compact simple Lie group of dimension less than or equal to 6. Therefore, by the classification of connected compact Lie groups, an immediate consequence of the above lemma is the following.

\begin{corollary}\label{coverexists} There exists a covering homomorphism
\[\rho:\Spin(3)\times\Spin(3)\to K,\]
so that $\mathfrak{k}\approx\mathfrak{so}(3)\oplus\mathfrak{so}(3)$.\end{corollary}

This is where the geometry of rolling spheres from Definition \ref{rollingsph} begins to enter our picture, via the isomorphism $\mathfrak{k}\approx\mathfrak{so}(3)\oplus\mathfrak{so}(3)$. To help keep the reader in the right frame of mind, we would like to remark here how fundamentally \textit{non-miraculous} this fact is, and we hope that our chosen method of proof above reflects this. Being a connected Lie group, $\Zhe_2$ necessarily has a connected maximal compact subgroup, and since we are restricted to such low dimension and rank, we have very limited options for what this could be; indeed, once we know the specific dimension and rank, there is exactly one option up to finite coverings. In other words, the fact that rolling spheres are relevant here should really be seen as a consequence of low rank forcing certain options upon us, rather than there being something particularly special about rolling spheres themselves.

\subsection{Specifying the form of \texorpdfstring{$\rho:\mathrm{Spin}(3)\times\mathrm{Spin}(3)\to K$}{the homomorphism}}\label{rhoform}\hfill \\
From Corollary \ref{coverexists} above, we know that there must exist \textit{some} covering homomorphism $\rho:\Spin(3)\times\Spin(3)\to K$, which necessarily gives us a Lie algebra isomorphism $\rho_*:\mathfrak{so}(3)\oplus\mathfrak{so}(3)\to\mathfrak{k}$. In this subsection, we will find a convenient and geometrically meaningful form for $\rho$.

Let us start by considering the element $\rotoiso\in\mathfrak{k}\cap\mathfrak{g}_0\approx\mathfrak{o}(2)$.

\begin{lemma}\label{setrotoiso} Up to conjugation, we may choose $\rho$ such that
\[\rho_*^{-1}\left(\rotoiso\right)=\left(\begin{smallbmatrix}0 & 0 & 0 \\ 0 & 0 & -1 \\ 0 & 1 & 0\end{smallbmatrix},\begin{smallbmatrix}0 & 0 & 0 \\ 0 & 0 & -1 \\ 0 & 1 & 0\end{smallbmatrix}\right).\]\end{lemma}
\begin{proof}Via $\Spin(3)$, every element of $\mathfrak{so}(3)$ is conjugate to one of the form
\[a\begin{smallbmatrix}0 & 0 & 0 \\ 0 & 0 & -1 \\ 0 & 1 & 0\end{smallbmatrix}\]
for some $a\geq 0$, so every element of $\mathfrak{so}(3)\oplus\mathfrak{so}(3)$ is conjugate to one of the form
\[\left(a\begin{smallbmatrix}0 & 0 & 0 \\ 0 & 0 & -1 \\ 0 & 1 & 0\end{smallbmatrix},b\begin{smallbmatrix}0 & 0 & 0 \\ 0 & 0 & -1 \\ 0 & 1 & 0\end{smallbmatrix}\right)\]
for some $a,b\geq 0$. In particular, we may choose for $\rho_*^{-1}(\rotoiso)$ to take this form. But the adjoint action of $\rotoiso$ preserves the 2-dimensional subspace $\{v+\theta(v)\in\mathfrak{k}:v\in\mathfrak{g}_{-1}\}$, acting as $\rotoiso\in\mathfrak{gl}_2\mathbb{R}$ does on $\mathbb{R}^2\approx\mathfrak{g}_{-1}$, so that
\[\ad_{\rotoiso}^2|_{\{v+\theta(v)\in\mathfrak{k}:v\in\mathfrak{g}_{-1}\}}=-\mathrm{id}_{\{v+\theta(v)\in\mathfrak{k}:v\in\mathfrak{g}_{-1}\}}.\]
Writing $\rho_*^{-1}(\rotoiso)$ in the form above, we have that
\[\textstyle\ad_{\rho_*^{-1}\left(\rotoiso\right)}^2\!\left(\begin{smallbmatrix}0 & -x & -y \\ x & 0 & -z \\ y & z & 0\end{smallbmatrix}\!,\begin{smallbmatrix}0 & -p & -q \\ p & 0 & -r \\ q & r & 0\end{smallbmatrix}\right)\!=\!\left(-a^2\!\begin{smallbmatrix}0 & -x & -y \\ x & 0 & 0 \\ y & 0 & 0\end{smallbmatrix}\!,-b^2\!\begin{smallbmatrix}0 & -p & -q \\ p & 0 & 0 \\ q & 0 & 0\end{smallbmatrix}\right)\!,\]
so since the subspace $\rho_*^{-1}(\{v+\theta(v):v\in\mathfrak{g}_{-1}\})$ cannot be contained in just one of the individual $\mathfrak{so}(3)$ summands of $\mathfrak{so}(3)\oplus\mathfrak{so}(3)$, we must have $a^2=b^2=1$, hence $a=b=1$ since $a,b\geq 0$.\mbox{\qedhere}\end{proof}

We will address the specifics of how to interpret this in Section \ref{comparison}, but for now, notice that our choice of $\rho_*^{-1}(\rotoiso)$ happens to generate the Lie algebra of the isotropy for the model geometry of rolling spheres from Definition \ref{rollingsph}.

As an aside, note that knowing $\rho_*^{-1}(\rotoiso)$ lets us determine that $\rho$ must be a double cover.

\begin{corollary}The homomorphism $\rho$ is a double cover onto $K$.\end{corollary}
\begin{proof}Since $\exp(2\pi\rotoiso)=e$ in $G_0<\Zhe_2$, it follows that
\[\exp(2\pi\rho_*^{-1}(\rotoiso))\in\ker(\rho)\leq\mathrm{Z}(\Spin(3)\times\Spin(3)).\]
However, $\hat{\theta}\in\mathrm{Z}(K)$ is a nontrivial element of the center of $K$, so
\[\ker(\rho)\neq\mathrm{Z}(\Spin(3)\times\Spin(3))\simeq\mathbb{Z}/2\mathbb{Z}\times\mathbb{Z}/2\mathbb{Z}.\]
Thus, $\ker(\rho)=\langle\exp(2\pi\rho_*^{-1}(\rotoiso))\rangle$, so $\rho$ is a double cover.\mbox{\qedhere}\end{proof}

\newpage
The subspace $\{v+\theta(v):v\in\mathfrak{g}_{-1}\approx\mathbb{R}^2\}$ appearing in the proof of Lemma \ref{setrotoiso} will play an important role in what follows. Because we have $\theta(\mathfrak{g}_i)=\mathfrak{g}_{-i}$ for each $i$, the intersection of $\mathfrak{k}$ with the filtration component $\mathfrak{g}^{-1}=\mathfrak{g}_{-1}+\mathfrak{g}_0+\mathfrak{g}_1+\mathfrak{g}_2+\mathfrak{g}_3$ of $\mathfrak{g}$ is given by
\[\mathfrak{k}\cap\mathfrak{g}^{-1}=\{X\in\mathfrak{g}^{-1}:X=\theta(X)\}=\{v+\theta(v):v\in\mathfrak{g}_{-1}\}+\mathfrak{k}\cap\mathfrak{g}_0.\]
Intuitively, then, we might think of $\{v+\theta(v):v\in\mathfrak{g}_{-1}\}\approx\mathfrak{g}^{-1}/\mathfrak{p}$ as the way that $K$ describes the canonical $(2,3,5)$-distribution of $(\Zhe_2,P)$, ignoring the isotropy part $\mathfrak{k}\cap\mathfrak{g}_0=\langle\rotoiso\rangle$ of $\mathfrak{k}\cap\mathfrak{g}^{-1}$.

\begin{lemma}\label{setrolleye} Up to further conjugation by an element of
\[\exp\left(\mathbb{R}\begin{smallbmatrix}0 & 0 & 0 \\ 0 & 0 & -1 \\ 0 & 1 & 0\end{smallbmatrix}\right)\times\exp\left(\mathbb{R}\begin{smallbmatrix}0 & 0 & 0 \\ 0 & 0 & -1 \\ 0 & 1 & 0\end{smallbmatrix}\right)\leq\mathrm{Z}_{\Spin(3)\times\Spin(3)}(\rho_*^{-1}(\rotoiso)),\]
which preserves our choice of $\rho_*^{-1}(\rotoiso)$, we may choose $\rho$ such that, for \textbf{some} $r_0,r_1>0$,
\[\rho_*^{-1}(v+\theta(v))=\left(\tfrac{1}{r_0}\begin{smallbmatrix}0 & -v^\top \\ v & 0\end{smallbmatrix},\tfrac{1}{r_1}\begin{smallbmatrix}0 & -v^\top \\ v & 0\end{smallbmatrix}\right)\]
for all $v\in\mathfrak{g}_{-1}$.\end{lemma}
\begin{proof}As we saw in the proof of Lemma \ref{setrotoiso}, the adjoint action of $\rotoiso$ squares to minus the identity on $\{v+\theta(v):v\in\mathfrak{g}_{-1}\}$, so our chosen form of $\rho_*^{-1}(\rotoiso)$ forces $\rho_*^{-1}(\{v+\theta(v):v\in\mathfrak{g}_{-1}\})$ to be contained in the subspace
\[\left\{\left(\begin{smallbmatrix}0 & -w^\top \\ w & 0\end{smallbmatrix},\begin{smallbmatrix}0 & -w'^\top \\ w & 0\end{smallbmatrix}\right)\in\mathfrak{so}(3)\oplus\mathfrak{so}(3):w,w'\in\mathbb{R}^2\right\}.\]
Because $\rho_*^{-1}(\{v+\theta(v):v\in\mathfrak{g}_{-1}\})$ cannot be contained in just one of the $\mathfrak{so}(3)$ summands and
\[\Ad_{\exp\left(t\begin{smallbmatrix}0 & 0 & 0 \\ 0 & 0 & -1 \\ 0 & 1 & 0\end{smallbmatrix}\right)}\!\begin{smallbmatrix}0 & -x & -y \\ x & 0 & 0 \\ y & 0 & 0\end{smallbmatrix}\!=\!\begin{smallbmatrix}0 & -\cos(t)x+\sin(t)y & -\sin(t)x-\cos(t)y \\ \cos(t)x-\sin(t)y & 0 & 0 \\ \sin(t)x+\cos(t)y & 0 & 0\end{smallbmatrix}\!,\]
we may find an element
\[\ell\in\exp\left(\mathbb{R}\begin{smallbmatrix}0 & 0 & 0 \\ 0 & 0 & -1 \\ 0 & 1 & 0\end{smallbmatrix}\right)\times\exp\left(\mathbb{R}\begin{smallbmatrix}0 & 0 & 0 \\ 0 & 0 & -1 \\ 0 & 1 & 0\end{smallbmatrix}\right)\]
such that each of the $\mathfrak{so}(3)$ summand components for $\Ad_\ell\rho_*^{-1}(e_1+\theta(e_1))$ is a positive multiple of $\big[\begin{smallmatrix}0 & -e_1^\top\\ e_1 & 0\end{smallmatrix}\big]$. Since $\Ad_\ell\rho_*^{-1}(\rotoiso)=\rho_*^{-1}(\rotoiso)$, we may therefore replace $\rho$ by $\rho\circ\mathrm{conj}_{\ell^{-1}}\!:g\mapsto\rho(\ell^{-1} g\ell)$ to get a covering homomorphism $\rho$ with the same value of $\rho_*^{-1}(\rotoiso)$ and such that
\[\rho_*^{-1}(e_1+\theta(e_1))=\left(\tfrac{1}{r_0}\begin{smallbmatrix}0 & -e_1^\top \\ e_1 & 0\end{smallbmatrix},\tfrac{1}{r_1}\begin{smallbmatrix}0 & -e_1^\top \\ e_1 & 0\end{smallbmatrix}\right)\]
for some $r_0,r_1>0$, which implies the desired condition by the action of $\rotoiso$ on $\{v+\theta(v):v\in\mathfrak{g}_{-1}\}$.\mbox{\qedhere}\end{proof}

Our peculiar use of $r_0$ and $r_1$ in the lemma above should, we hope, be redolent of the rolling distribution $D_{r_0/r_1}$ from Definition \ref{rollingdist}. Indeed, combining Lemmas \ref{setrotoiso} and \ref{setrolleye} tells us that we may choose $\rho$ in such a way that, for some $r_0,r_1>0$, $\rho_*$ restricts to an isomorphism from
\[\left\{\left(\tfrac{1}{r_0}\begin{smallbmatrix}0 & -v^\top \\ v & 0\end{smallbmatrix}\!,\tfrac{1}{r_1}\begin{smallbmatrix}0 & -v^\top \\ v & 0\end{smallbmatrix}\right):v\in\mathbb{R}^2\right\}+\mathfrak{o}(2)\subset\mathfrak{so}(3)\oplus\mathfrak{so}(3),\]
the subspace of $\mathfrak{so}(3)\oplus\mathfrak{so}(3)$ used to define the rolling distribution for the radii $r_0$ and $r_1$, to
\[\mathfrak{k}\cap\mathfrak{g}^{-1}=\{v+\theta(v):v\in\mathfrak{g}_{-1}\}+\mathfrak{k}\cap\mathfrak{g}_0,\]
which describes the canonical $(2,3,5)$-distribution for $(\Zhe_2,P)$ as it is seen by $K$. Again, we will explore the geometric meaning of this in more detail in Section \ref{comparison}.

For each $v\in\mathfrak{g}_{-1}\approx\mathbb{R}^2$, let
\[\bar{v}:=e_1\wedge e_2\otimes\rotoiso v\in\mathfrak{g}_{-3}\approx(\Lambda^2\mathbb{R}^2)\otimes\mathbb{R}^2.\]
Then,
\begin{align*}[v+\theta(v),\bar{v}+\theta(\bar{v})] & =[v,\bar{v}]+([\theta(v),\bar{v}]+[v,\theta(\bar{v})])+[\theta(v),\theta(\bar{v})] \\ & =0+(v^\top\rotoiso v)(e_1\wedge e_2-e_1^\vee\wedge e_2^\vee)+0=0,\end{align*}
so $v+\theta(v)$ commutes with $\bar{v}+\theta(\bar{v})$, and for $A\in\Orth(2)<\GLin_2\mathbb{R}\simeq G_0$,
\[\Ad_A(\bar{v}+\theta(\bar{v}))=\overline{Av}+\theta\!\left(\overline{Av}\right)\!,\]
so the map $v+\theta(v)\mapsto\bar{v}+\theta(\bar{v})$ is $\Orth(2)$-equivariant. Moreover, under the isomorphism between $\mathfrak{g}_{-3}+\mathfrak{g}_0+\mathfrak{g}_3$ and $\mathfrak{sl}_3\mathbb{R}$ from Section \ref{structurestuff}, the subalgebra
\[\{\bar{v}+\theta(\bar{v}):v\in\mathfrak{g}_{-1}\approx\mathbb{R}^2\}+\langle\rotoiso\rangle=\mathfrak{k}\cap(\mathfrak{g}_{-3}+\mathfrak{g}_0+\mathfrak{g}_3)\]
gets identified with $\mathfrak{so}(3)<\mathfrak{sl}_3\mathbb{R}$.

\begin{lemma}\label{setslideeye} Choosing $\rho$ to be as in Lemma \ref{setrolleye}, we must have
\[\rho_*^{-1}(\bar{v}+\theta(\bar{v}))=\pm\left(\begin{smallbmatrix}0 & -v^\top \\ v & 0\end{smallbmatrix}\!,\,-\!\begin{smallbmatrix}0 & -v^\top \\ v & 0\end{smallbmatrix}\right).\]
\end{lemma}
\begin{proof}Since $\rho_*^{-1}(\bar{v}+\theta(\bar{v}))\in\mathrm{Z}_{\mathfrak{so}(3)\oplus\mathfrak{so}(3)}(\rho_*^{-1}(v+\theta(v)))$, we know that there must exist $x,y\in\mathbb{R}$ such that
\[\rho_*^{-1}(\bar{v}+\theta(\bar{v}))=\left(x\begin{smallbmatrix}0 & -v^\top \\ v & 0\end{smallbmatrix}\!,\,y\begin{smallbmatrix}0 & -v^\top \\ v & 0\end{smallbmatrix}\right).\]
On top of this, we also have that $[\bar{e_1}+\theta(\bar{e_1}),\bar{e_2}+\theta(\bar{e_2})]=\rotoiso$,
so
\begin{align*}\rho_*^{-1}(\rotoiso) & =[\rho_*^{-1}(\bar{e_1}+\theta(\bar{e_1})),\rho_*^{-1}(\bar{e_2}+\theta(\bar{e_2}))] \\ & =\left(x^2\begin{smallbmatrix}0 & 0 & 0 \\ 0 & 0 & -1 \\ 0 & 1 & 0\end{smallbmatrix}\!,\, y^2\begin{smallbmatrix}0 & 0 & 0 \\ 0 & 0 & -1 \\ 0 & 1 & 0\end{smallbmatrix}\right)\!,\end{align*}
hence $x^2=y^2=1$.

To find the signs of $x$ and $y$, note that the Killing form $\kgf$ of $\mathfrak{g}$ restricts to a negative-definite inner product on $\mathfrak{k}$. Because $\kgf$ is $\Ad_K$-invariant, the restriction of $\kgf$ to $\mathfrak{k}\approx\mathfrak{so}(3)\oplus\mathfrak{so}(3)$ must be a linear combination of the Killing forms on the $\mathfrak{so}(3)$ summands, and in order for this to be negative-definite, it must be a \textit{positive} linear combination of those (negative-definite) Killing forms. In particular, since $r_0,r_1>0$ and
\begin{align*}0 & =\kgf(e_1+\theta(e_1),\bar{e_1}+\theta(\bar{e_1})) \\ & =\kgf\!\left(\rho_*\!\left(\tfrac{1}{r_0}\!\begin{smallbmatrix}0 & -e_1^\top \\ e_1 & 0\end{smallbmatrix}\!,\,\tfrac{1}{r_1}\!\begin{smallbmatrix}0 & -e_1^\top \\ e_1 & 0\end{smallbmatrix}\right)\!,\,\rho_*\!\left(x\!\begin{smallbmatrix}0 & -e_1^\top \\ e_1 & 0\end{smallbmatrix}\!,\,y\!\begin{smallbmatrix}0 & -e_1^\top \\ e_1 & 0\end{smallbmatrix}\right)\right) \\ & =\tfrac{x}{r_0}\kgf\!\left(\rho_*\!\left(\begin{smallbmatrix}0 & -e_1^\top \\ e_1 & 0\end{smallbmatrix}\!,0\right)\!,\,\rho_*\!\left(\begin{smallbmatrix}0 & -e_1^\top \\ e_1 & 0\end{smallbmatrix}\!,0\right)\right) \\ & \quad +\tfrac{y}{r_1}\kgf\!\left(\rho_*\!\left(0,\begin{smallbmatrix}0 & -e_1^\top \\ e_1 & 0\end{smallbmatrix}\right)\!,\,\rho_*\!\left(0,\begin{smallbmatrix}0 & -e_1^\top \\ e_1 & 0\end{smallbmatrix}\right)\right)\!,\end{align*}
$x$ and $y$ must have opposite signs. Up to swapping the $\mathfrak{so}(3)$ summands, $\rho_*^{-1}(\bar{v}+\theta(\bar{v}))$ must therefore take the form $([\begin{smallmatrix}0 & -v^\top \\ v & 0\end{smallmatrix}],-[\begin{smallmatrix}0 & -v^\top \\ v & 0\end{smallmatrix}])$.\mbox{\qedhere}\end{proof}

\subsection{The two intersections}\label{intersecttwice}\hfill \\
Consider the one-parameter subgroups $\exp(\mathbb{R}(e_1+\theta(e_1)))$ and \[\exp(\mathbb{R}(\bar{e_1}+\theta(\bar{e_1})))=\exp(\mathbb{R}(e_1\wedge e_2\otimes e_2+\theta(e_1\wedge e_2\otimes e_2)))\]
in $K<\Zhe_2$. These are the one-parameter subgroups corresponding to the root pairs $\pm\varepsilon_2$ and $\pm(\varepsilon_1-\varepsilon_3)$, respectively, and together they generate a maximal torus in $K$. Using the root diagram, we can see exactly how many times these one-parameter subgroups intersect.

\begin{lemma}\label{2intersections} The images of the one-parameter subgroups generated by $e_1+\theta(e_1)$ and $e_1\wedge e_2\otimes e_2+\theta(e_1\wedge e_2\otimes e_2)$ in $K$ intersect exactly twice.\end{lemma}
\begin{proof}From the root diagram on the left in Figure \ref{short&longirreps}, we can see that the copy of $\mathfrak{sl}_2\mathbb{R}$ generated by $e_1\in\mathfrak{g}_{\varepsilon_2}$ and $e_1^\vee=-\theta(e_1)\in\mathfrak{g}_{-\varepsilon_2}$, which we will call $\mathfrak{sl}_2\mathbb{R}_\text{short}$, breaks $\mathfrak{g}$ into irreducible subrepresentations as
\begin{gather*}\mathfrak{g}_{\varepsilon_3-\varepsilon_1}=\langle e_1\wedge e_2\otimes e_2\rangle, \\
\mathfrak{g}_{\varepsilon_2-\varepsilon_1}+\mathfrak{g}_{-\varepsilon_1}+\mathfrak{g}_{\varepsilon_3}+\mathfrak{g}_{\varepsilon_3-\varepsilon_2}, \\
\mathfrak{sl}_2\mathbb{R}_\text{short},\, \ker(\varepsilon_2)=\{Z\in\mathfrak{a}:\varepsilon_2(Z)=0\}, \\
\mathfrak{g}_{\varepsilon_2-\varepsilon_3}+\mathfrak{g}_{-\varepsilon_3}+\mathfrak{g}_{\varepsilon_1}+\mathfrak{g}_{\varepsilon_1-\varepsilon_2}, \\
\text{and }\mathfrak{g}_{\varepsilon_1-\varepsilon_3}=\langle e_1^\vee\wedge e_2^\vee\otimes e_2^\vee\rangle.\end{gather*}
Similarly, from the root diagram on the right in Figure \ref{short&longirreps}, we can see that the copy of $\mathfrak{sl}_2\mathbb{R}$ generated by $e_1\wedge e_2\otimes e_2=\bar{e_1}\in\mathfrak{g}_{\varepsilon_3-\varepsilon_1}$ and $e_1^\vee\wedge e_2^\vee\otimes e_2^\vee=-\theta(\bar{e_1})\in\mathfrak{g}_{\varepsilon_1-\varepsilon_3}$, which we will call $\mathfrak{sl}_2\mathbb{R}_\text{long}$, breaks $\mathfrak{g}$ into irreducible subrepresentations as
\begin{gather*}\mathfrak{g}_{\varepsilon_2-\varepsilon_1}+\mathfrak{g}_{\varepsilon_2-\varepsilon_3}, \\
\mathfrak{g}_{\varepsilon_2}=\langle e_1\rangle, \\
\mathfrak{g}_{-\varepsilon_1}+\mathfrak{g}_{\varepsilon_3}, \\
\mathfrak{sl}_2\mathbb{R}_\text{long},\, \ker(\varepsilon_1-\varepsilon_3)=\{Z\in\mathfrak{a}:(\varepsilon_1-\varepsilon_3)(Z)=0\}, \\
\mathfrak{g}_{-\varepsilon_3}+\mathfrak{g}_{\varepsilon_1}, \\
\mathfrak{g}_{-\varepsilon_2}=\langle e_1^\vee\rangle, \\
\text{and }\mathfrak{g}_{\varepsilon_3-\varepsilon_2}+\mathfrak{g}_{\varepsilon_1-\varepsilon_2}.\end{gather*}

\begin{figure}
\centering\includegraphics[width=0.44\textwidth]{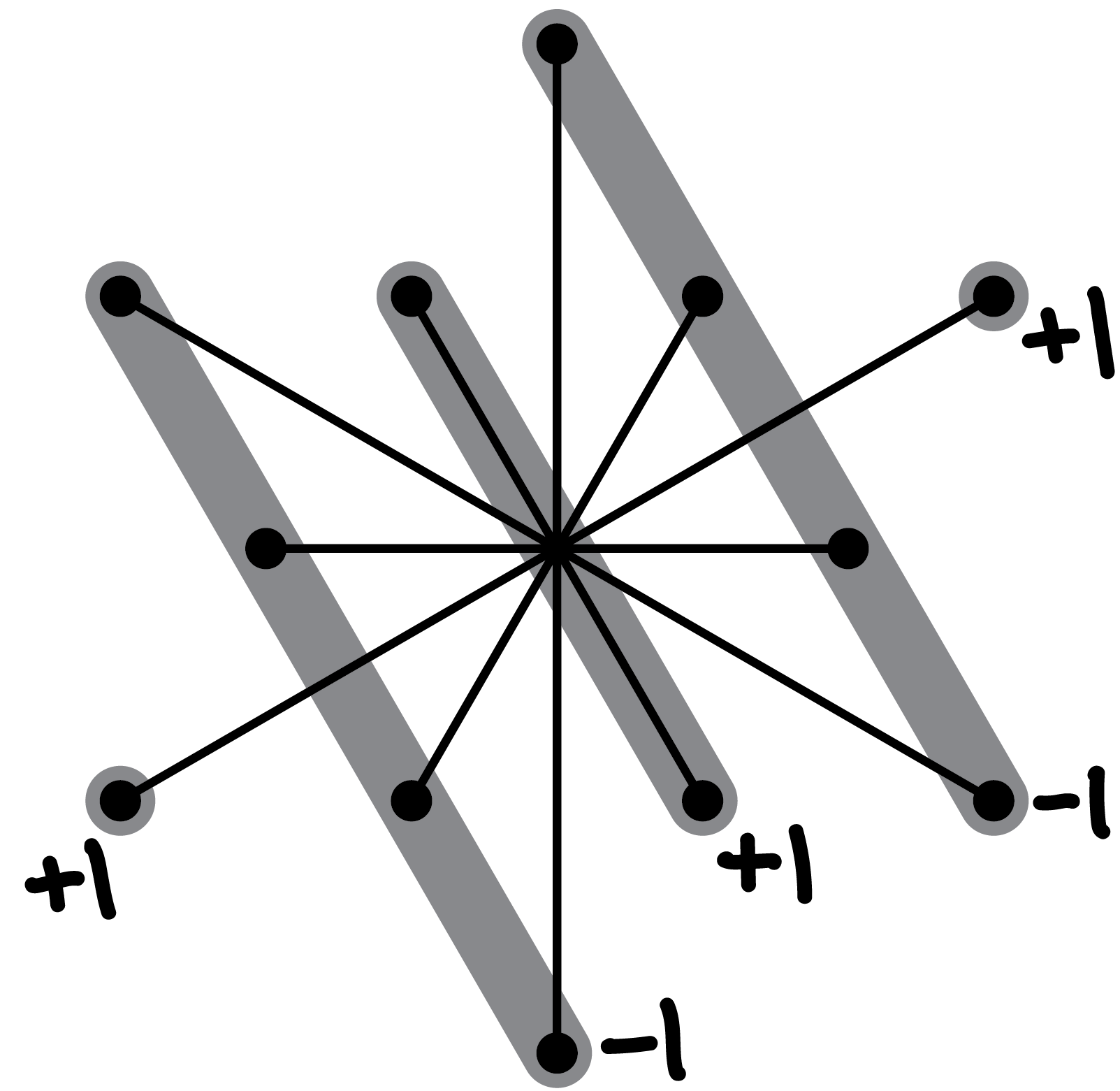}\hfill\includegraphics[width=0.44\textwidth]{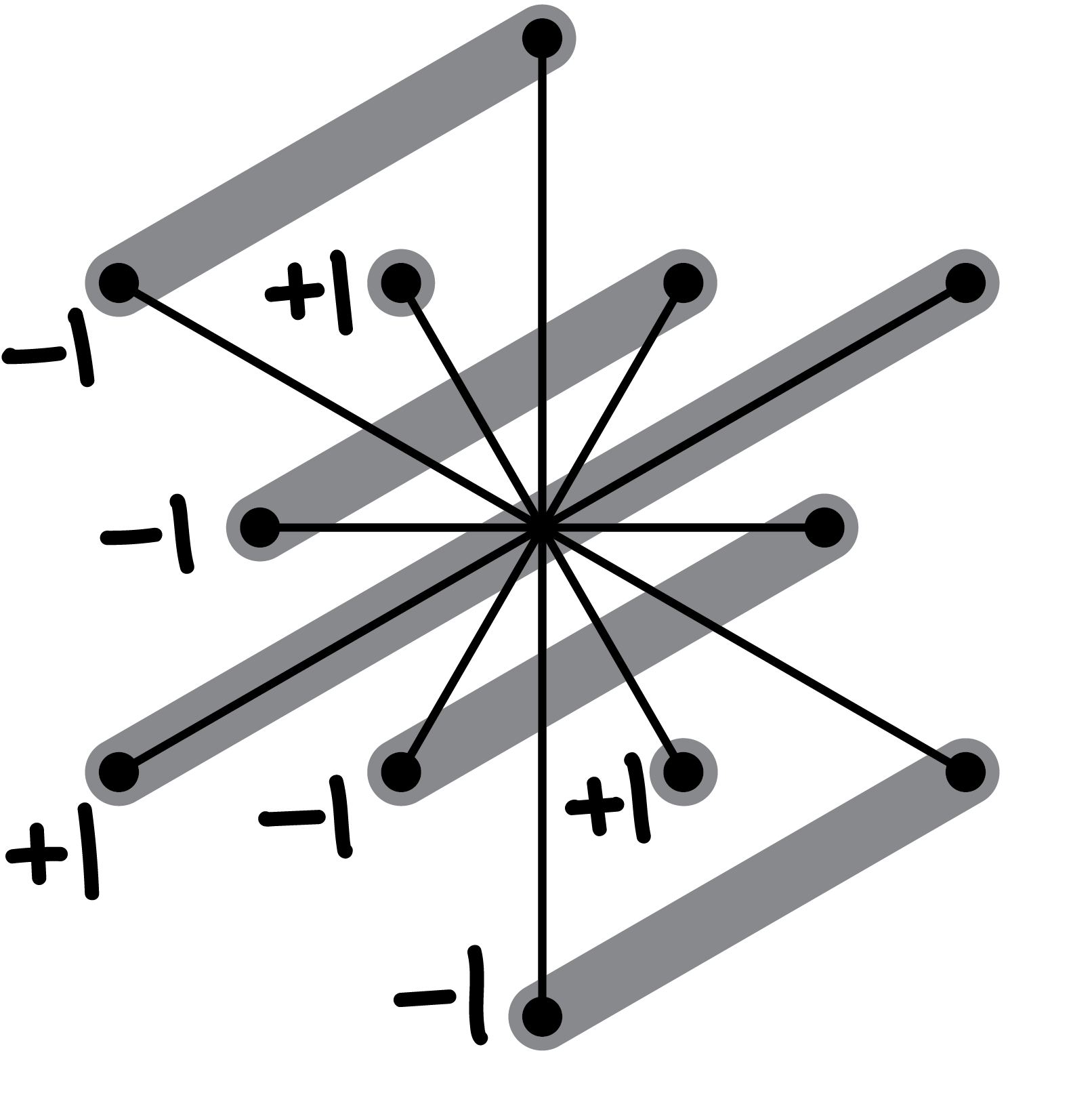}
\caption{A pair of root diagrams for $\mathfrak{g}$, with the roots differing by integer multiples of $\varepsilon_2$ (left) or $\varepsilon_3-\varepsilon_1$ (right) grouped together to indicate subrepresentations of $\mathfrak{g}$ with respect to $\mathfrak{sl}_2\mathbb{R}_\text{short}$ (left) and $\mathfrak{sl}_2\mathbb{R}_\text{long}$ (right), where each subrepresentation is labeled by either $+1$ or $-1$ according to whether it has odd or even dimension, respectively}
\label{short&longirreps}
\end{figure}

Denote by $\SLin_2\mathbb{R}_\text{short}$ and $\SLin_2\mathbb{R}_\text{long}$ the connected subgroups of $\Zhe_2$ generated by $\mathfrak{sl}_2\mathbb{R}_\text{short}$ and $\mathfrak{sl}_2\mathbb{R}_\text{long}$, respectively. Since $\Zhe_2$ is centerless by definition, so that $\Ad:\Zhe_2\to\Aut(\mathfrak{g})$ is a Lie group isomorphism, two elements of $\Zhe_2$ are the same if and only if their adjoint actions on $\mathfrak{g}$ are the same. Therefore, since $\SLin_2\mathbb{R}$ acts faithfully on each of its (nontrivial) even-dimensional irreducible representations and acts with kernel exactly $\langle-\mathds{1}\rangle<\SLin_2\mathbb{R}$ on each of its (nontrivial) odd-dimensional irreducible representations, both the subgroups $\SLin_2\mathbb{R}_\text{short}$ and $\SLin_2\mathbb{R}_\text{long}$ must be isomorphic to $\SLin_2\mathbb{R}$. Within these subgroups, we may identify
\[\SOrth(2)_\text{short}:=\exp(\mathbb{R}(e_1+\theta(e_1)))<\SLin_2\mathbb{R}_\text{short}\]
and
\[\SOrth(2)_\text{long}:=\exp(\mathbb{R}(\bar{e_1}+\theta(\bar{e_1})))<\SLin_2\mathbb{R}_\text{long}.\]

To see where these two copies of $\SOrth(2)$ intersect, note that each of the irreducible subrepresentations of $\mathfrak{g}$ with respect to $\SLin_2\mathbb{R}_\text{short}$ intersects the irreducible subrepresentations with respect to $\SLin_2\mathbb{R}_\text{long}$ in either a single root space or a subspace of $\mathfrak{a}$. Elements of $\SOrth(2)_\text{short}\cap\SOrth(2)_\text{long}$ must preserve each of these intersections, so all such elements must act by scalar multiplication on the root spaces and $\mathfrak{a}$. Given how $\SOrth(2)$ acts on these representations, the only nontrivial element of $\SOrth(2)_\text{short}$ that could possibly be contained in $\SOrth(2)_\text{long}$ is the element \[-\mathds{1}_\text{short}:=\exp(\pi(e_1+\theta(e_1)))\] corresponding to $-\mathds{1}=\exp(\pi\rotoiso)\in\SOrth(2)$, and likewise, the only nontrivial element of $\SOrth(2)_\text{long}$ that could be in the intersection is
\[-\mathds{1}_\text{long}:=\exp(\pi(\bar{e_1}+\theta(\bar{e_1}))).\]
Thus, to prove that we have exactly two intersections, it just remains to show that $-\mathds{1}_\text{short}=-\mathds{1}_\text{long}$.

Conveniently, this too is visible directly from the root diagram of $\mathfrak{g}$: because the element $-\mathds{1}\in\SOrth(2)$ acts by $-1$ on the even-dimensional irreducible representations and by $1$ on the odd-dimensional irreducible representations, we may label each root of the left root diagram in Figure \ref{short&longirreps} by $\pm1$ according to whether its root space is contained in an irreducible representation of odd or even dimension for $\SLin_2\mathbb{R}_\text{short}$. This tells us how $\Ad_{-\mathds{1}_\text{short}}$ acts on each root space, and doing the same thing for $-\mathds{1}_\text{long}$ in the right root diagram of Figure \ref{short&longirreps} gives the same result for each root,
so $\Ad_{-\mathds{1}_\text{short}}\!=\Ad_{-\mathds{1}_\text{long}}$ and hence $-\mathds{1}_\text{short}=-\mathds{1}_\text{long}$.\mbox{\qedhere}

\end{proof}

In Lemma \ref{setrolleye} above, we proved that $\rho_*^{-1}(e_1+\theta(e_1))$ can be chosen to be of the form $\big(\tfrac{1}{r_0}\big[\begin{smallmatrix}0 & -e_1^\top \\ e_1 & 0\end{smallmatrix}\big],\tfrac{1}{r_1}\big[\begin{smallmatrix}0 & -e_1^\top \\ e_1 & 0\end{smallmatrix}\big]\big)$ for \textit{some} positive real numbers $r_0$ and $r_1$, but we did not specify $r_0$ and $r_1$ further than this. Now, using Lemma \ref{2intersections}, we can see what they must be.

\begin{corollary}\label{ratioappears} In the setting of Lemma \ref{setrolleye}, either $r_0=1$ and $r_1=\tfrac{1}{3}$ or $r_0=\tfrac{1}{3}$ and $r_1=1$.\end{corollary}
\begin{proof}Consider the maximal torus
\[T=\exp\left(\mathbb{R}\begin{smallbmatrix}0 & -1 & 0 \\ 1 & 0 & 0 \\ 0 & 0 & 0\end{smallbmatrix}\right)\times\exp\left(\mathbb{R}\begin{smallbmatrix}0 & -1 & 0 \\ 1 & 0 & 0 \\ 0 & 0 & 0\end{smallbmatrix}\right)<\Spin(3)\times\Spin(3),\]
which is isomorphic to $\mathbb{R}^2/4\pi\mathbb{Z}^2$ by
\[\exp\left(x\!\begin{smallbmatrix}0 & -1 & 0 \\ 1 & 0 & 0 \\ 0 & 0 & 0\end{smallbmatrix}\!,\, y\!\begin{smallbmatrix}0 & -1 & 0 \\ 1 & 0 & 0 \\ 0 & 0 & 0\end{smallbmatrix}\right)\mapsto\begin{smallbmatrix}x \\ y\end{smallbmatrix}+4\pi\mathbb{Z}^2.\]
This maximal torus $T$ is the same one generated by $\rho_*^{-1}(e_1+\theta(e_1))$ and $\rho_*^{-1}(\bar{e_1}+\theta(\bar{e_1}))$. Under the isomorphism $T\simeq\mathbb{R}^2/4\pi\mathbb{Z}^2$, Lemma \ref{setslideeye} tells us that $\rho_*^{-1}(\bar{e_1}+\theta(\bar{e_1}))\in\mathfrak{t}<\mathfrak{k}$ corresponds to the vector $\pm\begin{smallbmatrix}1 \\ -1\end{smallbmatrix}$, which generates a one-parameter subgroup of slope $\frac{-1}{1}=-1$ in $\mathbb{R}^2/4\pi\mathbb{Z}^2$. Similarly, by Lemma \ref{setrolleye}, there exist positive real numbers $r_0$ and $r_1$ such that the element $\rho_*^{-1}(e_1+\theta(e_1))$ corresponds to the vector $[\begin{smallmatrix}1/r_0 \\ 1/r_1\end{smallmatrix}]$, which generates a one-parameter subgroup of (positive) slope $\tfrac{1/r_1}{1/r_0}=\tfrac{r_0}{r_1}$.

Now, consider the torus $\rho(T)<K$, with corresponding isomorphism
\[\rho(T)\simeq\mathbb{R}^2/(4\pi\mathbb{Z}^2+2\pi\mathbb{Z}\begin{smallbmatrix}1 \\ 1\end{smallbmatrix})=\mathbb{R}^2/(2\pi\mathbb{Z}\begin{smallbmatrix}1 \\ -1\end{smallbmatrix}+2\pi\mathbb{Z}\begin{smallbmatrix}1 \\ 1\end{smallbmatrix}).\]
Because the subgroups $\SOrth(2)_\text{short}$ and $\SOrth(2)_\text{long}$ intersect exactly twice, the corresponding one-parameter subgroups in $\mathbb{R}^2/(4\pi\mathbb{Z}^2+2\pi\mathbb{Z}\begin{smallbmatrix}1 \\ 1\end{smallbmatrix})$ must intersect exactly twice as well: once at the identity and once at
\[\pi\begin{smallbmatrix}1 \\ -1\end{smallbmatrix}+(2\pi\mathbb{Z}\begin{smallbmatrix}1 \\ -1\end{smallbmatrix}+2\pi\mathbb{Z}\begin{smallbmatrix}1 \\ 1\end{smallbmatrix})=\pi[\begin{smallmatrix}1/r_0 \\ 1/r_1\end{smallmatrix}]+(2\pi\mathbb{Z}\begin{smallbmatrix}1 \\ -1\end{smallbmatrix}+2\pi\mathbb{Z}\begin{smallbmatrix}1 \\ 1\end{smallbmatrix}).\]
However, since the slope $\tfrac{r_0}{r_1}$ is positive and $\SOrth(2)_\text{short}$ cannot intersect $\SOrth(2)_\text{long}$ more than once, we can see in Figure \ref{slopecheckfig} that we must have either $\tfrac{r_0}{r_1}=\tfrac{1}{3}$, so that $\pi[\begin{smallmatrix}1/r_0 \\ 1/r_1\end{smallmatrix}]=\pi\begin{smallbmatrix}1 \\ -1\end{smallbmatrix}+2\pi\begin{smallbmatrix}1 \\ 1\end{smallbmatrix}=\pi\begin{smallbmatrix}3 \\ 1\end{smallbmatrix}$, or $\tfrac{r_0}{r_1}=3$, so that $\pi[\begin{smallmatrix}1/r_0 \\ 1/r_1\end{smallmatrix}]=\pi\begin{smallbmatrix}1 \\ -1\end{smallbmatrix}+2\pi\begin{smallbmatrix}1 \\ 1\end{smallbmatrix}-2\pi\begin{smallbmatrix}1 \\ -1\end{smallbmatrix}=\pi\begin{smallbmatrix}1 \\ 3\end{smallbmatrix}$.\mbox{\qedhere}

\begin{figure}
\centering\includegraphics[width=0.6\textwidth]{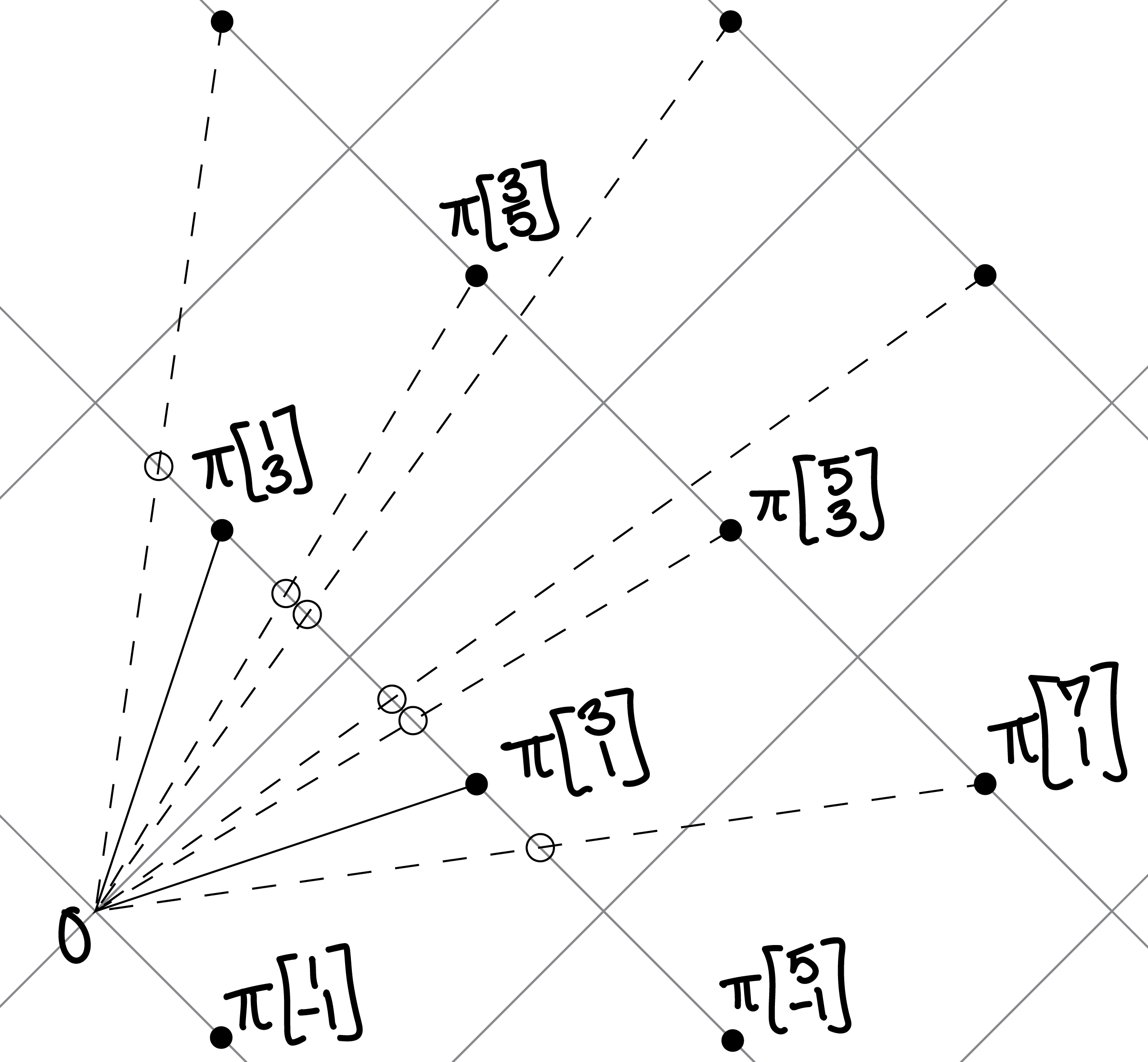}
\caption{A depiction of the universal cover of the torus $\mathbb{R}^2/(2\pi\mathbb{Z}\begin{smallbmatrix}1 \\ -1\end{smallbmatrix}+2\pi\mathbb{Z}\begin{smallbmatrix}1 \\ 1\end{smallbmatrix})$, with various line segments of positive slopes from the zero vector to the lifts of the point $\pi\begin{smallbmatrix}1 \\ -1\end{smallbmatrix}+(2\pi\mathbb{Z}\begin{smallbmatrix}1 \\ -1\end{smallbmatrix}+2\pi\mathbb{Z}\begin{smallbmatrix}1 \\ 1\end{smallbmatrix})$ indicated; dotted line segments are those that intersect $2\pi\begin{smallbmatrix}1 \\ 1\end{smallbmatrix}+\mathbb{R}\begin{smallbmatrix}1 \\ -1\end{smallbmatrix}$ before reaching their endpoint, with that intersection marked by a circle}
\label{slopecheckfig}
\end{figure}
\end{proof}

\subsection{Comparing \texorpdfstring{$(\Zhe_2,P)$}{(Zhe2,P)} with the geometry of rolling spheres}\label{comparison} \hfill \\
The importance of $K$ here comes from its transitive action on $\Zhe_2/P$.

\begin{lemma}\label{loclikerolling} The maximal compact subgroup $K<\Zhe_2$ acts transitively on $\Zhe_2/P$ with stabilizer $K\cap P\simeq\Orth(2)$.\end{lemma}
\begin{proof}Consider the quotient map $\mathrm{q}_\mathfrak{p}:\mathfrak{g}\to\mathfrak{g}/\mathfrak{p}$. Since $\mathrm{q}_\mathfrak{p}(\mathfrak{k})=\mathfrak{g}/\mathfrak{p}$, the $K$-orbit through $\quot{P}(e)\in\Zhe_2/P$ is of maximal dimension and therefore must be open. However, $K$ is compact, so all of its orbits are closed. Thus, the $K$-orbit through $\quot{P}(e)$ must be all of $\Zhe_2/P$, since $\Zhe_2/P$ is connected.

Recall from above that $P=G_0P_+$, so that every element of $P$ is of the form $A\exp(Y)$ for some $A\in G_0\simeq\GLin_2\mathbb{R}$ and $Y\in\mathfrak{p}_+$. To find the elements of $P$ that are contained in $K\cap P$, we therefore just compute that
\begin{align*}\hat{\theta}A\exp(Y)\hat{\theta}^{-1} & =\hat{\theta}A\hat{\theta}^{-1}\hat{\theta}\exp(Y)\hat{\theta}^{-1} =\hat{\theta}A\hat{\theta}^{-1}\exp(\theta(Y)) \\ & =(A^{-1})^\top\exp(\theta(Y)),\end{align*}
so $A\exp(Y)\in K\cap P=\{p\in P:\hat{\theta}p\hat{\theta}^{-1}=p\}$ if and only if $Y=0$ and $A\in\Orth(2)<\GLin_2\mathbb{R}\simeq G_0$.\mbox{\qedhere}\end{proof}


Because the maximal compact subgroup $K$ acts transitively on the connected manifold $\Zhe_2/P\cong K/K\cap P$, we get a new model geometry $(K,K\cap P)$ with the same underlying homogeneous space as $(\Zhe_2,P)$.

To help us interpret this model $(K,K\cap P)$, let us first think about
\begin{align*}\POrth(3):=\Orth(3)/\mathrm{Z}(\Orth(3)) & =\Orth(3)/\langle-\mathds{1}\rangle \\ & \simeq\SOrth(3)\simeq\Spin(3)/\mathrm{Z}(\Spin(3)),\end{align*}
which has a natural transitive action on $\mathbb{RP}^2$ induced by the action of $\Orth(3)$ on $\mathbb{R}^3$. Under this induced action, the stabilizer subgroup for the point of $\mathbb{RP}^2$ corresponding to the 1-dimensional subspace spanned by $\begin{smallbmatrix}1 & 0 & 0\end{smallbmatrix}^\top$ is the image of
\[\mathrm{Stab}_{\Orth(3)}\!\left(\begin{smallbmatrix}1 \\ 0 \\ 0\end{smallbmatrix}\right)\simeq\Orth(2)\]
under the quotient homomorphism $\quot{\mathrm{Z}(\Orth(3))}\!:\Orth(3)\to\POrth(3)$. Notably, this stabilizer subgroup of $\begin{smallbmatrix}1 & 0 & 0\end{smallbmatrix}^\top$ in $\Orth(3)$ is precisely the isotropy for 2-dimensional spherical geometry $(\Orth(3),\Orth(2))$, and because $-\mathds{1}$ is not contained in it, $\quot{\mathrm{Z}(\Orth(3))}\!$ restricts to an isomorphism on this isotropy subgroup in $\Orth(3)$. This gives us a model $(\POrth(3),\Orth(2))$ encoding the geometry on $\mathbb{RP}^2$ induced by 2-dimensional spherical geometry via the double cover $\mathbb{S}^2\cong\Orth(3)/\Orth(2)\to\POrth(3)/\Orth(2)\cong\mathbb{RP}^2$.

In a similar way, we can get a model for rolling projective planes via the quotient homomorphism
\[\quot{\mathrm{Z}(\mathrm{S}(\Orth(3)\times\Orth(3)))}\!:\mathrm{S}(\Orth(3)\times\Orth(3))\to\POrth(3)\times\POrth(3).\]

\begin{definition}The model for \emph{rolling projective planes} is given by $(\POrth(3)\times\POrth(3),\Orth(2))$, where $\Orth(2)$ is embedded in $\POrth(3)\times\POrth(3)$ as
\[\left\{(\pm A,\pm A)\in\POrth(3)\times\POrth(3):A\in\mathrm{Stab}_{\Orth(3)}\!\left(\begin{smallbmatrix}1 \\ 0 \\ 0\end{smallbmatrix}\right)\right\},\]
which is the precisely the image of the isotropy for the model geometry of rolling spheres under the quotient homomorphism $\quot{\mathrm{Z}(\mathrm{S}(\Orth(3)\times\Orth(3)))}$.\end{definition}

Via the covering homomorphism $\quot{\mathrm{Z}(\mathrm{S}(\Orth(3)\times\Orth(3)))}$ and the covering map
\[\mathrm{S}(\Orth(3)\times\Orth(3))/\Orth(2)\to(\POrth(3)\times\POrth(3))/\Orth(2)\]
induced by it, we should interpret the geometry of rolling spheres as a double cover of the geometry of rolling projective planes. Like with rolling spheres, we can think of $(\POrth(3)\times\POrth(3))/\Orth(2)$ as the space of relative configurations for a pair of projective planes that are kept tangent to each other, and pictorially, we can imagine the elements $(\pm A_0,\pm A_1)\in\POrth(3)\times\POrth(3)$ as acting by applying $\pm A_0$ to one copy of the projective plane and $\pm A_1$ to the other.

In order to relate this back to $(K,K\cap P)$, let us note that the quotient homomorphism
\[\quot{\mathrm{Z}(K)}\!:K\to K/\mathrm{Z}(K)\simeq\POrth(3)\times\POrth(3)\]
of $K$ by $\mathrm{Z}(K)=\langle\hat{\theta}\rangle$ is \textit{also} a double cover, and since $\hat{\theta}\not\in K\cap P\simeq\Orth(2)$, the restriction $\quot{\mathrm{Z}(K)}|_{K\cap P}$ is \textit{also} an isomorphism onto its image. In fact, choosing our isomorphism $K/\mathrm{Z}(K)\simeq\POrth(3)\times\POrth(3)$ so that
\[\quot{\mathrm{Z}(K)}\circ\rho:\Spin(3)\times\Spin(3)\to\POrth(3)\times\POrth(3)\]
is given by taking the quotient $\Spin(3)\to\SOrth(3)\simeq\POrth(3)$ in each of the two factors individually, we can see from Lemma \ref{setrotoiso} that
\[\quot{\mathrm{Z}(K)}|_{K\cap P}(A)=\left(\pm\begin{smallbmatrix}1 & 0 \\ 0 & A\end{smallbmatrix},\pm\begin{smallbmatrix}1 & 0 \\ 0 & A\end{smallbmatrix}\right),\]
so that the image of $K\cap P$ under $\quot{\mathrm{Z}(K)}$ is precisely the isotropy for the model geometry of rolling projective planes. We should, therefore, interpret $(K,K\cap P)$ as a double cover of the model of rolling projective planes in the same way that the model $(\mathrm{S}(\Orth(3)\times\Orth(3)),\Orth(2))$ for rolling spheres is.

These are not the same double cover, however. Indeed, as discussed in Section 7 of \cite{BorMontgomery2009} and Section 5 of \cite{BaezHuerta2014}, they are not even the same double cover at the level of the base homogeneous spaces up to diffeomorphism: $K/K\cap P\cong\mathbb{RP}^2\times\Spin(3)$, while $\mathrm{S}(\Orth(3)\times\Orth(3))/\Orth(2)\cong\mathbb{S}^2\times\SOrth(3)$. We can distinguish between these double covers intuitively by looking at the kernels of their respective covering homomorphisms. In the case of rolling spheres, the kernel is $\mathrm{Z}(\mathrm{S}(\Orth(3)\times\Orth(3)))=\langle(-\mathds{1},-\mathds{1})\rangle$; the action of the element $(-\mathds{1},-\mathds{1})$ amounts to applying the antipodal map to both spheres---flipping both of their orientations---at the same time, so quotienting out this center is essentially removing the orientability of the two spheres and leaving us with a pair of projective planes. For the model geometry $(K,K\cap P)$, on the other hand, the kernel of the covering homomorphism is $\mathrm{Z}(K)=\langle\hat{\theta}\rangle=\rho(\mathrm{Z}(\Spin(3)\times\Spin(3)))$, where
\[\hat{\theta}=\rho\!\left(\exp\!\left(2\pi\!\begin{smallbmatrix}0 & 0 & 0 \\ 0 & 0 & -1 \\ 0 & 1 & 0\end{smallbmatrix}\right)\!,\,e\right)=\rho\!\left(e,\exp\!\left(2\pi\!\begin{smallbmatrix}0 & 0 & 0 \\ 0 & 0 & -1 \\ 0 & 1 & 0\end{smallbmatrix}\right)\right);\]
we can think of $\hat{\theta}$ as flipping the ``spinorialness'' in one of the $\Spin(3)$ factors, but since
\[\ker(\rho)=\left\langle\left(\exp\!\left(2\pi\!\begin{smallbmatrix}0 & 0 & 0 \\ 0 & 0 & -1 \\ 0 & 1 & 0\end{smallbmatrix}\right)\!,\,\exp\!\left(2\pi\!\begin{smallbmatrix}0 & 0 & 0 \\ 0 & 0 & -1 \\ 0 & 1 & 0\end{smallbmatrix}\right)\right)\right\rangle,\]
this is the same as flipping the ``spinorialness'' of the other factor, so the ``spinorialness'' is really being shared between the two. In other words, when we quotient out the center of $K$, we are removing a kind of ``joint spinorialness'' between the two $\Spin(3)$ factors to get a geometry over a pair of projective planes, so we might think of ${(K,K\cap P)}$ as the geometry of a pair of \emph{``jointly spinorial'' projective planes} rolling along each other; see Chapter 3 of \cite{EricksonThesis} for more details on this perspective. We suspect that this can be interpreted more formally as rolling a pair of supermanifolds (whose reduced manifolds are projective planes) along each other, but the author of this paper is not yet familiar enough with how analogues of Riemannian geometry for supermanifolds work to say this with any authority.

\begin{remark}The reader will note that our description of $(K,K\cap P)$ is distinct from the interpretation provided in \cite{BaezHuerta2014}, which identified it as the geometry of a ``spinorial ball rolling along a projective plane'', and we feel obligated to remark that this interpretation from \cite{BaezHuerta2014} is \textit{incorrect}. Indeed, if one of the two ``rolling surfaces'' for the model were spinorial but not the other, then the symmetry group for the model would need to be either $\Spin(3)\times\SOrth(3)$ or $\SOrth(3)\times\Spin(3)$, neither of which is isomorphic to $K\simeq(\Spin(3)\times\Spin(3))/\ker(\rho)$. We think that Baez and Huerta were led astray by the erroneous assumption that each of the two factors of $\mathbb{RP}^2\times\Spin(3)\cong K/K\cap P$ can be uniquely associated with one of the two ``rolling surfaces'' of the geometry. In truth, these factors have a more complicated relationship with the ``surfaces'' that is not easily visible from the extrinsic viewpoint for rolling used in \cite{BaezHuerta2014}. For example, the $\mathbb{RP}^2$ factor comes from an immersed copy of 2-dimensional projective geometry induced by the identification of $\mathfrak{g}_{-3}+\mathfrak{g}_0+\mathfrak{g}_3$ with $\mathfrak{sl}_3\mathbb{R}$, and can be seen as corresponding to both copies of $\mathbb{RP}^2$ for the geometry of rolling projective planes; for more on this, we again refer to Chapter 3 of \cite{EricksonThesis}.\end{remark}

Regardless of the persnickety particulars of the global interpretation, the considerations above tell us that the geometry of $(K,K\cap P)$ is \textit{locally} isomorphic to that of both rolling spheres and rolling projective planes, with $\mathfrak{k}\approx\mathfrak{so}(3)\oplus\mathfrak{so}(3)$ and $K\cap P\simeq\Orth(2)$ in such a way that the adjoint action of the isotropy is equivalent. This allows us to reinterpret motion and local structure in $(K,K\cap P)$ in terms of motion and local structure in $(\mathrm{S}(\Orth(3)\times\Orth(3)),\Orth(2))$.

Most notably, for $r_0,r_1>0$, we get a distribution
\[\tilde{D}_{r_0/r_1}:=\mathrm{q}_{{}_{K\cap P}*}\MC{K}^{-1}\!\left(\rho_*\!\left(\left\{\left(\tfrac{1}{r_0}\!\begin{smallbmatrix}0 & -v^\top \\ v & 0\end{smallbmatrix}\!,\tfrac{1}{r_1}\begin{smallbmatrix}0 & -v^\top \\ v & 0\end{smallbmatrix}\right)\!:v\in\mathbb{R}^2\right\}+\mathfrak{o}(2)\right)\right)\]
on $K/K\cap P$ locally isomorphic to the rolling distribution
\[D_{r_0/r_1}:=\mathrm{q}_{{}_{\Orth(2)}*}\MC{\mathrm{S}(\Orth(3)\times\Orth(3))}^{-1}\!\left(\left\{\left(\tfrac{1}{r_0}\!\begin{smallbmatrix}0 & -v^\top \\ v & 0\end{smallbmatrix}\!,\tfrac{1}{r_1}\!\begin{smallbmatrix}0 & -v^\top \\ v & 0\end{smallbmatrix}\right)\!:v\in\mathbb{R}^2\right\}+\mathfrak{o}(2)\right)\]
for rolling spheres. Just like with rolling spheres, no choice of $\tilde{D}_{r_0/r_1}$ is inherently more canonical than any other from the perspective of $(K,K\cap P)$, as the geometry does not single out any specific choice for us. However, $\Zhe_2/P\cong K/K\cap P$, and $(\Zhe_2,P)$ \textit{does} have a canonical choice of $(2,3,5)$-distribution, coming from the filtration on $\mathfrak{g}$:
\[D:=\mathrm{q}_{{}_P*}\MC{\Zhe_2}^{-1}(\mathfrak{g}^{-1})=\mathrm{q}_{{}_{K\cap P}*}\MC{K}^{-1}(\mathfrak{k}\cap\mathfrak{g}^{-1}).\]
This is where the additional local infinitesimal symmetries come from.

\begin{lemma}For radii $r_0\neq r_1$, the Lie algebra of local infinitesimal symmetries for the rolling distribution $D_{r_0/r_1}$ is isomorphic to $\mathfrak{g}$ if and only if either $\tilde{D}_{r_0/r_1}\!=D$ or $\tilde{D}_{r_1/r_0}\!=D$.\end{lemma}
\begin{proof}Swapping the $\Spin(3)$ factors of
\[K=\rho(\Spin(3)\times\Spin(3))\simeq(\Spin(3)\times\Spin(3))/\ker(\rho)\]
gives a Lie group automorphism of $K$ that preserves the isotropy ${K\cap P}$. This induces a $K$-equivariant diffeomorphism of $K/K\cap P$ of the form $\quot{K\cap P}(\rho(A_0,A_1))\mapsto\quot{K\cap P}(\rho(A_1,A_0))$, which provides an isomorphism between $\tilde{D}_{r_0/r_1}$ and $\tilde{D}_{r_1/r_0}$, analogous to the isomorphism between $D_{r_0/r_1}$ and $D_{r_1/r_0}$ coming from swapping the roles of the two rolling spheres. 
The $(2,3,5)$-distribution $\tilde{D}_{r_0/r_1}$ will, therefore, have Lie algebra of local infinitesimal symmetries isomorphic to $\mathfrak{g}$ at some point (hence at every point, by the transitive action of $K$) if and only if $\tilde{D}_{r_1/r_0}$ does too.

When $\tilde{D}_{r_0/r_1}\!=D$, the action of $K$ on $K/K\cap P$ by symmetries of the distribution $\tilde{D}_{r_0/r_1}$ extends to the action of $\Zhe_2$ on $\Zhe_2/P\cong K/K\cap P$ by symmetries of $D$. The local infinitesimal symmetries at a point of $K/K\cap P$ then give local infinitesimal symmetries for $D_{r_0/r_1}$ because of the local isomorphism between the geometries of $(K,K\cap P)$ and $(\mathrm{S}(\Orth(3)\times\Orth(3)),\Orth(2))$, so if $\tilde{D}_{r_0/r_1}\!=D$, then $D_{r_0/r_1}\cong D_{r_1/r_0}$ has local infinitesimal symmetry algebra isomorphic to $\mathfrak{g}$.

Conversely, if $D_{r_0/r_1}$ has Lie algebra of local infinitesimal symmetries isomorphic to $\mathfrak{g}$, then the local isomorphism between the geometries tells us that the Lie algebra of local infinitesimal symmetries for $\tilde{D}_{r_0/r_1}$ is also $\mathfrak{g}$. From the natural filtration of $\mathfrak{g}$ at the end of Section \ref{tanakalight}, we can see that $\tilde{D}_{r_0/r_1}$ must coincide with the image of $\mathfrak{g}^{-1}$ at each point, so since $K$ is acting transitively on $K/K\cap P$ by symmetries of $\tilde{D}_{r_0/r_1}$, we get a Lie algebra embedding $i:\mathfrak{k}\to\mathfrak{g}$ with
\[\tilde{D}_{r_0/r_1}=\mathrm{q}_{{}_{K\cap P}*}\MC{K}^{-1}(i^{-1}(i(\mathfrak{k})\cap\mathfrak{g}^{-1})).\]
There are exactly two choices of embedding for $\mathfrak{k}\approx\mathfrak{so}(3)\oplus\mathfrak{so}(3)$ into $\mathfrak{g}$ up to conjugacy, namely the inclusion map and the composition of the inclusion map with the swap between the two summands; in other words, if $\tilde{D}_{r_0/r_1}$ is not equal to $D=\mathrm{q}_{{}_{K\cap P}*}\MC{K}^{-1}(\mathfrak{k}\cap\mathfrak{g}^{-1})$, then $\tilde{D}_{r_1/r_0}$ must be.\mbox{\qedhere}\end{proof}

\begin{corollary}The rolling distribution $D_{r_0/r_1}$ admits a Lie algebra of local infinitesimal symmetries isomorphic to $\mathfrak{g}$ if and only if $\tfrac{r_0}{r_1}\in\{3,\tfrac{1}{3}\}$.\end{corollary}
\begin{proof}By Lemmas \ref{setrotoiso} and \ref{setrolleye}, we know that there exist some $r_0,r_1>0$ for which $\tilde{D}_{r_0/r_1}=\mathrm{q}_{{}_{K\cap P}*}\MC{K}^{-1}(\mathfrak{k}\cap\mathfrak{g}^{-1})=D$, and by Corollary \ref{ratioappears}, we know that when we put $\rho$ into the form required for this, we must have
\[\rho_*^{-1}(v+\theta(v))\in\left\{\left(\begin{smallbmatrix}0 & -v^\top \\ v & 0\end{smallbmatrix}\!,3\begin{smallbmatrix}0 & -v^\top \\ v & 0\end{smallbmatrix}\right),\left(3\begin{smallbmatrix}0 & -v^\top \\ v & 0\end{smallbmatrix}\!,\begin{smallbmatrix}0 & -v^\top \\ v & 0\end{smallbmatrix}\right)\right\}.\]
Thus, $\tilde{D}_{r_0/r_1}=D$ or $\tilde{D}_{r_1/r_0}=D$ if and only if the ratio of the radii is either $\tfrac{1}{1/3}=3$ or $\tfrac{1/3}{1}=\tfrac{1}{3}$.\mbox{\qedhere}\end{proof}

As we mentioned at the beginning of the paper, we can interpret this result from the perspective of the rolling spheres themselves, essentially by translating Corollary \ref{ratioappears} into the language of that geometry via the local isomorphism between $(K,K\cap P)$ and $(\mathrm{S}(\Orth(3)\times\Orth(3)),\Orth(2))$. To do this, let us consider the one-parameter subgroups of $\mathrm{S}(\Orth(3)\times\Orth(3))$ generated by $(\tfrac{1}{r_0}[\begin{smallmatrix}0 & e_1^\top \\ e_1 & 0\end{smallmatrix}],\tfrac{1}{r_1}[\begin{smallmatrix}0 & e_1^\top \\ e_1 & 0\end{smallmatrix}])$ and $([\begin{smallmatrix}0 & e_1^\top \\ e_1 & 0\end{smallmatrix}],-[\begin{smallmatrix}0 & e_1^\top \\ e_1 & 0\end{smallmatrix}])$. As mentioned in Section \ref{rollattach}, $(\tfrac{1}{r_0}[\begin{smallmatrix}0 & e_1^\top \\ e_1 & 0\end{smallmatrix}],\tfrac{1}{r_1}[\begin{smallmatrix}0 & e_1^\top \\ e_1 & 0\end{smallmatrix}])$ generates a one-parameter subgroup corresponding to rolling a sphere of radius $r_0$ and a sphere of radius $r_1$ together along a pair of great circles without letting them slip or twist. The one-parameter subgroup generated by $([\begin{smallmatrix}0 & e_1^\top \\ e_1 & 0\end{smallmatrix}],-[\begin{smallmatrix}0 & e_1^\top \\ e_1 & 0\end{smallmatrix}])$, on the other hand, moves the two spheres along the same pair of great circles, rotating each sphere by the same amount but in the opposite direction as the other. These correspond to the left and right curves depicted in Figure \ref{benotafraid}, respectively. Like we saw in Corollary \ref{ratioappears}, these commuting one-parameter subgroups generate a maximal torus in $\mathrm{S}(\Orth(3)\times\Orth(3))$ of the form $\exp(\mathbb{R}[\begin{smallmatrix}0 & e_1^\top \\ e_1 & 0\end{smallmatrix}])\times\exp(\mathbb{R}[\begin{smallmatrix}0 & e_1^\top \\ e_1 & 0\end{smallmatrix}])$, which is isomorphic to $\mathbb{R}^2/2\pi\mathbb{Z}^2$ by
\[\left(\exp(x[\begin{smallmatrix}0 & e_1^\top \\ e_1 & 0\end{smallmatrix}]),\exp(y[\begin{smallmatrix}0 & e_1^\top \\ e_1 & 0\end{smallmatrix}])\right)\mapsto\begin{smallbmatrix}x \\ y\end{smallbmatrix}+2\pi\mathbb{Z}^2.\]
Conveniently, we can think of the maximal torus for $K$ as a double cover of this maximal torus for $\mathrm{S}(\Orth(3)\times\Orth(3))$, via their identifications with $\mathbb{R}^2/(2\pi\mathbb{Z}[\begin{smallmatrix}1 \\ 1\end{smallmatrix}]+2\pi\mathbb{Z}[\begin{smallmatrix}1 \\ -1\end{smallmatrix}])$ and $\mathbb{R}^2/2\pi\mathbb{Z}^2$, respectively. Because it is a double cover, the corresponding one-parameter subgroups must intersect twice as many times in this torus as they did in $K$, so since Lemma \ref{2intersections} necessitates exactly 2 intersections in $K$, the one-parameter subgroups in $\mathrm{S}(\Orth(3)\times\Orth(3))$ must intersect exactly 4 times, which happens if and only if the ratio of the radii is 1:3 or 3:1.

\subsection{The other ratios of radii have no additional symmetries}\label{otherratios}\hfill \\
For $r_0=r_1$, the rolling distribution $D_{r_0/r_1}$ is not a $(2,3,5)$-distribution. Indeed, it is integrable in that case, so its local infinitesimal symmetry algebra contains the infinite-dimensional space of vector fields tangent to the distribution. Beyond this little punctilio, however, we will prove in this subsection that only the 1:3 ratio of radii yields additional local symmetry, which is the only thing left to prove for the main theorem.

\begin{lemma}When $\tfrac{r_0}{r_1}\!\not\in\!\{3,1,\tfrac{1}{3}\}$, the Lie algebra of local infinitesimal symmetries for $D_{r_0/r_1}$ is precisely $\mathfrak{so}(3)\oplus\mathfrak{so}(3)$.\end{lemma}
\begin{proof}In this case, since $\tfrac{r_0}{r_1}\neq 1$, $D_{r_0/r_1}$ must be a $(2,3,5)$-distribution, so by Proposition \ref{gradedimageprop}, the graded image $\mathrm{gr}(\mathfrak{s})$ of the local infinitesimal symmetry algebra $\mathfrak{s}$ at a point of $\mathrm{S}(\Orth(3)\times\Orth(3))/\Orth(2)$ must embed as a graded subalgebra of $\mathfrak{g}$. Moreover, as we noted at the end of Section \ref{tanakalight},\linebreak if $\mathfrak{s}=\mathfrak{so}(3)\oplus\mathfrak{so}(3)$, coming from the transitive action of $\mathrm{S}(\Orth(3)\times\Orth(3))$, then the graded image will be $\mathfrak{g}_-+\mathfrak{o}(2)$. If $\mathfrak{s}$ were to properly contain this copy of $\mathfrak{so}(3)\oplus\mathfrak{so}(3)$, then $\mathrm{gr}(\mathfrak{s})$ would need to contain $\mathfrak{g}_-+\mathfrak{o}(2)$ as a proper subalgebra as well.

Let us suppose, by way of contradiction, that $\eta\in\mathfrak{s}$ is not contained in $\mathfrak{so}(3)\oplus\mathfrak{so}(3)$. Because $\mathrm{S}(\Orth(3)\times\Orth(3))$ is acting transitively, we may assume that $\eta$ vanishes at the point under consideration, so that $\eta\in\mathfrak{s}^0$. If we further had $\eta\in\mathfrak{s}^1$, then the graded image $\mathrm{gr}(\mathfrak{s})$ would contain a nonzero element in a positive grading component, but since the graded image must also contain $\mathfrak{o}(2)<\mathfrak{g}_0\approx\mathfrak{gl}_2\mathbb{R}$, it would necessarily contain the rest of that grading component as well; taken together with $\mathfrak{g}_-$, this would generate all of $\mathfrak{g}$, but we know $\mathrm{gr}(\mathfrak{s})\neq\mathfrak{g}$ because $\tfrac{r_0}{r_1}\not\in\{3,\tfrac{1}{3}\}$, so this is impossible. Thus, we may assume that $\mathfrak{s}^1=\{0\}$, which tells us that $\mathfrak{s}^0\approx\mathrm{gr}_0(\mathfrak{s})<\mathfrak{g}_0$. Knowing that $\eta\in\mathfrak{s}^0$ is not contained in the copy of $\mathfrak{o}(2)$ coming from the isotropy of $\mathfrak{so}(3)\oplus\mathfrak{so}(3)$, this leaves us three options for what $\mathrm{gr}_0(\mathfrak{s})$ can be, corresponding to the three subalgebras properly containing $\mathfrak{o}(2)$ in $\mathfrak{gl}_2\mathbb{R}$: it must be either the sum $\mathbb{R}\mathds{1}\oplus\mathfrak{o}(2)$, the semisimple part $\mathfrak{g}_0^\text{ss}\approx\mathfrak{sl}_2\mathbb{R}$ of $\mathfrak{g}_0$, or $\mathfrak{g}_0$ itself.

In all three of these cases, $\mathfrak{s}$ must be isomorphic to its graded image. For $\mathrm{gr}_0(\mathfrak{s})$ equal to either $\mathbb{R}\mathds{1}\oplus\mathfrak{o}(2)$ or $\mathfrak{g}_0$, this comes from the adjoint\linebreak action of $\mathds{1}\in\mathrm{gr}_0(\mathfrak{s})\approx\mathfrak{s}^0$ being multiplication by the integer $i$ on each grading component $\mathfrak{g}_{-i}$. When $\mathrm{gr}_0(\mathfrak{s})=\mathfrak{g}_0^\text{ss}$, the argument is slightly more involved: as a representation of $\mathfrak{g}_0^\text{ss}\approx\mathfrak{sl}_2\mathbb{R}$, $\mathfrak{s}$ breaks into weight spaces of $\langle\begin{smallbmatrix}1 & 0 \\ 0 & -1\end{smallbmatrix}\rangle=\mathfrak{a}\cap\mathfrak{g}_0^\text{ss}$, with the bracket of two weight spaces $\mathfrak{s}_{\nu_1}$ and $\mathfrak{s}_{\nu_2}$ contained in $\mathfrak{s}_{\nu_1+\nu_2}$ in the usual way. The automorphism of $\mathfrak{s}$ induced by the reflection $\begin{smallbmatrix}1 & 0 \\ 0 & -1\end{smallbmatrix}\in\Orth(2)$ as an element of $\mathrm{S}(\Orth(3)\times\Orth(3))$ naturally commutes with the adjoint action of the elements of $\mathfrak{a}\cap\mathfrak{g}_0^\text{ss}$, so we can further decompose the weight spaces according to their eigenvalue for this automorphism. This singles out a distinguished copy of each root space for $\mathfrak{a}$ in $\mathfrak{g}_-$ inside of $\mathfrak{s}$, with the behavior of the roots under the bracket forcing the $\mathfrak{g}_0^\text{ss}$-subrepresentation complementary to $\mathfrak{g}_0^\text{ss}$ in $\mathfrak{s}$ to be isomorphic to $\mathfrak{g}_-$.

However, we are assuming that $\mathfrak{s}$ contains $\mathfrak{so}(3)\oplus\mathfrak{so}(3)$, so since the graded image of $\mathfrak{s}$ must be contained in $\mathfrak{g}_-+\mathfrak{g}_0$ and $\mathfrak{so}(3)\oplus\mathfrak{so}(3)$ cannot embed as a subalgebra of $\mathfrak{g}_-+\mathfrak{g}_0$, $\mathfrak{s}$ cannot be isomorphic to its graded image. This is a contradiction, so we must have $\mathfrak{s}=\mathfrak{so}(3)\oplus\mathfrak{so}(3)$.\mbox{\qedhere}\end{proof}

\bibliographystyle{plain}
\bibliography{rolling-refs}

@book{CapSlovakPG1,
    author={{\v{C}}ap, Andreas and Slov\'{a}k, Jan},
    title={{P}arabolic {G}eometries {I}: {B}ackground and {G}eneral {T}heory},
    series={Mathematical Surveys and Monographs},
    volume={154},
    publisher={Amer. Math. Soc.},
    address={Providence, RI},
    year={2009}
}

@book{Sharpe1997,
    author={Sharpe, R. W.},
    title={{D}ifferential {G}eometry: {C}artan's {G}eneralization of {K}lein's {E}rlangen {P}rogram},
    series={Graduate Texts in Mathematics},
    volume={166},
    publisher={Springer-Verlag},
    address={New York, NY},
    year={1997}
}

@phdthesis{EricksonThesis,
    author={Erickson, Jacob W.},
    title={Closed surface pairs with maximal local rolling symmetries},
    school={University of Maryland, College Park},
    year={2025}
}

@article{Cartan1910,
    author={Cartan, \'{E}lie},
    title={Les syst\`{e}mes de {P}faff \`{a} cinq variables et les \'{e}quations aux d\'{e}riv\'{e}es partielles du second ordre},
    journal={Ann. Sci. \'{E}c. Norm. Sup\'{e}r.},
    volume={27},
    pages={109-192},
    year={1910}
}

@article{Agrachev2007,
    author={Agrachev, A. A.},
    title={Rolling balls and octonions},
    journal={Proc. Steklov Inst. Math.},
    volume={258}, pages={12--22}, year={2007}
}

@article{AnNurowski2014,
    author={An, Daniel and Nurowski, Pawe\l{}},
    title={Twistor space for rolling bodies},
    journal={Communications in Mathematical Physics},
    volume={326},
    pages={393-414},
    year={2014}
}

@article{BorMontgomery2009,
    author={Bor, Gil and Mongomery, Richard},
    title={$\mathrm{G}_2$ and the ``rolling distribution''},
    journal={Enseign. Math.},
    volume={55},
    pages={157-196},
    year={2009}
}

@article{BaezHuerta2014,
    author={Baez, John C. and Huerta, John},
    title={$\mathrm{G}_2$ and the rolling ball},
    journal={Trans. Amer. Math. Soc.},
    volume={366},
    pages={5257-5293},
    year={2014}
}

@article{Zelenko2006,
    author={Zelenko, Igor},
    title={On variational approach to differential invariants of rank two distributions},
    journal={Diff. Geom. Appl.},
    volume={24},
    number={3},
    pages={235-259},
    year={2006}
}

@article{Tanaka1970,
    author={Tanaka, Noboru},
    title={On differential systems, graded {L}ie algebras and pseudo-groups},
    journal={J. Math. Kyoto Univ.},
    volume={10},
    pages={1-82},
    year={1970}
}

@article{LeistnerNurowski2012,
    author={Leistner, Thomas and Nurowski, Pawe{\l}},
    title={Ambient metrics with exceptional holonomy},
    journal={Ann. Sc. Norm. Super. Pisa Cl. Sci. (5)},
    volume={XI}, pages={407--436}, year={2012}
}

@article{Evans2025,
    author={Evans, Parker},
    title={Geometric structures for the $\mathrm{G}_2'$-{H}itchin component},
    journal={Adv. Math.},
    volume={462},
    pages={110091},
    year={2025}
}

@misc{The2022,
    author={The, Dennis},
    title={A {C}artan-theoretic classification of multiply-transitive (2,3,5)-distributions},
    howpublished={arXiv:2205.03387},
    year={2022}
}

\end{document}